\documentclass[11pt]{amsart}
\usepackage{hyperref,pb-diagram,amssymb,epic,eepicemu, verbatim,graphicx,graphics,epsfig,psfrag,braket}
\hypersetup{colorlinks} \usepackage{color}
\definecolor{darkred}{rgb}{0.5,0,0}
\definecolor{darkgreen}{rgb}{0,0.5,0}
\definecolor{darkblue}{rgb}{0,0,0.5} \hypersetup{ colorlinks,
  linkcolor=darkblue, filecolor=darkgreen, urlcolor=darkred,
  citecolor=darkblue }

\newtheorem{theorem}{Theorem}[section]
\newtheorem{corollary}[theorem]{Corollary}

\newtheorem{proposition}[theorem]{Proposition}
\newtheorem{lemma}[theorem]{Lemma}
\newtheorem{lem}[theorem]{}
\theoremstyle{definition}
\newtheorem{definition}[theorem]{Definition}

\theoremstyle{remark}
\newtheorem{remark}[theorem]{Remark}
\newtheorem{example}[theorem]{Example}
\newcommand{\blem}{\begin{lem} \rm}
\newcommand{\elem}{\end{lem}}
\newcommand\B{\mathcal{B}}

\newcommand\M{\mathcal{M}}

\renewcommand\M{\mathcal{M}}

\renewcommand\sharp{\setlength{\unitlength}{0.00013333in}
\begin{picture}(688,703)(0,-10)
\path(244,644)(244,44)
\path(444,644)(444,44)
\path(44,444)(644,444)
\path(644,244)(44,244)
\end{picture}
}

\newcommand{\J}{\mathcal{J}}

\newcommand{\R}{\mathbb{R}}

\newcommand{\C}{\mathbb{C}}
\newcommand{\cC}{\mathcal{C}}
\newcommand{\cT}{\mathcal{T}}

\newcommand{\Z}{\mathbb{Z}}

\newcommand{\ddt}{\frac{d}{dt}}

\renewcommand{\P}{\mathbb{P}}

\newcommand\lie[1]{\mathfrak{#1}}

\newcommand{\g}{\lie{g}}

\renewcommand{\t}{\lie{t}}
\newcommand{\Alc}{\lie{A}}

\newcommand{\su}{\lie{su}}

\newcommand{\on}{\operatorname}
\newcommand{\ainfty}{{$A_\infty$\ }}

\newcommand{\dual}{\vee}

\newcommand{\Edge}{\on{Edge}}
\newcommand{\Symp}{\on{Symp}}

\newcommand{\Lag}{\on{Lag}}

\newcommand{\Ver}{\on{Vert}}

\newcommand{\Aut}{ \on{Aut} }

\newcommand{\Ad}{ \on{Ad} } 
\newcommand{\ad}{ \on{ad} }

\newcommand{\Hom}{ \on{Hom}}

\newcommand{\ind}{ \on{ind}}

\newcommand{\diag}{  \on{diag}}

\newcommand{\codim}{\on{codim}}

\newcommand{\ssm}{\kern-.5ex \smallsetminus \kern-.5ex}

\newcommand\dirac{/\kern-1.2ex\partial} 
\newcommand\qu{/\kern-.7ex/} 
\newcommand\lqu{\backslash \kern-.7ex \backslash} 

\newcommand\dr{r_+ \kern-.7ex - \kern-.7ex r_-}
 
\def\pd{\partial}

\newcommand\quo[2]{
                \text{\raise.8ex\hbox{$\scriptstyle#1\!$}/\lower.8ex\hbox{$\!\scriptstyle#2$}}
}

\renewcommand{\d}{{\mbox{d}}}
\newcommand{\ol}{\overline}

\newcommand\Phinv{\Phi^{-1}}

\newcommand\eps{\epsilon}

\newcommand{\f}{\frac}

\newcommand{\qq}{\text{\tiny $1/4$}}

\newcommand{\ti}{\tilde}

\newcommand\pt{\on{pt}}

\newcommand\rk{\on{rk}}

\newcommand\Map{\on{Map}}
\newcommand\rank{\on{rank}}

\newcommand\Vect{\on{Vect}}
\newcommand\ul{\underline}

\newcommand\reg{{\on{reg}}}

\newcommand\bdefn{\begin{definition}}
\newcommand\edefn{\end{definition}}
\newcommand\bea{\begin{eqnarray*}}
\newcommand\eea{\end{eqnarray*}}
\newcommand\bcv{\left[ \begin{array}{r} }
\newcommand\ecv{\end{array} \right] }

\newcommand\bma{\left[ \begin{array} }
\newcommand\ema{\end{array} \right]}
\newcommand\ben{\begin{enumerate}}
\newcommand\een{\end{enumerate}}
\newcommand\beq{\begin{equation}}
\newcommand\eeq{\end{equation}}
\newcommand\bex{\begin{example}}
\newcommand\bsj{\left\{ \begin{array}{rrr} }
\newcommand\esj{\end{array} \right\}}

\newcommand\Id{\on{Id}}
\newcommand\cI{\mathcal{I}}

\newcommand\eex{\end{example}}
\newcommand\Crit{{\on{Crit}}}

\newcommand\sx{*\kern-.5ex_X}

\def\mathunderaccent#1{\let\theaccent#1\mathpalette\putaccentunder}
\def\putaccentunder#1#2{\oalign{$#1#2$\crcr\hidewidth \vbox
to.2ex{\hbox{$#1\theaccent{}$}\vss}\hidewidth}}

\newcommand{\Tan}{\on{Tan}}

\newcommand{\Graph}{\on{Graph}}
\newcommand{\ff}{{\f{1}{5}}}

\newcommand{\Diff}{\on{Diff}}
\newcommand\GFuk{{\on{Fuk}^{\sharp}}}
\newcommand\Fuk{{\on{Fuk}}}

\renewcommand{\bigcirc}{\includegraphics[width=.1in]{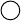}}

\begin{document}

\title[Floer field theory for tangles]{Floer field theory for tangles}

\author{Katrin Wehrheim} 
\address{
Department of Mathematics,
University of California, Berkeley, 
Berkeley, CA 94720.
{\em E-mail address: katrin@math.berkeley.edu}}

\author{Chris Woodward}
\address{Department of Mathematics, 
Rutgers University,
Piscataway, NJ 08854.
{\em E-mail address: ctw@math.rutgers.edu}}

\thanks{Partially supported by NSF grants CAREER 0844188 and DMS
  0904358}

\begin{abstract}  
We use quilted Floer theory to construct functor-valued invariants of tangles, relating 
Fukaya categories of moduli spaces of flat bundles on the incoming and outgoing punctured
surfaces.  As an application, we show the non-triviality of certain elements in
the symplectic mapping class groups of moduli spaces of flat
bundles on punctured spheres.
\end{abstract} 
 
\maketitle

\tableofcontents

\section{Introduction} 

\label{sec:field}

In this paper we apply quilt theory in Lagrangian Floer cohomology,
developed in Wehrheim-Woodward \cite{we:co} and Ma'u-Wehrheim-Woodward
\cite{Ainfty}, to produce functor-valued invariants of tangles via
moduli spaces of flat bundles.  In the
gauge-theoretic interpretation of the Jones polynomial provided by Witten
\cite{wi:jo}, ``quantizing'' moduli spaces of flat bundles gives rise
to knot invariants.  In particular any tangle gives rise to a map
between the spaces of quantum states by ``quantization'' of the
corresponding Lagrangian correspondence.  Several mathematicians and
physicists (in particular Kronheimer-Mrowka \cite{km:kh} and Witten
\cite{wi:kh}) have investigated whether ``categorifying'' the moduli
spaces of flat bundles leads to group-valued knot or tangle invariants.
One naturally expects, according to a suggestion of Fukaya
\cite{fuk:bo}, that Lagrangian correspondences associated to tangles
give rise to functors between Fukaya categories.  The goal here is the
modest one of constructing functor-valued invariants for tangles via
Lagrangian Floer theory.

Our starting point is the observation that a tangle in a cylinder
defines a Lagrangian correspondence between moduli spaces of flat
bundles.  Let $X$ be a compact oriented surface, and let $K$ be a
tangle in the cylinder $Y = [-1,1] \times X$, meeting the ends $\{
\pm 1 \} \times X$ in finite sets of markings $\ul{x}_\pm$, with
components labelled by conjugacy classes of a special unitary group.
The set of flat bundles on the boundary of the complement of $K$ that
extend over the complement defines a {\em formal} Lagrangian
correspondence
\[ L(Y,K) \subset M(X,\ul{x}_-) \times M(X,\ul{x}_+) \]
between the moduli spaces $M(X,\ul{x}_\pm)$ of flat bundles on the
marked, labelled surfaces at the two ends.  In good situations, our previous work \cite{we:co} and work with  Ma'u \cite{Ainfty} associates to such a
correspondence a functor between the generalized Fukaya categories of
the symplectic moduli spaces associated to the two ends:
\[ \Phi(L(Y,K)) : \GFuk(M(X,\ul{x}_-),w_-) \to \GFuk(M(X,\ul{x}_+),w_+) \]
for any integer $w_-$ and a suitable integer $w_+$.  Here $\GFuk(M,w)$ is the category whose objects
are generalized simply-connected monotone Lagrangian submanifolds of
a symplectic manifold $M$ with disk invariant $w$, and morphisms are
Floer cochains.  The disk invariant $w$ is the signed count of Maslov two
index disks passing through a generic point in the Lagrangian, see
Proposition \ref{prop:disk number} below.

One problem with this naive construction is that the moduli spaces of
flat bundles over surfaces are in general not even smooth, let alone
monotone as required for quilt invariants without Novikov coefficients.  We resolve this problem by making
admissibility assumptions on the number of labels and conjugacy
classes.  By Proposition~\ref{prop:smooth} and Theorem~\ref{thm:monthm} this
assumption guarantees smooth monotone symplectic manifolds. A second
problem with the construction is that the Lagrangian correspondence is
in general a singular subset of the product.  To solve this we cut
the cylinder at levels $\{ b_i \} \times X$ and so decompose the
tangle into {\em elementary tangles}
\[(Y,K) = (Y_1,K_1) \cup \ldots \cup (Y_m,K_m), \quad Y_i = [b_{i-1},
  b_i] \times X ,\]
 as in Figure \ref{fig:dtangle}.

\begin{figure}[h]
\includegraphics[height=2in]{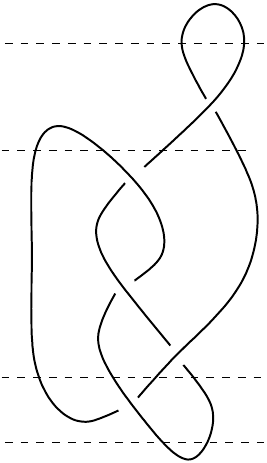}
\caption{Decomposition of a tangle (in this case, a knot) into
  elementary pieces}
\label{fig:dtangle}  \end{figure}

\noindent On each elementary piece $(Y_i,K_i)$, the projection to
$[b_{i-1},b_i]$ restricts to a function on $K_i$ with at most one
critical point.  For such pieces the associated Lagrangian
correspondences $L(Y_i,K_i)$ are smooth and even monotone.  A
decomposition into elementary tangles is obtained by isotoping $K$ so
that the projection $[-1,1] \times X \to [-1,1]$ restricts to a Morse
function on $K$ whose critical points have different values, and
cutting at levels between the critical values.  This yields a
sequence of Lagrangian correspondences giving rise to our
functor-valued invariant
\[ \Phi(Y,K) = \Phi(L(Y_m,K_m)) \circ \ldots \circ \Phi(L(Y_1,K_1))
.\]

A precise version of our main result is stated in the language of
category-valued field theories.  Given a compact oriented surface $X$
and coprime integers $r,d$ let $\Tan(X,r,d)$ denote the {\em tangle
  category} whose objects are finite oriented subsets $\ul{x}$ of $X$
with admissible labels $\ul{\mu}$ as in Proposition \ref{prop:smooth},
and whose morphisms are isotopy classes of tangles $K$ in $[-1,1]
\times X$.  

\begin{theorem} \label{thm:main}  {\rm (Floer field theory for tangles)} 
Let $X$ be a compact oriented surface as above and $w$ an integer.
There exists a functor from $\Tan(X,r,d)$ to the category of (small
\ainfty categories, homotopy classes of \ainfty functors) that assigns
to any finite subset $\ul{x} \subset X$ with labels $\ul{\mu}$ the
generalized Fukaya category $\GFuk(M(X,\ul{x}),w + \delta(\ul{x}))$,
where $\delta$ is a function on objects determined by the disk
invariants of the correspondences associated to cups, see Proposition
\ref{prop:delta}.
\end{theorem} 

\begin{remark}
 We make some comments on related work, some of it appearing after the first version of this preprint.  Gauge-theoretic invariants of knots were constructed using
instantons by Collin-Steer \cite{co:in} and Kronheimer-Mrowka
\cite{km:kh}.  See also Jacobsson-Rubinsztein \cite{jac:st} for the
similarities with Khovanov homology.  Kronheimer-Mrowka have since
defined instanton homologies for webs and foams, for $SO(3)$
\cite{km:tait} and for $SU(3)$ \cite{km:su3}, and instanton invariants
of links using higher rank bundles with prescribed meridional holonomy
were studied by Daemi-Xie \cite{dx:sutured} and Daemi-Iida-Scaduto
\cite{dis:rank3}.  One expects the functors defined in this paper to be
related to the instanton knot invariants by a version of the
Atiyah-Floer conjecture; a version of that conjecture for admissible
$SO(3)$-bundles over three-manifolds was proved by
Daemi-Fukaya-Lipyanskiy \cite{dfl:af}.

Several other constructions of knot and tangle invariants via
Lagrangian Floer theory have appeared.  Seidel-Smith \cite{ss:li} and
Manolescu \cite{ma:li} defined link invariants from Floer cohomology in
nilpotent slices and related spaces, and Abouzaid-Smith
\cite{as:khovanov} showed that the Seidel-Smith invariant recovers
Khovanov homology with coefficients in a field of characteristic zero.
Related constructions for higher rank appear in the work of
Cautis-Kamnitzer \cite{ck:kh} on derived categories of coherent
sheaves and of Mak-Smith \cite{ms:hilb} on Fukaya-Seidel categories of
Hilbert schemes.

Several other papers utilize character
varieties of tangles.  Manolescu and the second author
\cite{mw:extended} define a symplectic instanton homology of
three-manifolds using extended moduli spaces, and Horton
\cite{horton:si} and White \cite{sik:knot} defined symplectic instanton
homologies of three-manifolds resp.\ knots from traceless $SU(2)$
character varieties via Heegaard splittings.  Hedden, Herald and Kirk
\cite{hhk:pillow1,hhk:pillow2} and Hedden, Herald, Hogancamp and Kirk
\cite{hhhk:fukaya} associate to four-ended tangles immersed curves in
the pillowcase, the traceless character variety of the four-punctured
sphere, and compare the resulting Floer theory with Khovanov
cohomology.  Cazassus, Herald and Kirk \cite{chk:tangles} show that
holonomy perturbed traceless character varieties of tangles define
Lagrangian immersions into the character varieties of the boundary,
and some of the correspondences are computed in \cite{chkk:earring,hk:pillowcase} and
\cite{ks:tanglesums}.  Cazassus \cite{caz:twocat} constructs a
$(1+1+1)$ field theory with values in a two-category of Hamiltonian
manifolds. \end{remark}

The computation of these invariants is rather difficult since 
generators for the corresponding Fukaya categories are not presently
known.  However, one particular computation by Seidel \cite{se:le}
gives some information in the case of a five-punctured two-sphere with
equal labels.  In the last section we leverage Seidel's computation
to make a computation in the symplectic mapping class group
\[\Map(M(X,\ul{x}),\omega) = \pi_0(\Symp (M(X,\ul{x}),\omega)) \]
of the moduli spaces of flat bundles.  Since these moduli spaces are
simply connected, $H^1(M(X,\ul{x});\R) = 0$, and two
symplectomorphisms are symplectically isotopic if and only if they are
Hamiltonian isotopic.

\begin{theorem} \label{thm:braid} {\rm (Non-triviality of twists on 
moduli spaces of bundles on punctured spheres)} Let $X$ be a
  two-sphere and $\ul{x} \subset X$ an odd number of at least five
  marked points.  Let $M(X,\ul{x})$ be the moduli space of flat
  $SU(2)$-bundles with traceless holonomies on $X - \ul{x}$ and
  $\varphi: M(X,\ul{x}) \to M(X,\ul{x})$ the symplectomorphism induced
  by a full twist around two markings.  Then $\varphi$ is not
  Hamiltonian isotopic to the identity but is smoothly isotopic to the
  identity:
\[ [\varphi] \neq [\on{Id}] \in \Map(M(X,\ul{x}),\omega), \quad
  [\varphi] = [\on{Id}] \in \Map(M(X,\ul{x})) .\]
\end{theorem}

\begin{remark} \label{rem:braidlit}
The first statement of Theorem \ref{thm:braid} has been
strengthened since the first version of the paper appeared.  For five markings, the moduli space $M(X,\ul{x})$ is
the monotone blow-up of the projective plane at five points, and the
natural map from the extended mapping class group of $(X,\ul{x})$ to
the symplectic mapping class group of $M(X,\ul{x})$ is an isomorphism
by work of Evans \cite{ev:smcg} building on Seidel \cite{se:le}.
Daemi-Scaduto \cite{ds:symplectic} show that for any odd number $n$
of markings on a surface of genus $g$ with $3g + n > 4$ this map is
injective; in particular, the full twist $\varphi$ around two markings
defines an element of infinite order.  See also Li-Li-Wu
\cite{llw:torelli} for symplectic Torelli groups of rational surfaces.
\end{remark}

The structure of the paper is as follows.  In Section 2 we describe
our strategy for defining tangle invariants via Cerf decompositions.
In Section 3 we then show that the moduli spaces of flat bundles
with admissible holonomy labels fit into this blueprint. In
particular, the sequence of Lagrangian correspondences obtained as
sketched above is independent of the choice of decomposition up to an
equivalence relation generated by embedded composition of Lagrangian
correspondences, cf.\ \cite{ww:isom}.  In Section 4 we introduce a suitable notion of
Fukaya category adapted to the moduli spaces of flat bundles under
consideration.  In Section 5 we combine the constructions of Sections
2,3,4 to obtain a category-valued field theory, or rather, a functor
from our tangle categories to (small \ainfty categories, homotopy
classes of \ainfty functors).  The equivalence of generalized
Lagrangian correspondences proved in Section 3 combines with the
results of \cite{we:co} to show that the resulting functor is
independent up to isomorphism of the decomposition into elementary
pieces.  This section also contains an extension to graphs, needed for
a surgery exact triangle.

\vskip .1in 
\noindent {\em Acknowledgements: } 
We thank P. Seidel for encouragement and for sharing his ideas.  We
also thank R. Rezazadegan for helpful comments.  The present paper is
an updated and more detailed version of a paper the authors have
circulated since 2007. The authors have unreconciled differences over
the exposition in the paper, and explain their points of view at
{\small
  \href{https://math.berkeley.edu/~katrin/wwpapers/}{math.berkeley.edu/$\sim$katrin/wwpapers}}
resp.  {\small
  \href{http://christwoodwardmath.blogspot.com/}{christwoodwardmath.blogspot.com}}. The
publication in the current form is the result of a mediation.

\vskip .1in 
\noindent {\em Artificial intelligence methodology: }  This revision  benefited from several referee reports from large language models, most notably a report from Claude Opus 5.5.  Some of the text added to the previous arxiv version was produced by Claude as part of the refereeing discussion and edited for clarity.

\section{Field theory for tangles} 

In this section we introduce various notions and constructions of
(topological) field theories for tangles.  Roughly speaking a field
theory is a functor from a bordism category to some other category.
In Section~\ref{ss:tangle} we use embedded bordisms in cylinders to
construct a category of tangles. Section~\ref{ss:cerf} discusses Cerf
decompositions in this category and shows how to use them in the
construction of general field theories.  Section~\ref{ss:symp} then
specializes this construction to a symplectic target category.

\subsection{The tangle category} \label{ss:tangle}

Our language for topological field theories for tangles adapts that in
Lurie \cite{lurie:class}, rephrasing the earlier definition of Atiyah.
Roughly speaking, a tangle is a bordism between marked surfaces
containing an embedded one-manifold.  The field theories constructed
in this paper are defined on tangles in a fixed cylinder, up to
isotopy.

\begin{definition} \label{def:tancat}
\begin{enumerate}
\item {\rm (Marked surfaces)} A {\em marking} of a compact oriented
  surface $X$ is a collection
\[\ul{x} = \{ x_1,\ldots, x_n \} \subset X\] 
of distinct points for some non-negative integer $n$, equipped with an
orientation given by a function
\[ \eps :\ul{x} \to \{ \pm 1 \} .\]  
A {\em marked surface} is a tuple $(X,\ul{x})$ of a compact, oriented
surface $X$ equipped with a marking $\ul{x}$.
\item \label{tangdef} {\rm (Tangles)} A {\em tangle}
from $(X_-,\ul{x}_-)$ to $(X_+,\ul{x}_+)$ is a tuple $(Y,K,\phi)$
consisting of
\begin{enumerate} 
\item a compact oriented $3$-manifold-with-boundary $Y$;
\item an orientation-preserving diffeomorphism 
\[ \phi: \partial Y \to \ol{X}_- \cup X_+ \]
where $\ol{X}_-$ denotes the manifold $X_-$ with reversed
orientation;
\item a compact oriented $1$-dimensional submanifold $K \subset Y$
  meeting the boundary transversally in $\partial K = K\cap\partial
  Y$, so that $\phi$ restricts to an orientation preserving
  identification
\[ \phi |_{\partial K}: \ \partial K \cong {\ul{\ol{x}}_-} \cup
  \ul{x}_+\] 
where ${\ul{\ol{x}}_-}$ denotes the marking $\ul{x}_-$ with reversed
orientation.
\end{enumerate} 
An {\em equivalence} between two tangles $(Y_0,K_0,\phi_0)$ and $
(Y_1,K_1,\phi_1)$, both from $(X_-,\ul{x}_-)$ to $(X_+,\ul{x}_+)$, is
an orientation-preserving diffeomorphism inducing the identity on the
boundary surfaces:
\[\psi: Y_0 \to Y_1, \quad \psi(K_0)=K_1, \quad
\phi_1\circ\psi|_{\partial Y_0} = \phi_0 .\]
\item {\rm (Labelled tangles)} Let $\B$ be a set, which we call a set
  of {\em labels}.  A {\em decorated surface resp. tangle } is a
  marked surface $(X,\ul{x})$ resp. tangle $(Y,K,\phi)$ equipped with
  a {\em labelling} of the components $\ul{x} \to \B$ resp. $\pi_0(K)
  \to \B$.
\item {\rm (Cylindrical tangles)} Cylindrical tangles are tangles
  in a fixed product of a surface with an interval.  Let $X$ be a fixed
  compact, oriented surface, and let $t$ denote the coordinate on the first factor of
  $[-1,1] \times X$.  An {\em $X$-cylindrical tangle} from a marking
  $\ul{x}_-$ of $X$ to a marking $\ul{x}_+$ of $X$ is a compact
  oriented $1$-dimensional submanifold
\[ K \subset [-1,1] \times X, \quad \partial K = K \cap ( \{ \pm 1
\} \times X) \]
which is {\em vertical near the boundary}, in the sense that for some
$\eps > 0$
\[ K \cap ( [-1,-1+\eps) \times X ) = [-1,-1+\eps) \times \ul{x}_- \] 
\[ K \cap ( (1-\eps,1] \times X) = (1 - \eps,1] \times
  \ul{x}_+ ,\]
with orientations such that $\partial K = \ul{\ol{x}}_- \cup
\ul{x}_+$.  We regard $K$ as a tangle $([-1,1] \times X, K, \phi)$
from $(X,\ul{x}_-)$ to $(X,\ul{x}_+)$, where $\phi$ is the obvious
identification of $\partial([-1,1] \times X)$ with $\ol{X} \cup X$.
More generally, for an interval $[b_-,b_+]$ we use the same
terminology for submanifolds of $[b_-,b_+] \times X$ vertical near the
boundary.
\item {\rm (Isotopy)} Cylindrical tangles are isotopic if they can be
  deformed into each other through cylindrical tangles.  Two
  $X$-cylindrical tangles $K_0,K_1$ from
  $\ul{x}_-$ to $\ul{x}_+$ are {\em isotopic} if they are connected by
  a smooth family $K_s, s \in [0,1]$ of $X$-cylindrical tangles from
  $\ul{x}_-$ to $\ul{x}_+$, vertical near the boundary for a constant
  $\eps > 0$ independent of $s$.
\item {\rm (Braids)} Braids are cylindrical tangles that are
  monotone in the interval direction.  An $X$-cylindrical tangle $K$
  is a {\em braid} if
  the restriction $t|_K$ of the coordinate $t$ has no critical points.
\end{enumerate} 
\end{definition} 

\begin{remark} \label{rem:isotopy}
Isotopy is a finer relation than equivalence of tangles.  Indeed, by a
relative version of the isotopy extension theorem, whose absolute
version is \cite[Theorem 1.6, Chapter 8]{hirsch} and whose relative
version is proved in the same way, an isotopy $K_s, s \in [0,1]$ of
$X$-cylindrical tangles is of the form
\[ K_s = \Psi_s(K_0), \quad s \in [0,1] \]
for a smooth family of diffeomorphisms $\Psi_s$ of $[-1,1] \times X$
equal to the identity near the boundary, with $\Psi_0$ the identity.
So isotopic tangles are equivalent.  The converse does not hold in
general, since a diffeomorphism equal to the identity near the
boundary need not be isotopic to the identity through such
diffeomorphisms.  For example, let $X = S^2$ and let $g_t \in SO(3), t
\in [-1,1]$ be a loop of rotations, constant equal to the identity
near $t = \pm 1$, representing the generator of $\pi_1(SO(3)) \cong
\Z_2$.  The diffeomorphism
\[ \Psi: [-1,1] \times S^2 \to [-1,1] \times S^2, \quad (t,x) \mapsto
(t, g_t(x)) \]
maps the trivial tangle $[-1,1] \times \ul{x}$ with $n \ge 3$ strands
to the full twist, which represents a non-trivial element of the
spherical braid group.  The field theories in this paper are defined
on isotopy classes.  The full twist is trivial in the mapping class
group of $(S^2,\ul{x})$, and so acts trivially on the moduli spaces of
flat bundles below.  In fact the field theory of Theorem
\ref{thm:extendt} depends only on equivalence classes of tangles, see
Proposition \ref{prop:descent}.
\end{remark}

Our field theories fit into the language of {\em topological field
  theories}. These are functors from bordism categories equipped with
additional data.

\begin{definition}   {\rm (Tangle category)} The {\em tangle category} $\Tan$ is the
  category whose
\begin{enumerate}
\item objects are marked surfaces;
\item morphisms are equivalence classes of tangles 
$[Y,K,\phi]$;
\item composition is defined by gluing: Let
  $(Y_{01},K_{01},\phi_{01})$ be a tangle from $(X_0,\ul{x}_0)$ to
  $(X_1,\ul{x}_1)$ and let $(Y_{12},K_{12},\phi_{12})$ be a tangle
  from $(X_1,\ul{x}_1)$ to $(X_2,\ul{x}_2)$.  Choose collar
  neighborhoods
\[ \kappa_1 : (X_1 \times
  (-\eps,0),\ul{x}_1 \times (-\eps,0)) \to (Y_{01},K_{01}) \] 
resp.  
\[ \kappa_2: (X_1 \times (0,\eps), \ul{x}_1 \times (0,\eps)) \to
(Y_{12},K_{12}).\] 
Define the composition $(Y_{01},K_{01},\phi_{01}) \circ
(Y_{12},K_{12},\phi_{12})$ to be the union
\begin{equation} \label{eq:composeY} \bigl( (Y_{01},K_{01}) \sqcup
  (X_1 \times (-\eps,\eps), \ul{x}_1 \times (-\eps,\eps)) \sqcup
  (Y_{12},K_{12}) \bigr) / \sim \end{equation}
where $\sim$ is the natural equivalence relation defined by
$\kappa_1,\kappa_2$, and equipped with the diffeomorphism of the
boundary to $(X_0,\ul{x}_0) \sqcup (X_2,\ul{x}_2)$ induced by
$\phi_{01}$ and $\phi_{12}$;
\item the identity for $(X,\ul{x})$ is the equivalence class of the
  cylindrical bordism $\bigl[ [-1,1]\times X, [-1,1]\times \ul{x}
    \bigr]$ equipped with the obvious identification of the boundary
  $\{ -1,1\} \times (X,\ul{x})$ with two copies of $(X,\ul{x})$.
\end{enumerate}
Composition is independent, up to equivalence, of the choice of collar
neighborhood and representatives, since any two collar neighborhoods
are isotopic.  The equivalence class of a composition of
representatives is denoted
\[[(Y_{01},K_{01},\phi_{01})] \circ [(Y_{12},K_{12},\phi_{12})] =
[(Y_{01},K_{01},\phi_{01}) \circ (Y_{12},K_{12},\phi_{12})] .\]
\end{definition} 

\begin{definition} \label{def:cobtangles} \label{sec:tangle} {\rm
    (Cylindrical tangle category)} Let $X$ be a compact oriented
  surface.  The {\em cylindrical tangle category} $\Tan(X)$ is the
  category whose
\begin{enumerate}
\item objects are markings $\ul{x}$ of $X$;
\item morphisms from $\ul{x}_-$ to $\ul{x}_+$ are isotopy classes
  $[K]$ of $X$-cylindrical tangles from $\ul{x}_-$ to $\ul{x}_+$;
\item \label{stack} composition is defined by stacking: For $X$-cylindrical tangles
  $K_{01}$ from $\ul{x}_0$ to $\ul{x}_1$ and $K_{12}$ from $\ul{x}_1$
  to $\ul{x}_2$, define
\[ [K_{01}] \circ [K_{12}] := [ \sigma_-(K_{01}) \cup \sigma_+(K_{12})
] \]
where 
\[ \sigma_\pm: [-1,1] \times X \to [-1,1] \times X, \quad (t,x)
\mapsto ((t \pm 1)/2, x) ;\]
\item the identity of $\ul{x}$ is the class of the trivial tangle
  $[-1,1] \times \ul{x}$.
\end{enumerate}
The union in \eqref{stack} is a smooth submanifold since the tangles are
vertical near the boundary.  Composition is well-defined on isotopy
classes, associative, and unital.   Each of these claims follows from the fact that 
the two sides of each identity differ by a reparametrization of the $t$-coordinate fixing a
neighborhood of $t = \pm 1$, which is isotopic to the identity through
such reparametrizations, uniquely up to homotopy.  The category $\Tan(X)$ is related to $\Tan$
by the functor mapping $\ul{x}$ to $(X,\ul{x})$ and $[K]$ to the
equivalence class of $([-1,1] \times X,K,\phi)$, see Remark
\ref{rem:isotopy}.
\end{definition} 

\begin{definition}\label{def:field theory}  {\rm (Field theories)}  
Let $X$ be a compact oriented surface and let $\cC$ be a category.  A
{\em $\cC$-valued field theory for cylindrical tangles} in $X$ is a
functor $\Phi: \Tan(X) \to \cC$.
\end{definition} 

\subsection{Cerf theory for tangles}  \label{ss:cerf}

Field theories for cylindrical tangles can be constructed by
decomposing tangles into pieces, each containing at most one critical
point of the height function.

\begin{definition}  \label{def:handletangle}
\begin{enumerate} 
\item \label{morsedatum} {\rm (Cylindrical Morse datum)} A cylindrical Morse datum
  consists of a height function together with a choice of cutting
  levels.  Let $K$ be an $X$-cylindrical tangle from $\ul{x}_-$ to $\ul{x}_+$, and $Y = [-1,1] \times X$.  A
  {\em cylindrical Morse datum} for $K$ is a pair $(f,\ul{b})$
  consisting of
\begin{enumerate} 
\item a smooth function $f: Y \to [-1,1]$ such that
\[ \partial_t f > 0 \ \text{on} \ Y, \quad f(t,x) = t \ \text{for} \ |t|
\ \text{close to} \ 1, \]
and such that the restriction $f|_K$ is a Morse function; and
\item an ordered tuple 
\[ \ul{b}= ( -1 = b_0 < b_1 < \ldots < b_{m} = 1 ) \in \R^{m+1} \]
\end{enumerate} 
such that the following hold:
\begin{enumerate}
\item \label{first} The function $f|_K$ has distinct values at its
  critical points.
\item \label{second} The values $b_1,\ldots, b_{m-1}$ are regular values
  of $f|_K$, and each interval $(b_{i-1},b_i)$ contains at most one
  critical value of $f|_K$:
\[ \# \bigl( \Crit(f|_K) \cap f^{-1}(b_{i-1},b_i) \bigr) \leq 1,
\quad i = 1,\ldots, m . \]
\end{enumerate}
The basic example is $f = t$, in the case that $t|_K$ is Morse with
distinct critical values.
\item \label{levelsets} {\rm (Level sets)} A cylindrical Morse datum identifies each
  level set with $X$.  Since $\partial_t f > 0$ and $f(\pm 1,x) = \pm
  1$, for each $b \in [-1,1]$ the level set $f^{-1}(b)$ is the graph of
  a smooth function $X \to [-1,1]$, and the projection
\[ \pi_b: f^{-1}(b) \to X, \quad (t,x) \mapsto x \]
is a diffeomorphism.  Furthermore,
\[ F: Y \to Y, \quad (t,x) \mapsto (f(t,x),x) \]
is a diffeomorphism equal to the identity near $\partial Y$ with $t
\circ F = f$, and $F$ is isotopic to the identity through such
diffeomorphisms via the family 
\[ F_s(t,x) = ((1-s)t + s f(t,x), x) . \]
Hence $F(K)$
is an $X$-cylindrical tangle isotopic to $K$, and $(f,\ul{b})$ is a
cylindrical Morse datum for $K$ if and only if $(t,\ul{b})$ is a
cylindrical Morse datum for $F(K)$.
\item 
\label{cerfdecomp} {\rm (Cerf decomposition)} A cylindrical Morse
datum cuts a tangle into pieces along level sets.  The {\em
  cylindrical Cerf decomposition} of $K$ induced by $(f,\ul{b})$ is
the sequence of pieces
\[ Y_i := f^{-1}( [b_{i-1},b_i]), \quad K_i := K \cap Y_i, \quad i =
1,\ldots, m \]
with the intermediate markings
\[ \ul{x}_i := \pi_{b_i} (K \cap f^{-1}(b_i)) \subset X, \quad i =
0,\ldots, m ,\]
so that $\ul{x}_0 = \ul{x}_-$ and $\ul{x}_m = \ul{x}_+$.  The map $F$
identifies $(Y_i,K_i)$ with $([b_{i-1},b_i] \times X, F(K) \cap
([b_{i-1},b_i] \times X))$.  Since $F(K)$ is transverse to the levels
$\{ b_i \} \times X$, after an isotopy of $K$ supported near the level
sets $f^{-1}(b_i)$ and creating no critical points of $f|_K$ we may
assume that $F(K)$ is vertical near each level $\{ b_i \} \times X$.
After rescaling $[b_{i-1},b_i]$ affinely to $[-1,1]$, each piece then
defines an $X$-cylindrical tangle, still denoted $K_i$, from
$\ul{x}_{i-1}$ to $\ul{x}_i$, well-defined up to isotopy, and
\begin{equation} \label{eq:decompYK}
[K] = [K_1] \circ [K_2] \circ \ldots \circ [K_m] 
\end{equation}
in $\Tan(X)$.  
\item {\rm (Elementary tangles)} Elementary tangles are those with
  at most one critical point.  An $X$-cylindrical tangle $K$ is an
  {\em elementary cylindrical tangle} if it admits a cylindrical Morse
  datum with a single piece. That is, if there is a function $f$ as in
  \eqref{morsedatum} such that $f|_K$ has at most one critical point:
\[ \# \Crit(f|_K) \leq 1 .\]
Equivalently, after replacing $K$ by the isotopic tangle $F(K)$, the
restriction $t|_K$ has at most one critical point.  The class $[K]$ of
an elementary cylindrical tangle is an {\em elementary morphism}.  A
braid is an elementary cylindrical tangle without critical points.
\item {\rm (Equivalence of decompositions)} Decompositions that
  differ by moving the intermediate markings are considered
  equivalent.  Two cylindrical Cerf decompositions
\begin{eqnarray*} [K] &=& [K_1] \circ \ldots \circ [K_m] \\ 
&=& [K_1'] \circ \ldots \circ [K_m'] \end{eqnarray*}
with intermediate markings $\ul{x}_i$ resp. $\ul{x}_i'$ are {\em
  equivalent} if there exist braids $\beta_i$ from $\ul{x}_i$ to
$\ul{x}_i'$ for $i = 0,\ldots, m$, with $\beta_0$ and $\beta_m$
trivial, such that
\[ [K_i] \circ [\beta_i] = [\beta_{i-1}] \circ [K_i'] , \quad i =
1,\ldots, m. \]
\end{enumerate} 
\end{definition} 

The following is a special case of Cerf theory, for cylindrical Cerf
decompositions.

\begin{theorem} \label{thm:cerftangle} {\rm (Cerf theory for tangles)}
  Let $K_0$ and $K_1$ be isotopic $X$-cylindrical tangles.  Then any
  cylindrical Cerf decomposition of $K_0$ and any cylindrical Cerf
  decomposition of $K_1$ are related, up to equivalence, by a finite
  sequence of the following moves and their inverses:
\begin{enumerate} 
\item \label{cancel} {\rm (Critical point cancellation)} Two elementary morphisms
  $[K_i]$ and $[K_{i+1}]$ that carry, for some cylindrical Morse
  datum, a local minimum $y_i \in K_i$ resp.\ local maximum $y_{i+1}
  \in K_{i+1}$,
\[ \d^2_{y_i} f |_{K_i} > 0 , \quad \d^2_{y_{i+1}} f|_{K_{i+1}} < 0 \]
both of which lie on the same strand of $K_i \cup K_{i+1}$, are
replaced by the elementary morphism $[K_i] \circ [K_{i+1}]$, which is
the class of a braid.  The inverse move is {\em critical point
  creation}.
\item \label{reversal} {\rm (Critical point reversal)} Two elementary morphisms
  $[K_i], [K_{i+1}]$ that carry critical points
\[y_i \in K_i, \ y_{i+1} \in K_{i+1}, \quad \d_{y_i} f = \d_{y_{i+1}}
  f = 0, \quad k = \ind(y_i), \quad l = \ind(y_{i+1})\]
on strands whose intersections with the level set containing
$\ul{x}_i$ are disjoint, are replaced by two elementary morphisms
$[K_i'], [K_{i+1}']$ carrying critical points of index $l$ and $k$,
such that
\[ [K_i] \circ [K_{i+1}] = [K_i'] \circ [K_{i+1}'] . \]
\item \label{braidglue} {\rm (Braid gluing)} Two elementary morphisms $[K_i],
  [K_{i+1}]$, one of which is the class of a braid, are replaced by
  the elementary morphism $[K_i] \circ [K_{i+1}]$.  The inverse move
  splits off a braid.
\end{enumerate}
\end{theorem}

See Figures \ref{fig:minmax1} and \ref{fig:switch} for depictions of the first
two moves. 

\begin{figure}[h]
\setlength{\unitlength}{0.00037489in}
\begingroup\makeatletter\ifx\SetFigFont\undefined
\def\x#1#2#3#4#5#6#7\relax{\def\x{#1#2#3#4#5#6}}%
\expandafter\x\fmtname xxxxxx\relax \def\y{splain}%
\ifx\x\y   
\gdef\SetFigFont#1#2#3{%
  \ifnum #1<17\tiny\else \ifnum #1<20\small\else
  \ifnum #1<24\normalsize\else \ifnum #1<29\large\else
  \ifnum #1<34\Large\else \ifnum #1<41\LARGE\else
     \huge\fi\fi\fi\fi\fi\fi
  \csname #3\endcsname}%
\else
\gdef\SetFigFont#1#2#3{\begingroup
  \count@#1\relax \ifnum 25<\count@\count@25\fi
  \def\x{\endgroup\@setsize\SetFigFont{#2pt}}%
  \expandafter\x
    \csname \romannumeral\the\count@ pt\expandafter\endcsname
    \csname @\romannumeral\the\count@ pt\endcsname
  \csname #3\endcsname}%
\fi
\fi\endgroup
{\renewcommand{\dashlinestretch}{30}
\begin{picture}(6774,2829)(0,-10)
\path(4512,2667)(4512,2665)(4513,2659)
	(4514,2650)(4515,2635)(4517,2614)
	(4520,2588)(4524,2556)(4529,2520)
	(4536,2480)(4543,2438)(4551,2395)
	(4560,2351)(4570,2307)(4581,2265)
	(4593,2225)(4607,2186)(4621,2150)
	(4636,2116)(4653,2084)(4671,2054)
	(4691,2026)(4712,2000)(4735,1975)
	(4761,1952)(4788,1929)(4818,1908)
	(4850,1887)(4875,1872)(4901,1857)
	(4929,1842)(4957,1828)(4988,1813)
	(5019,1799)(5052,1784)(5086,1769)
	(5122,1754)(5158,1739)(5196,1724)
	(5235,1709)(5274,1694)(5315,1678)
	(5356,1662)(5397,1646)(5440,1630)
	(5482,1614)(5525,1597)(5567,1580)
	(5609,1563)(5652,1546)(5693,1529)
	(5734,1511)(5775,1493)(5814,1475)
	(5853,1457)(5891,1439)(5927,1421)
	(5963,1402)(5997,1383)(6030,1364)
	(6061,1344)(6092,1325)(6120,1305)
	(6148,1284)(6174,1263)(6200,1242)
	(6226,1217)(6251,1192)(6274,1165)
	(6296,1137)(6317,1108)(6336,1077)
	(6353,1044)(6370,1010)(6385,973)
	(6400,934)(6414,892)(6426,848)
	(6438,802)(6449,753)(6460,701)
	(6470,647)(6479,592)(6487,534)
	(6495,477)(6502,419)(6509,361)
	(6514,306)(6519,253)(6524,205)
	(6527,161)(6530,122)(6532,89)
	(6534,63)(6535,43)(6536,28)
	(6537,19)(6537,14)(6537,12)
\dashline{60.000}(4062,2307)(6762,2307)
\dashline{60.000}(4062,732)(6762,732)
\path(3162,1632)(3612,1632)
\path(3492.000,1602.000)(3612.000,1632.000)(3492.000,1662.000)
\path(462,2802)(462,2800)(462,2795)
	(462,2786)(462,2771)(463,2751)
	(463,2725)(464,2692)(464,2653)
	(465,2609)(466,2560)(467,2507)
	(468,2451)(470,2394)(471,2336)
	(473,2277)(475,2220)(477,2164)
	(479,2110)(481,2058)(484,2008)
	(487,1961)(489,1916)(492,1874)
	(496,1834)(499,1797)(503,1761)
	(507,1728)(511,1696)(516,1666)
	(520,1637)(526,1610)(531,1583)
	(537,1557)(545,1524)(554,1492)
	(564,1462)(574,1432)(585,1402)
	(597,1374)(609,1346)(622,1320)
	(636,1294)(651,1269)(666,1246)
	(681,1224)(697,1203)(714,1184)
	(730,1166)(747,1150)(764,1136)
	(781,1123)(798,1112)(815,1102)
	(831,1094)(848,1088)(864,1083)
	(880,1080)(896,1078)(912,1077)
	(929,1078)(947,1080)(964,1084)
	(982,1089)(1000,1096)(1018,1104)
	(1036,1114)(1054,1125)(1073,1137)
	(1091,1151)(1109,1166)(1127,1181)
	(1144,1198)(1161,1215)(1177,1233)
	(1193,1252)(1207,1270)(1221,1289)
	(1234,1308)(1246,1326)(1258,1345)
	(1268,1363)(1278,1381)(1287,1399)
	(1297,1421)(1306,1443)(1314,1465)
	(1322,1488)(1329,1510)(1336,1534)
	(1343,1558)(1349,1582)(1356,1606)
	(1362,1631)(1368,1655)(1375,1680)
	(1381,1704)(1388,1727)(1395,1750)
	(1402,1772)(1410,1794)(1418,1815)
	(1427,1836)(1437,1857)(1447,1876)
	(1458,1895)(1469,1914)(1482,1933)
	(1496,1952)(1510,1971)(1526,1990)
	(1543,2009)(1560,2027)(1578,2044)
	(1597,2061)(1617,2077)(1637,2092)
	(1656,2105)(1676,2117)(1696,2128)
	(1716,2137)(1736,2144)(1755,2150)
	(1774,2154)(1793,2156)(1812,2157)
	(1827,2156)(1842,2155)(1856,2152)
	(1872,2148)(1887,2143)(1902,2137)
	(1918,2129)(1933,2119)(1949,2109)
	(1965,2096)(1980,2083)(1996,2067)
	(2011,2050)(2027,2032)(2042,2012)
	(2056,1991)(2070,1969)(2084,1945)
	(2097,1920)(2109,1894)(2121,1866)
	(2132,1838)(2143,1809)(2153,1779)
	(2162,1747)(2171,1715)(2179,1682)
	(2187,1647)(2192,1621)(2197,1593)
	(2202,1565)(2206,1535)(2211,1504)
	(2215,1472)(2218,1438)(2222,1402)
	(2225,1364)(2228,1324)(2231,1282)
	(2234,1237)(2236,1190)(2239,1140)
	(2241,1088)(2243,1033)(2245,976)
	(2247,917)(2249,856)(2251,794)
	(2252,730)(2254,667)(2255,604)
	(2256,542)(2257,483)(2258,426)
	(2259,374)(2260,326)(2260,284)
	(2261,247)(2261,217)(2262,193)
	(2262,174)(2262,162)(2262,153)
	(2262,149)(2262,147)
\dashline{60.000}(12,2307)(2712,2307)
\dashline{60.000}(12,1632)(2667,1632)
\dashline{60.000}(12,732)(2712,732)
\end{picture}
}
\caption{Critical point cancellation}
\label{fig:minmax1}
\end{figure}
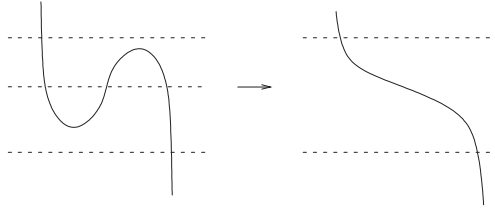

\begin{proof} 
The proof examines a generic one-parameter family of tangles
connecting the two decompositions.  We first reduce to the case that
both Morse data have height function $t$.  Let $(f_j,\ul{b}_j)$ be 
cylindrical Morse datum for $K_j$, $j = 0,1$, and $F_j$ the
diffeomorphism of Definition \ref{def:handletangle} \eqref{levelsets}.  By
construction, the decomposition of $K_j$ induced by $(f_j,\ul{b}_j)$
is the decomposition of the isotopic tangle $F_j(K_j)$ induced by
$(t,\ul{b}_j)$.  So we may assume that $f_0 = f_1 = t$, and that $K_0$
and $K_1$ are isotopic tangles on which $t$ restricts to a Morse
function with distinct critical values.

We put the isotopy in general position as follows.  By the relative isotopy
extension theorem as in Remark \ref{rem:isotopy}, there is a smooth
family of diffeomorphisms $\Psi_s$ of $[-1,1] \times X$, equal to the
identity near the boundary, with $\Psi_0$ the identity and
$\Psi_1(K_0) = K_1$.  Let $K_s := \Psi_s(K_0)$ and consider the family
of functions
\[ g_s := t \circ \Psi_s |_{K_0} : K_0 \to [-1,1], \quad s \in [0,1] .\]
Since $K_0$ is vertical near the boundary and $\Psi_s$ is the identity
there, the functions $g_s$ have no critical points near $\partial
K_0$.  After a perturbation of the family $g_s$ fixed near $s = 0,1$
and near $\partial K_0$, we may assume by \cite[Theorem 2.4]{gg:stable}
that $g_s$ is Morse except at finitely many values of $s$ where a
birth/death singularity occurs, and, as in Cerf \cite[top of p.
  11]{ce:st}, that $g_s$ is injective on its critical set except at
finitely many further values of $s$, where two critical values cross.
Such a perturbation is realized by an isotopy: A small perturbation
$g_s + h_s$, with $h_s$ supported away from $\partial K_0$, is the
family of functions for the tangles obtained from $K_s$ by moving each
point $\Psi_s(y)$ in the $t$-direction by $h_s(y)$.
For $h_s$
small, the sum is again an isotopy of $X$-cylindrical tangles.  Let $c_1 <
\ldots < c_k$ denote the parameters at which $t|_{K_s}$ fails to be
Morse with distinct critical values.

Away from these parameters the decompositions change only by
equivalence.  On an interval $[c_j + \eps, c_{j+1} - \eps]$, choose
smoothly varying values 
\[ b_1(s) < \ldots < b_{m-1}(s) \] 
separating the
critical values of $t|_{K_s}$.  The pieces
\[ K_s \cap ( [b_{i-1}(s), b_i(s)] \times X ), \quad i = 1,\ldots,
m \]
straightened near the levels and rescaled to $[-1,1] \times X$ as in
Definition \ref{def:handletangle} \eqref{cerfdecomp}, form smooth families $K_i(s)$
of elementary cylindrical tangles from $\ul{x}_{i-1}(s)$ to
$\ul{x}_i(s)$, where $\ul{x}_i(s)$ is the projection of $K_s \cap (\{
b_i(s) \} \times X)$ to $X$.  For parameters $s < s'$ in the interval,
let $\beta_i$ be the braid traced out by the markings
$\ul{x}_i(u), u \in [s,s']$, after reparametrizing $[s,s']$ to
$[-1,1]$.  Since $K_i(u), u \in [s,s']$ is a smooth family,
\[ [K_i(s)] \circ [\beta_i] = [\beta_{i-1}] \circ [K_i(s')] , \quad i
= 1,\ldots, m ,\]
and $\beta_0, \beta_m$ are trivial since $b_0 = -1$ and $b_m = 1$ are
fixed.  Hence the decompositions at $s$ and $s'$ are equivalent.  A
different choice of separating values changes the decomposition by
splitting off or gluing braids, that is, by braid gluing and its
inverse.  The same holds at $s = 0,1$ for the given tuples
$\ul{b}_0,\ul{b}_1$.

It remains to consider dependence on the choice of location of the cuts.  The decompositions for
parameters slightly smaller and slightly larger than $c_j$ agree up to
equivalence except for the one or two pieces containing the critical
points involved.  For these pieces, one has a critical point reversal
if two critical values cross, and a critical point cancellation or
creation at a birth/death singularity, by the local model for the cusp
singularity \cite[p.157]{ma:st5} of the family $g_s$.

In the case of a crossing of critical values, the strands through the
two critical points meet the intermediate level set in disjoint sets,
as required in \eqref{reversal}.  Indeed, suppose that for parameters
$s$ slightly less than $c_j$ some strand through the lower critical
point $y$ meets the intermediate level set and continues to the upper
critical point $y'$.  Then the arc of $K_s$ from $y$ to $y'$ contains
no other critical points, so $g_s$ is strictly monotone along it.  As
$s \to c_j$ the arc persists and the values of $g_s$ at its two ends
converge to the same value, so $g_{c_j}$ is monotone along the
limiting arc with equal values at its ends, and hence constant on it.
Every point of the arc would then be a critical point of $g_{c_j}$,
contradicting the genericity of the family, for which $g_{c_j}$ has
finitely many critical points.
\end{proof}

\begin{figure}[h]
\setlength{\unitlength}{0.00047489in}
\begingroup\makeatletter\ifx\SetFigFont\undefined
\def\x#1#2#3#4#5#6#7\relax{\def\x{#1#2#3#4#5#6}}%
\expandafter\x\fmtname xxxxxx\relax \def\y{splain}%
\ifx\x\y   
\gdef\SetFigFont#1#2#3{%
  \ifnum #1<17\tiny\else \ifnum #1<20\small\else
  \ifnum #1<24\normalsize\else \ifnum #1<29\large\else
  \ifnum #1<34\Large\else \ifnum #1<41\LARGE\else
     \huge\fi\fi\fi\fi\fi\fi
  \csname #3\endcsname}%
\else
\gdef\SetFigFont#1#2#3{\begingroup
  \count@#1\relax \ifnum 25<\count@\count@25\fi
  \def\x{\endgroup\@setsize\SetFigFont{#2pt}}%
  \expandafter\x
    \csname \romannumeral\the\count@ pt\expandafter\endcsname
    \csname @\romannumeral\the\count@ pt\endcsname
  \csname #3\endcsname}%
\fi
\fi\endgroup
{\renewcommand{\dashlinestretch}{30}
\begin{picture}(7269,2064)(0,-10)
\path(3162,912)(4062,912)
\path(3942.000,882.000)(4062.000,912.000)(3942.000,942.000)
\dashline{60.000}(12,1587)(2982,1587)
\dashline{60.000}(12,237)(2937,237)
\dashline{60.000}(12,912)(2937,912)
\dashline{60.000}(4242,912)(7212,912)
\dashline{60.000}(4287,237)(7257,237)
\dashline{60.000}(4287,1587)(7257,1587)
\path(462,2037)(462,2034)(462,2029)
	(462,2018)(463,2002)(463,1979)
	(464,1950)(464,1914)(465,1873)
	(466,1827)(468,1778)(469,1726)
	(471,1672)(473,1619)(475,1566)
	(477,1515)(480,1465)(482,1418)
	(485,1373)(488,1331)(492,1292)
	(495,1255)(499,1220)(504,1188)
	(508,1157)(513,1128)(518,1100)
	(524,1074)(530,1049)(537,1024)
	(546,995)(556,966)(566,938)
	(578,911)(590,884)(603,859)
	(617,833)(631,809)(647,786)
	(663,764)(680,743)(697,723)
	(715,705)(733,689)(751,674)
	(770,660)(788,649)(806,639)
	(824,630)(842,624)(860,618)
	(877,615)(895,613)(912,612)
	(929,613)(947,615)(964,618)
	(982,624)(1000,630)(1018,639)
	(1036,649)(1054,660)(1073,674)
	(1091,689)(1109,705)(1127,723)
	(1144,743)(1161,764)(1177,786)
	(1193,809)(1207,833)(1221,859)
	(1234,884)(1246,911)(1258,938)
	(1268,966)(1278,995)(1287,1024)
	(1294,1049)(1300,1074)(1306,1100)
	(1311,1128)(1316,1157)(1320,1188)
	(1325,1220)(1329,1255)(1332,1292)
	(1336,1331)(1339,1373)(1342,1418)
	(1344,1465)(1347,1515)(1349,1566)
	(1351,1619)(1353,1672)(1355,1726)
	(1356,1778)(1358,1827)(1359,1873)
	(1360,1914)(1360,1950)(1361,1979)
	(1361,2002)(1362,2018)(1362,2029)
	(1362,2034)(1362,2037)
\path(1812,12)(1812,15)(1812,20)
	(1812,31)(1813,47)(1813,70)
	(1814,99)(1814,135)(1815,176)
	(1816,222)(1818,271)(1819,323)
	(1821,377)(1823,430)(1825,483)
	(1827,534)(1830,584)(1832,631)
	(1835,676)(1838,718)(1842,757)
	(1845,794)(1849,829)(1854,861)
	(1858,892)(1863,921)(1868,949)
	(1874,975)(1880,1000)(1887,1024)
	(1896,1054)(1906,1083)(1916,1111)
	(1928,1138)(1940,1165)(1953,1190)
	(1967,1216)(1981,1240)(1997,1263)
	(2013,1285)(2030,1306)(2047,1326)
	(2065,1344)(2083,1360)(2101,1375)
	(2120,1389)(2138,1400)(2156,1410)
	(2174,1419)(2192,1425)(2210,1431)
	(2227,1434)(2245,1436)(2262,1437)
	(2279,1436)(2297,1434)(2314,1431)
	(2332,1425)(2350,1419)(2368,1410)
	(2386,1400)(2404,1389)(2423,1375)
	(2441,1360)(2459,1344)(2477,1326)
	(2494,1306)(2511,1285)(2527,1263)
	(2543,1240)(2557,1216)(2571,1190)
	(2584,1165)(2596,1138)(2608,1111)
	(2618,1083)(2628,1054)(2637,1024)
	(2644,1000)(2650,975)(2656,949)
	(2661,921)(2666,892)(2670,861)
	(2675,829)(2679,794)(2682,757)
	(2686,718)(2689,676)(2692,631)
	(2694,584)(2697,534)(2699,483)
	(2701,430)(2703,377)(2705,323)
	(2706,271)(2708,222)(2709,176)
	(2710,135)(2710,99)(2711,70)
	(2711,47)(2712,31)(2712,20)
	(2712,15)(2712,12)
\path(4467,2037)(4467,2034)(4467,2026)
	(4468,2013)(4469,1994)(4470,1970)
	(4472,1940)(4474,1908)(4477,1875)
	(4480,1841)(4483,1808)(4487,1777)
	(4492,1748)(4497,1721)(4502,1696)
	(4509,1672)(4516,1650)(4523,1629)
	(4532,1608)(4542,1587)(4552,1568)
	(4563,1549)(4574,1530)(4587,1511)
	(4601,1492)(4615,1473)(4631,1454)
	(4648,1435)(4665,1417)(4683,1400)
	(4702,1383)(4722,1367)(4742,1352)
	(4761,1339)(4781,1327)(4801,1316)
	(4821,1307)(4841,1300)(4860,1294)
	(4879,1290)(4898,1288)(4917,1287)
	(4934,1288)(4952,1290)(4969,1293)
	(4987,1298)(5005,1304)(5023,1312)
	(5041,1321)(5059,1331)(5078,1342)
	(5096,1355)(5115,1368)(5133,1383)
	(5150,1398)(5167,1414)(5184,1431)
	(5200,1448)(5215,1465)(5230,1483)
	(5243,1500)(5256,1518)(5268,1535)
	(5279,1553)(5290,1570)(5300,1587)
	(5310,1608)(5320,1629)(5330,1650)
	(5338,1672)(5346,1696)(5354,1721)
	(5361,1748)(5368,1777)(5374,1808)
	(5381,1841)(5387,1875)(5393,1908)
	(5398,1940)(5402,1970)(5406,1994)
	(5409,2013)(5411,2026)(5412,2034)(5412,2037)
\path(5862,12)(5862,15)(5862,23)
	(5863,36)(5864,55)(5865,79)
	(5867,109)(5869,141)(5872,174)
	(5875,208)(5878,241)(5882,272)
	(5887,301)(5892,328)(5897,353)
	(5904,377)(5911,399)(5918,420)
	(5927,441)(5937,462)(5947,481)
	(5958,500)(5969,519)(5982,538)
	(5996,557)(6010,576)(6026,595)
	(6043,614)(6060,632)(6078,649)
	(6097,666)(6117,682)(6137,697)
	(6156,710)(6176,722)(6196,733)
	(6216,742)(6236,749)(6255,755)
	(6274,759)(6293,761)(6312,762)
	(6331,761)(6350,759)(6369,755)
	(6388,749)(6408,742)(6428,733)
	(6448,722)(6468,710)(6487,697)
	(6507,682)(6527,666)(6546,649)
	(6564,632)(6581,614)(6598,595)
	(6614,576)(6628,557)(6642,538)
	(6655,519)(6666,500)(6677,481)
	(6687,462)(6697,441)(6706,420)
	(6713,399)(6720,377)(6727,353)
	(6732,328)(6737,301)(6742,272)
	(6746,241)(6749,208)(6752,174)
	(6755,141)(6757,109)(6759,79)
	(6760,55)(6761,36)(6762,23)
	(6762,15)(6762,12)
\end{picture}
} 
\caption{Critical point reversal}
\label{fig:switch}
\end{figure}
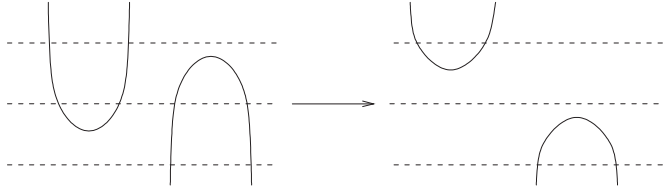

By Theorem \ref{thm:cerftangle}, in order to construct field theories for
cylindrical tangles it suffices to construct the theory on elementary
tangles and check that the Cerf moves are satisfied.

\begin{theorem} \label{thm:suffices} {\rm (Field theories for tangles via elementary tangles)}  
Let $\cC$ be a category and $X$ a compact oriented surface.  Suppose
there is given a partially defined functor $\Phi$ from $\Tan(X)$ to
$\cC$ that associates
\begin{enumerate}
\item to each marking $\ul{x}$ of $X$, an object $\Phi(\ul{x})$ of
  $\cC$;
\item to each elementary morphism $[K]$ from $\ul{x}_-$ to
  $\ul{x}_+$, a morphism $\Phi([K])$ from $\Phi(\ul{x}_-)$ to
  $\Phi(\ul{x}_+)$;
\end{enumerate} 
and satisfies the following {\em Cerf relations}:
\begin{enumerate}
\item \label{cerfrel1} If $K = [-1,1] \times \ul{x}$ is a trivial tangle, then
  $\Phi([K])$ is the identity.
\item \label{cerfrel2} If $K_1$ from $\ul{x}_0$ to $\ul{x}_1$ and $K_2$ from $\ul{x}_1$
  to $\ul{x}_2$ are elementary cylindrical tangles related to the
  braid $K_1 \circ K_2$ by critical point cancellation, then
\[\Phi([K_1]) \circ \Phi([K_2]) = \Phi([K_1] \circ [K_2]) ;\]
\item \label{cerfrel3} If $K_1,K_2$ and $K_1',K_2'$ are elementary cylindrical tangles
  related by critical point reversal, then
\[\Phi([K_1]) \circ \Phi([K_2]) = \Phi([K_1']) \circ \Phi([K_2']) ;\]
\item \label{cerfrel4} If $K_1,K_2$ are composable elementary cylindrical tangles, one
  of which is a braid, then
\[\Phi([K_1]) \circ \Phi([K_2]) = \Phi([K_1] \circ [K_2]) .\]
\end{enumerate} 
Then there is a unique $\cC$-valued field theory extending $\Phi$.
\end{theorem} 

\begin{proof} 
The extension is defined by decomposing a tangle into elementary
pieces.  Let $[K]$ be a morphism of $\Tan(X)$ from $\ul{x}_-$ to
$\ul{x}_+$.  After an isotopy of $K$ supported away from the boundary,
the restriction $t|_K$ is a Morse function with distinct critical
values, since such functions are dense among functions with no
critical points near $\partial K$ and any small perturbation of $t|_K$
is realized by moving $K$ in the $t$-direction, as in the proof of
Theorem \ref{thm:cerftangle}.  Choosing values $b_1 < \ldots < b_{m-1}$
separating the critical values gives a cylindrical Morse datum
$(t,\ul{b})$, and hence a cylindrical Cerf decomposition
\[ [K] = [K_1] \circ \ldots \circ [K_m] \]
into elementary morphisms.  Define
\[ \Phi([K]) := \Phi([K_1]) \circ \ldots \circ \Phi([K_m]) .\]

The definition does not depend on the choices made.  By Theorem
\ref{thm:cerftangle}, any two cylindrical Cerf decompositions of
representatives of $[K]$ are related, up to equivalence, by the moves
\eqref{cancel}--\eqref{braidglue} and their inverses, and the right-hand side of the
definition is invariant under each move by the Cerf relations \eqref{cerfrel2}--\eqref{cerfrel4}.
The definition is also invariant under equivalence: Suppose that
\[ [K_i] \circ [\beta_i] =
[\beta_{i-1}] \circ [K_i'] \] 
for braids $\beta_i$ with $\beta_0,\beta_m$
trivial.  The relation \eqref{cerfrel4}, applied to both sides, gives
\[ \Phi([K_i]) \circ \Phi([\beta_i]) = \Phi([\beta_{i-1}]) \circ
\Phi([K_i']), \quad i = 1,\ldots, m .\]
Since $\Phi([\beta_0])$ and $\Phi([\beta_m])$ are identities by \eqref{cerfrel1},
moving the factors $\Phi([\beta_i])$ from right to left one step at a
time gives
\[ \Phi([K_1]) \circ \ldots \circ \Phi([K_m]) = \Phi([K_1']) \circ
\ldots \circ \Phi([K_m']) .\]

The resulting map is a functor extending $\Phi$, and is unique.  The
identity is preserved by \eqref{cerfrel1}.  For composable morphisms $[K_{01}]$ and
$[K_{12}]$, stacking cylindrical Cerf decompositions of $K_{01}$ and
$K_{12}$ gives a cylindrical Cerf decomposition of $\sigma_-(K_{01})
\cup \sigma_+(K_{12})$, so that
\[ \Phi([K_{01}] \circ [K_{12}]) = \Phi([K_{01}]) \circ \Phi([K_{12}])
.\]
An elementary morphism has a decomposition with a single piece, so the
functor extends the given values.  Any functor extending $\Phi$ is
determined on all morphisms by the decomposition above, which proves
uniqueness.
\end{proof}

In other words, to define a field theory for tangles it suffices to
define the morphisms for elementary bordisms and prove the Cerf
relations.

\subsection{Symplectic-valued field theories} \label{ss:symp}

In this section we specialize to field theories with values in the
{\em symplectic category}.  A symplectic-valued field theory for
tangles in particular assigns to any tangle a sequence of Lagrangian
correspondences, up to equivalence, as in \cite{we:co}. 

\begin{definition} {\rm (Geometric composition of Lagrangian correspondences)}  
Let $M_j$ be symplectic manifolds with symplectic forms $\omega_{M_j}$
for $j = 0 ,1,2$.
\begin{enumerate} 
\item 
A {\em Lagrangian correspondence} from $M_1$ to $M_2$ is a Lagrangian
submanifold $L\subset M_1^-\times M_2$ with respect to the symplectic
structure $-\omega_{M_1} \oplus \omega_{M_2}$.
\item 
The {\em geometric composition} of Lagrangian correspondences
\[L_{01}\subset M_0^-\times M_1, \quad L_{12}\subset M_1^-\times M_2\]
is the point set
\begin{equation} \label{eq:compose} 
L_{01}\circ L_{12} := \pi_{M_0\times M_2} \bigl( ( L_{01} \times
L_{12} ) \cap ( M_0 \times \Delta_{M_1} \times M_2 ) \bigr) \subset
M_0\times M_2 .\end{equation} 
\item A geometric composition is called {\em transverse} if the
  intersection in \eqref{eq:compose} is transverse (and hence smooth).
  The geometric composition is {\em embedded} if, in addition, the
  restriction of the projection $\pi_{M_0\times M_2}$ is an injection
  of the smooth intersection, hence an embedding.  In that case the
  image is a smooth Lagrangian correspondence $L_{01} \circ L_{12}
  \subset M_0^- \times M_2$.
\end{enumerate}  
\end{definition}

\begin{definition}  
\begin{enumerate} 
\item {\rm (Generalized correspondences)} Let $M_-,M_+$ be symplectic
  manifolds.  A {\em generalized Lagrangian correspondence} $\ul{L}$
  from $M_-$ to $M_+$ consists of
\begin{enumerate}
\item a sequence $N_0,\ldots,N_r$ of any length $r\geq 0$ of
  symplectic manifolds with $N_0 = M_-$ and $N_r = M_+$, and
\item a sequence $L_{01},\ldots, L_{(r-1)r}$ of compact Lagrangian
  correspondences with $L_{(j-1)j} \subset N_{j-1}^-\times N_{j}$ for
  $j=1,\ldots,r$.
\end{enumerate}
\item {\rm (Algebraic composition)} Let $M,M',M''$ be symplectic
  manifolds.  The {\em algebraic composition} of generalized
  Lagrangian correspondences $\ul{L}$ from $M$ to $M'$ and $\ul{L}'$
  from $M'$ to $M''$ is given by concatenation
\[\ul{L}{\sharp}\ul{L}' :=(L_{01},\ldots, L_{(m-1)m},L_{1}',\ldots, L_{(m'-1)m'}')
  .
\]
\item {\rm (Symplectic category)} The {\em symplectic category}
  $\Symp^\#$ is the category defined as follows.\footnote{We used the notation $\Symp$ in earlier papers to denote the category of (symplectic manifolds, symplectomorphisms).   The superscript $^\#$ indicates that we now allow generalized Lagrangian correspondences as morphisms.}    
\begin{enumerate} 
\item Objects are smooth compact symplectic manifolds.
\item Morphisms from an object $M_-$ to an object $M_+$ are
  generalized Lagrangian correspondences from $M_-$ to $M_+$ modulo
  the {\em composition equivalence} relation $\sim$ generated by
\begin{equation} \label{eq:compequiv} 
 \bigl(\ldots,L_{(j-1)j},L_{j(j+1)},\ldots \bigr) \sim
 \bigl(\ldots,L_{(j-1)j}\circ L_{j(j+1)},\ldots\bigr)
\end{equation} 
for all sequences and $j$ such that $L_{(j-1)j} \circ L_{j(j+1)}$ is
transverse and embedded.  We also set the empty sequence $\emptyset$
to be equivalent to the diagonal $\Delta_M \subset M^-\times M$.
\item Composition of morphisms
\[[\ul{L}]\in\Hom(M,M'), \quad [\ul{L}']\in\Hom(M',M'')\] 
for symplectic manifolds $M,M',M''$ is defined by
\[ [\ul{L}]\circ[\ul{L}'] := [\ul{L}{\sharp}\ul{L}'] \;\in\Hom(M,M'')
  ;\]
\item The identity $1_M\in\Hom(M,M)$ is the equivalence class of the
  empty sequence $1_M = \emptyset$ of length zero.  The identity $1_M$
  is also the equivalence class $1_M:=[\Delta_M]$ of the diagonal.
  Indeed, the sequence of any number of diagonals $\Delta_M \subset
  M^- \times M$ is equivalent to the empty set:
\[ \emptyset  \sim (\Delta_M) \sim (\Delta_M,\Delta_M) \sim \ldots .\]
\end{enumerate} 
\item {\rm (Monotone symplectic manifolds and correspondences)} A
  symplectic manifold $(M,\omega)$ is {\em monotone} with monotonicity
  constant $\tau > 0 $ if the symplectic class is positively
  proportional to the first Chern class: $\tau c_1(M) = [\omega]$ in
  $H^2(M)$.  A Lagrangian submanifold $L \subset M$ is monotone if
\[2 \int u^*\omega = \tau I(u), \quad \forall u: (D,\partial D) \to
  (M,L) \]
where $I(u)$ is the Maslov index.  A generalized Lagrangian
correspondence $\ul{L} = (L_{01},\ldots, L_{(r-1)r})$ is {\em
  monotone} with constant $\tau$ if each component $L_{(j-1)j}$ is a
monotone Lagrangian submanifold of $N_{j-1}^- \times N_j$ with
constant $\tau$.
\item {\rm (Monotone symplectic category)} Monotone symplectic
  manifolds with a fixed monotonicity constant form a category whose
  morphisms are built from simply connected correspondences.  For
  $\tau > 0$ denote by $\Symp^\#_\tau$ the category whose objects are
  monotone symplectic manifolds $M$ with monotonicity constant $\tau$
  and whose morphisms from $M_-$ to $M_+$ are equivalence classes of
  generalized Lagrangian correspondences $\ul{L}$ from $M_-$ to $M_+$
  whose components are compact, oriented, simply
  connected\footnote{For convenience; alternatively one can impose
    further monotonicity conditions.} and monotone with constant
  $\tau$, and equipped with relative spin structures.  The
  equivalence relation is generated by the moves
  \eqref{eq:compequiv} for those $j$ for which the geometric
  composition is transverse and embedded and the composite
\[ L_{(j-1)j} \circ L_{j(j+1)} \subset N_{j-1}^- \times N_{j+1} \]
is again simply connected, equipped with the orientation and relative
spin structure induced from those of the factors.  This restriction is
necessary, since an embedded geometric composition of simply connected
correspondences need not be simply connected.  The components carry
gradings in the sense of Section \ref{sec:derived fh}: since they are
simply connected, they admit gradings with respect to any $N$-fold
Maslov covers of the ambient manifolds, for example for $N = 2$, where
a grading is equivalent to an orientation.
\item {\rm (Symplectic-valued field theories)} Field theories with
  values in the symplectic category assign Lagrangian correspondences
  to tangles.  A {\em symplectic-valued} resp. {\em monotone
  symplectic-valued field theory for cylindrical tangles} for a compact
  oriented surface $X$ is a functor
\[ \Phi: \Tan(X) \to \Symp^\# \quad \text{resp.} \quad \Phi: \Tan(X)
\to \Symp^\#_\tau .\]
\end{enumerate} 
\end{definition} 

\section{Flat bundles on complements of tangles} 

In this section we construct a symplectic-valued field theory for a
particular class of labelled tangle categories.  For suitable choices
of the labels, this field theory will be monotone.  The basic
construction is well-known: associated to any tangle there is a moduli
space of flat bundles with fixed holonomies around the components.  If
smooth and embedded this moduli space defines a Lagrangian
correspondence in the moduli spaces of flat bundles with fixed
holonomies on the boundary.  For elementary tangles, the
correspondences are smooth and embedded, and we check that the Cerf
relations hold.

\subsection{Moduli spaces via holonomy}
\label{sec:moduli of flat G bundles}

We choose to describe the moduli spaces via representations of the
fundamental group, rather than gauge theory as in \cite{field}.  We
begin with some Lie-theoretic notation for the special unitary group.

\begin{definition}  Let $r \ge 2$ and $G = SU(r)$ the group of special
unitary $r \times r$ matrices.  We identify the Lie algebra $\g = \su(r)$
with traceless skew-Hermitian $r \times r$ matrices. 
\begin{enumerate} 
\item {\rm (Weyl alcove)} The {\em Weyl alcove} for $SU(r)$ is the
  subset
\[ \Alc = \left\{ (\lambda_1 \geq \ldots \geq \lambda_r) \in \R^r
  \ \left| \ \sum_{i=1}^r \lambda_i = 0, \lambda_1 - \lambda_r \leq 1
  \right.\right\} .\]
A point $\mu \in \Alc$ will be called a {\em label}.  The alcove
$\Alc$ embeds as a subset of the Lie algebra $\g$ via the diagonal
map,
\[ \Alc \to \g, \quad (\mu_1,\ldots,\mu_r) \to
\diag(\mu_1,\ldots,\mu_r) .\]
\item 
{ \rm (Conjugacy classes for the special unitary group)} Conjugacy
classes in $SU(r)$ are parametrized by the Weyl alcove via
\[ \cC_\mu = \{ g \exp(\diag( 2\pi i \mu)) g^{-1} \ | \ g \in SU(r)
\}, \quad \mu \in \Alc .\]
Each conjugacy class 
\begin{equation} \label{eq:quotient} \cC_\mu \cong SU(r) / S(U(m_1) \times
\ldots U(m_k)) \end{equation}
is diffeomorphic to the quotient of $SU(r)$ by a centralizer subgroup
isomorphic to $ S(U(m_1) \times \ldots U(m_k))$ where $m_i, i
=1,\dots, k$ are the multiplicities of the eigenvalues.  Thus each
$\cC_\mu$ is diffeomorphic to a partial flag variety.  This implies
that $\cC_\mu$ is simply connected.
\item {\rm (Involution)} Taking inverses defines a (possibly trivial)
  involution of the alcove
\begin{equation} \label{eq:inv}
 *: \Alc \to \Alc, \ (\lambda_1,\ldots, \lambda_r) \mapsto
 (-\lambda_r,\ldots, -\lambda_1), \quad \cC_{*\mu} = \cC_\mu^{-1}
 .\end{equation}
\item {\rm (Vertices)} Let
\[\omega_k = (\underbrace{(r-k)/r, \ldots, (r-k)/r}_k, \underbrace{
  -k/r,\ldots, -k/r}_{r-k}), \quad 1 \leq k \leq r-1, \quad \omega_0 =
  0 \]
denote the vertices of $\Alc$.
\item {\rm (Barycenter)} Let
\[\rho = ( r-1,r-3, \ldots,-r+3, -r+1)/2 \]
denote the barycenter of the re-scaled alcove $r\Alc$.  The vector
$\rho$ is the unique vector with components $\rho_i$ satisfying
\[ \rho_i - \rho_{i+1} = 1, \ i = 1,\ldots,r-1, \quad \sum_{i=1}^r
\rho_i = 0 .\]
The element $\rho/r$ is the barycenter of $\Alc$.
\end{enumerate} 
\end{definition} 

Next we introduce notation for manifolds of flat bundles with fixed
holonomies.  We define these via representations of the fundamental
group.

\begin{definition}  
\begin{enumerate}
\item {\rm (Loops around strands)} Let $X$ be a compact, connected,
  oriented manifold, possibly with boundary.  Let $K \subset X$ be an
  oriented, embedded submanifold of codimension $2$.  Let
  $K_1,\ldots,K_n$ denote the connected components of $K$.  Let
\[\gamma_j: S^1 \to X\setminus K, \quad j = 1,\ldots, n\]
be small loops around $K_j$, so that the induced orientation on the
normal bundle of $K_j$ agrees with that induced by the orientations of
$K_j$ and $X$.  Each $\gamma_j$ defines a conjugacy class $[\gamma_j]
\subset \pi_1(X\setminus K)$ of loops obtained by joining $\gamma_j$
to a base point.  We implicitly fix a base point in the definition of
the fundamental group $\pi_1(X\setminus K)$.
\item {\rm (Moduli of flat bundles with fixed holonomies)} For labels
  $\ul{\mu} = (\mu_1,\ldots,\mu_n) \in \Alc^n$ let $M(X,K)$
  denote the moduli space of flat $G$-bundles on $X \ssm K$ whose
  holonomy around $\gamma_j$ lies in the conjugacy class
  $\cC_{\mu_j}$.  We call the element $\mu_j$ the {\em label} of the
  component $K_j$.  The moduli space $M(X,K)$ of flat bundles
  with fixed holonomy has a description in terms of representations of
  the fundamental group, that we take as a definition:
\begin{equation} \label{eq:holmarked}
 M(X,K) := \bigl\{ \varphi\in \Hom(\pi_1(X \ssm K), G)) \, \big|\,
\varphi([\gamma_j]) \subset \cC_{\mu_j} 
\forall j \bigr\}/G .\end{equation}
Here $G$ acts by conjugation so that $ (g \varphi)([\gamma]) = g
\varphi([\gamma]) g^{-1} .$ In case $X$ is not connected, say the
union of connected components $X_1,\ldots, X_k$, this definition is
replaced by the product of moduli spaces for the connected components
of $X$,
\[ M(X,K) = M(X_1, K \cap X_1) \times \ldots \times M(X_k, K \cap X_k ) .\]
\end{enumerate} 
\end{definition} 

\begin{remark}  {\rm (Effect of orientation change of the tangle)}  
Changing the orientation of a component $K_j$ (i.e.\ of $\gamma_j$)
corresponds to changing the label $\mu_j$ by the involution $*$ of the
alcove $\Alc$ in \eqref{eq:inv}.  That is, if $\ti{K}$ denotes the tangle
obtained by changing the orientation on $K_j$ and $\ti{\ul{\mu}}$ is
the set of labels obtained by replacing $\mu_j$ with $*\mu_j$ then
there is a canonical homeomorphism $M(X,K) \to M(X,\ti{K})$.
\end{remark} 

\begin{remark} {\rm (Alternative description in the finite-order case)}  
Suppose that the conjugacy classes $ \cC_{\ul{\mu}} = \cC_{\mu_1}
\times \ldots \times \cC_{\mu_n}$ each have finite order as in all our
examples, so that 
\[ \exists k_1,\ldots,k_n \in \Z, \quad 
g_i^{k_i} = e, \quad \forall g_i \in \cC_{\mu_i}, i = 1,\ldots,
n .\] 
In this case one can identify the moduli space $M(X,K)$ with the
moduli space of flat bundles on an orbibundle over $X$, see
\cite{ms:pb,km:em,me:lo} for two- and four-dimensional cases.
However, we will avoid using the equivariant description via gauge
theory.  Instead we check explicitly, in specific presentations, the
smoothness of those moduli spaces that enter our constructions.
\end{remark}

\subsection{Moduli spaces of bundles for surfaces}

The key feature of moduli spaces of bundles on compact, oriented
surfaces is their symplectic nature.  Below we review the description
of the symplectic structure in the holonomy description.

\begin{remark} 
Let $X$ be a compact, connected, oriented surface of genus $g$, and
let $\ul{x} = \{x_1,\ldots,x_n \}$ be a marking.
\begin{enumerate}
\item {\rm (Presentation of the fundamental group)} Recall that
  $\eps_j = \pm 1$ depending on whether the orientation of $x_j$
  agrees with the standard orientation of a point.  The fundamental
  group $\pi_1(X \ssm \ul{x})$ has standard presentation
\[ \pi_1(X \ssm \ul{x}) \cong \big\langle \; \alpha_1,\ldots,\alpha_{2g},\gamma_1,\ldots,\gamma_n 
 \ | \ \Pi_{j=1}^g [\alpha_{2j-1},\alpha_{2j}]  \Pi_{j=1}^n \gamma_j^{\eps_j} = 1 \; \big\rangle , \]
where $\gamma_j$ is a loop around $x_j$, oriented corresponding to $\eps_j$.  
\item {\rm (Presentation of the moduli space of flat bundles)} Let
  $\ul{\mu}\in\Alc^n$ be a set of labels for $\ul{x}$.  The moduli
  space of flat $G$-bundles with fixed holonomy can be described in
  terms of a standard presentation of $\pi_1(X\setminus\ul{x})$ by
\begin{align} \label{eq:cs}
 M(X,\ul{x}) 
 & = 
\bigl\{ \varphi\in \Hom(\pi_1(X \ssm \ul{x}), G)) \, \big| \,
\varphi([\gamma_j]) \subset \cC_{\mu_j} \forall j \bigr\}/G  \nonumber \\
&\cong \bigl\{ (\ul{a},\ul{b}) \in
G^{2g} \times \cC_{\ul{\mu}} \  | \ 
\Phi(\ul{a},\ul{b}) = 1  \bigr\} / G ,
\end{align} 
where $G$ acts on $G^{2g} \times \cC_{\ul{\mu}}$ diagonally by conjugation and
\begin{equation} \label{eq:phi} 
\Phi((a_1,\ldots,a_{2g}),(b_1,\ldots,b_n)) = 
\prod_{j=1}^{g} [a_{2j-1},a_{2j}] 
\prod_{j=1}^n b_j^{\eps_j} .
\end{equation} 
\item {\rm (Symplectic form on the moduli space)} For any
  $X,\ul{x},\ul{\mu}$ the space $M(X,\ul{x})$ can be realized as the
  moduli space of flat bundles on the trivial $G$-bundle over $X
  \ssm \ul{x}$ with fixed holonomies around $\ul{x}$ (see
  \cite{ms:pb,me:lo}) and as such has a symplectic form.  The form can
  be described explicitly in the holonomy description \cite{al:mom} as
  follows.  

First we recall the symplectic structure on the moduli space in the
case of a surface without markings.  Let $\theta,\ol{\theta} \in
\Omega^1(G,\g)$ be the left and right-invariant Maurer-Cartan forms,
defined so that
\[ \theta_e(\xi) = \xi, \quad \ol{\theta}_e(\xi) = \xi, \quad \forall \xi \in \g \cong T_e G.\]  
Define a form $\omega_1$ on $G^2$ by
\[ \omega_1 \in \Omega^2(G^2), \ \ \ \ \omega_1 = 
 \langle l^* \theta \wedge r^* \ol{\theta}\rangle/2 + 
 \langle l^* \ol{\theta} \wedge r^* \theta\rangle/2 \]
where $l,r: G^2 \to G$ are the projections on the first and second
factor.  For $g \ge 1$ define two-forms $\omega_g \in
\Omega^2(G^{2g})$ inductively by
\begin{equation} \label{eq:fus}  \omega_g = 
\omega_{g_1} + \omega_{g_2} + \langle \Phi_{g_1}^* \theta \wedge
\Phi_{g_2}^* \ol{\theta} \rangle/2 \end{equation}
where $g = g_1 + g_2$ is any splitting with $g_1,g_2 \ge 1$,
\[\Phi_{g_j}: (G^{2})^{g_j} \to G, \quad 
(a_1,\ldots,a_{2g_j}) \mapsto \prod_{j=1}^{g_j} [a_{2j-1},a_{2j}] \]
is the product of commutators.  We omit pull-backs to the factors of
$G^{2g} \cong G^{2g_1} \times G^{2g_2}$ from the notation to save
space.  A theorem of Alekseev-Malkin-Meinrenken \cite[Theorem
  9.3]{al:mom}, extending earlier work of Weinstein, Jeffrey, and
Karshon, states that the restriction of $\omega_g$ to the identity
level set of $\Phi_g$ descends to the symplectic form on the locus of
irreducible representations in $M(X,\ul{x})$.

Next we define the symplectic structure for a marked surface.  For any
label $\mu \in \Alc$, define a $2$-form $\omega_\mu$ on the conjugacy
class $\cC_\mu$ by
\[
 \omega_{\mu}(v_\xi(g),v_\eta(g))
= \langle \theta(v_\eta(g))+ \ol{\theta}(v_\eta(g)), \xi \rangle , \ \ \ g \in
\cC_\mu, \ \ \ \ \xi,\eta \in \g\]
where $v_\xi,v_\eta \in \Vect(\cC_\mu)$ are the generating vector
fields for $\xi,\eta$.  Define two-forms $\omega_{g,\ul{\mu}} \in
\Omega^2(G^{2g} \times \cC_{\ul{\mu}})$ inductively as follows.
First, set 
\[\omega_{g,\emptyset} = \omega_g, \quad  \omega_{0,\{\mu\}} =
\omega_\mu .\] 
For any splitting $g = g_1 + g_2, \ul{\mu} = \ul{\mu}_1 \cup
\ul{\mu}_2$, where for each $j = 1,2$ either $g_j > 0$ or $\ul{\mu}_j$
is non-empty, set
\begin{equation} \label{eq:fus2}  \omega_{g,\ul{\mu}} = 
\omega_{g_1,\ul{\mu}_1} + \omega_{g_2,\ul{\mu}_2} + \langle
\Phi_{g_1,\ul{\mu}_1}^* \theta \wedge \Phi_{g_2,\ul{\mu}_2}^*
\ol{\theta} \rangle/2 \end{equation}
where 
\[\Phi_{g_j,\ul{\mu}_j}: G^{2g_j} \times \cC_{\ul{\mu}_j} \to G,
\ \ \ \ (\ul{h},(c_\mu)_{\mu\in\ul{\mu}_j}) \mapsto \Phi_{g_j}(\ul{h})
\prod_{\mu \in \ul{\mu}_j} c_\mu .\]
By \cite[page 27]{al:mom}, the restriction of $\omega_{g,\ul{\mu}}$ to
the identity level set of $\Phi_{g,\ul{\mu}}$ descends to the
symplectic form on the locus of irreducible representations in
$M(X,\ul{x})$.
\end{enumerate}
In case $X$ is not connected, we define $M(X,\ul{x})$ to be the
product of moduli spaces for its connected components.
\end{remark} 

In order to construct Floer theory we wish for our moduli spaces to be
smooth.  The next result is a sufficient condition.

\begin{proposition}
\label{prop:smooth}    {\rm (Sufficient conditions for smoothness)} 
Let $(X,\ul{x})$ be a marked surface of genus $g$ with $n$ markings,
signs $\ul{\eps}$ and labels $\ul{\mu}$.  For $\mu \in \Alc$ write
$\mu^{+1} = \mu$ and $\mu^{-1} = *\mu$.  Suppose that each label is
half of a vertex, and the signed sum of labels satisfies a parity
condition:
\begin{equation} 
\{ \mu_i, i \in \{ 1,\ldots, n\}  \}  \subset \{\omega_j/2, j
  \in \{ 1,\ldots,r-1 \} \}, \quad \label{eq:smooth3} 2 \sum_{i=1}^n
  \mu_i^{\eps_i} = \omega_d \ \ \text{mod} \ \Lambda \end{equation}
for some $d$ coprime to $r$, where $\Lambda = \exp^{-1}(1) \subset
\t$.  Then $M(X,\ul{x})$ is a smooth compact symplectic manifold and
we call the labels $\ul{\mu}$ {\em admissible}.
\end{proposition} 

\begin{remark} \label{rem:parityrem} {\rm (Reformulation of the parity condition)} We may write the parity condition 
in the statement of the Proposition as a condition involving
only the indices as follows. 
Under the identification 
\[ \t \cong \Set{ \lambda \in \R^r \, | \, \sum_{i=1}^r
\lambda_i = 0 }, \quad \diag(2\pi i \lambda) \mapsto \t \] 
implicit
in the definition of $\cC_\mu$, the lattice $\Lambda = \exp^{-1}(1)$
consists of the integer vectors with sum zero.  This is the coroot
lattice of $SU(r)$, spanned by the simple coroots $\alpha_k = e_k -
e_{k+1}$, $k = 1,\ldots, r-1$, since $SU(r)$ is simply connected.  It
is strictly smaller than the coweight lattice $\{ \lambda \, | \,
\langle \lambda, \alpha_k \rangle \in \Z \ \forall k \}$, which
contains the vertices $\omega_j$; the quotient of the coweight lattice
by $\Lambda$ is cyclic of order $r$, generated by the class of
$\omega_1$, and
\[ \omega_j - j \omega_1 = (1-j, 1, \ldots, 1, 0, \ldots, 0) \in
\Lambda \]
with $j-1$ entries equal to $1$.  Moreover 
\[ *(\omega_j/2) =
\omega_{r-j}/2 \quad \omega_j + \omega_{r-j} \equiv
r\omega_1 \equiv 0 \  \text{mod} \Lambda ,\] 
the first by \eqref{eq:inv}.
Hence if $\mu_i = \omega_{j_i}/2$
for $i = 1,\ldots, n$, then \eqref{eq:smooth3} is equivalent to
\[ \textstyle \sum_{i=1}^n \eps_i j_i \equiv d \ \ \text{mod} \ r .\]
In particular, removing two markings with the same label and opposite
signs does not affect admissibility.  For $r = 2$ the condition says
that the number of markings is odd.  The Weyl group acts trivially on
the quotient of the coweight lattice by $\Lambda$, so $w \omega_j -
\omega_j \in \Lambda$ for all $w \in W$.
\end{remark} 

\begin{proof}[Proof of Proposition \ref{prop:smooth}]   We first reduce to the case that all signs are positive.  Replacing
each holonomy $b_i$ with $\eps_i = -1$ by $b_i^{-1} \in \cC_{*\mu_i}$
identifies $M(X,\ul{x})$ with the moduli space for the same marked
surface with all signs positive and labels $\mu_i^{\eps_i}$.  These
labels are again halves of vertices, since $*(\omega_j/2) =
\omega_{r-j}/2$, and they satisfy the same parity condition
\eqref{eq:smooth3}.  So we may assume that all signs are positive.

Recall the description of the moduli space of flat bundles as a group-valued symplectic
quotient from \cite{al:mom}.  Let $\cC_{\ul{\mu}}$ denote the
corresponding product of conjugacy classes $\cC_{\mu_j}, j = 1,\ldots,
n$.  The space $M(X,\ul{x})$ can be realized as a symplectic quotient
of the group-valued Hamiltonian $G$-manifold $\widetilde{M}(X,\ul{x})
:= G^{2g} \times \cC_{{\ul{\mu}}}$ with group-valued moment map $\Phi:
\widetilde{M}(X,\ul{x}) \to G$ given by the product \eqref{eq:cs}.  The
identity level set of the moment map is cut out transversally if and
only if all stabilizers are discrete, by \cite[Definition
  2.2,Condition B3]{al:mom}.  Let $R_+$ be the set of positive roots.
For each tuple 
\[ \ul{w} = (w_1,\ldots, w_n) \in W^n \]
and subset $I \subset
R_+$ such that the span of $( \alpha)_{\alpha \in I}$ is not all of
$\t^\dual \cong \t$, define the {\em wall} corresponding to $(\ul{w},I)$ by
\[\Theta_{\ul{w},I} := \exp\Bigl( \textstyle{\sum_{j=1}^n} w_j \mu_j + \on{span }(\alpha)_{\alpha \in I} \Bigr) .\]  
Let $T^{X,\on{sing}}$ denote the singular values of $\Phi$ contained
in $T$.  We claim that the set of singular values is contained in the
union of walls:
\[T^{X,\on{sing}} \subseteq \cup_{\ul{w},I} \Theta_{\ul{w},I} .\]
Any orbit stratum in $\cC_{\mu_j}$ with infinite stabilizer group
contains a $T$-fixed point in its closure, since the same is true for
coadjoint orbits by equivariant formality of Hamiltonian actions
\cite[Appendix C]{gu:moc}.  Similarly, the closure of any orbit
stratum in $G^{2g}$ with infinite stabilizer is equal to $H^{2g}$, for
some subgroup $H$ containing $T$ up to conjugacy.  Now $T^{2g}$ maps
to the identity under the product of commutators $\Phi$.  Putting
everything together, any orbit-type stratum $Y$ in
$\widetilde{M}(X,\ul{x})$ contains a $T$-fixed point $y$ in its
closure.  The image of such a fixed point $y$ under the group-valued
moment map satisfies
\[ \Phi(y) \in \exp \left( \sum_{j=1}^n w_j \mu_j \right) \]
for some $\ul{w} \in W^n$, since the commutators of elements in the maximal
torus vanish.  The tangent space $T_y Y$ is a sum of root spaces.   The assumption that the stabilizer of $Y$ is infinite implies that the
span of the roots appearing in the sum is not all of $\t^\dual$.  

Next we reduce to walls spanned by simple roots.  Let $V \subset
\t$ be a proper subspace spanned by roots, and choose a non-zero $\xi
\in V^\perp$.  Choose $w \in W$ such that $w \xi$ lies in the closed
fundamental Weyl chamber.  The stabilizer of $w \xi$ in $W$ is
generated by the simple reflections fixing $w \xi$, so every root
orthogonal to $w \xi$ lies in the span of the simple roots orthogonal
to $w \xi$.  Since $w \xi \neq 0$, at least one simple root $\alpha_j$
is not orthogonal to $w \xi$, and so
\[ w V \subseteq \on{span}( \alpha_k , k \neq j ) = \omega_j^\perp .\]
Here we use that $\langle \alpha_k, \omega_j \rangle = \delta_{jk}$ in
the standard inner product on $\R^r$.  Note that $\Lambda$ is
$W$-invariant and 
\[ w \sum_{i=1}^n w_i \mu_i = \sum_{i=1}^n (w w_i) \mu_i .\] 
Hence if $1
\in \Theta_{\ul{w},I}$ then 
\[ \sum_{i=1}^n (w w_i) \mu_i \in \Lambda +
\omega_j^\perp \] 
for some $j$.  Pairing with $\omega_j$ and using
$\langle \Lambda, \omega_j \rangle = \Z$, it follows that the moduli
space is an orbifold if for each tuple $\ul{w} = (w_1,\ldots,w_n) \in W^n$ we have
\begin{equation} \label{eq:presmooth}
 \langle 
{\textstyle \sum_{i=1}^n } w_i \mu_i , \omega_j \rangle \notin \Z 
\qquad \forall j=1,\ldots,r-1 .
\end{equation}

We now impose the restriction on the markings given in the statement of the Proposition:  Suppose that each $\mu_i$ is equal to $\frac{1}{2} \omega_{j_i}$
for some $j_i \in \{ 1,\ldots,r-1 \}$, and that \eqref{eq:smooth3} holds,
so that $2 \sum_{i=1}^n \mu_i = \omega_d + \lambda$ for some $\lambda
\in \Lambda$.  By Remark \ref{rem:parityrem}, 
\[ w_i \omega_{j_i} -
\omega_{j_i} \in \Lambda. \]
It follows that 
\[  {\textstyle \sum_{i=1}^n} w_i \mu_i - {\textstyle \sum_{i=1}^n}
\mu_i \in {\textstyle \frac{1}{2}} \Lambda. \] 
Note that the difference need not lie in $\Lambda$: for example 
\[ s_j
(\omega_j/2) = \omega_j/2 - \alpha_j/2 . \]
Since $\langle \Lambda,
\omega_j \rangle = \Z$, we obtain
\[ \langle {\textstyle \sum_{i=1}^n} w_i \mu_i , \omega_j \rangle
\equiv \langle {\textstyle \sum_{i=1}^n} \mu_i , \omega_j \rangle
\equiv {\textstyle \frac{1}{2}} \langle \omega_d, \omega_j \rangle \ \
\text{mod} \ {\textstyle \frac{1}{2}} \Z .\]
A direct computation from the definition of the vertices gives
\[ \langle \omega_d, \omega_j \rangle = \frac{\min(d,j) (r -
  \max(d,j))}{r} \equiv - \frac{dj}{r} \ \ \text{mod} \ \Z .\]
Hence 
\begin{equation} \label{eq:modz2} \langle {\textstyle \sum_{i=1}^n} w_i \mu_i , \omega_j \rangle
\equiv - \frac{dj}{2r} \ \ \text{mod} \ {\textstyle \frac{1}{2}} \Z
.\end{equation}
If the left-hand side of \eqref{eq:modz2} is an integer, then $dj/2r \in \frac{1}{2}\Z$,
that is, $r$ divides $dj$.  This is impossible since $d$ is coprime to
$r$ and $1 \le j \le r-1$.  This proves \eqref{eq:presmooth}.
Each moduli space is the quotient of a
smooth manifold by the free action of a compact group, and so is a manifold.
\end{proof}

\begin{proposition}  {\rm (Admissibility on one end implies admissibility on the other)} 
Suppose that $K \subset Y= [-1,1] \times X$ is a tangle from $X_- = \{
-1 \} \times X$ to $X_+ = \{ + 1 \} \times X$, such that each label is
half of a vertex of the alcove, and the labels for $X_- \cap K$ are
admissible in the sense of \eqref{eq:smooth3}.  Then the same hold for
the induced labels of $X_+ \cap K$.
\end{proposition} 

\begin{proof}  The labelled markings on $X_\pm$ are the same except 
for pairs of markings that have disappeared at maxima of $f|_K$ and
pairs that have appeared at minima of $f|_K$.  The two markings in
such a pair are the endpoints of a single strand, so they carry the
same label with opposite signs.  By Remark \ref{rem:parityrem} removal of the pair does not 
affect the parity condition \eqref{eq:smooth3}.
\end{proof} 

Sufficient conditions for monotonicity of the moduli space are
provided by Meinrenken-Woodward \cite{me:can}.  The condition is a
discrete condition on the set of labels, although in special cases the
set of monotone labels may be larger and non-discrete.

\begin{definition}  \label{def:admissible} {\rm (Monotone labels)} 
A label $\mu\in\Alc$ is {\em monotone} if $\mu$ is the orthogonal
projection of the barycenter $\rho/r$ of the Weyl alcove onto a face
$\sigma$ of the alcove, that is,
\[ \exists \sigma \subset \Alc, \quad \mu =
\on{proj}_\sigma(\rho/r)\] 
where $\on{proj}_\sigma$ denotes orthogonal projection onto $\sigma$.
\end{definition}  

\begin{example} {\rm (Examples of monotone labels)} 
\begin{enumerate}
\item {\rm (Rank two case)} If $G = SU(2)$, then we
  identify $\Alc \cong [0,{1/2}]$ by the map $(\lambda,-\lambda)\mapsto
  \lambda$.  The conjugacy class $\cC_\mu$ consists of matrices with
  eigenvalues $\exp( \pm 2\pi i \mu)$.  We have $\rho = {1/2}$.  If $\mu
  = {1/4}$, $\cC_\mu$ consists of all traceless matrices.  The monotone
  elements are $0$, ${1/2}$, ${1/4} \in \Alc$.
\item {\rm (Rank three case)} If $G = SU(3)$, the alcove $\Alc$ is the
  convex hull of the vectors $(0,0,0), (2/3,-1/3,-1/3)$ and
  $(1/3,1/3,-2/3)$.  We have $\rho = (1,0,-1)$.  See Figure
  \ref{fig:rankthreemonotone} for the monotone labels.
\end{enumerate}
\end{example}

\begin{figure}[h]
\includegraphics[height=1.5in]{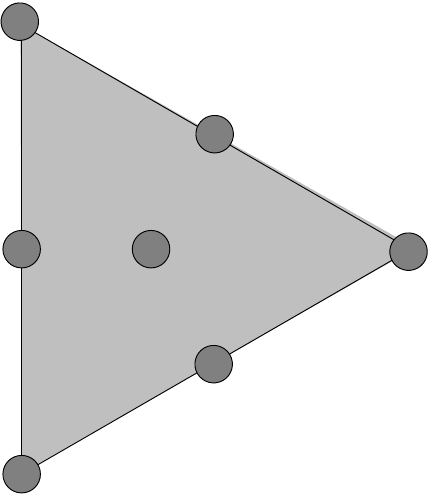}
\caption{Monotone conjugacy classes for $SU(3)$}
\label{fig:rankthreemonotone}
\end{figure}

\begin{theorem} 
 \label{thm:monthm} {\rm (Sufficient conditions for monotonicity, 
\cite[Theorem 4.2]{me:can})} Let $(X,\ul{x})$ be a marked surface with
 labels $\ul{\mu}$.  If each label $\mu_j$ is monotone and
 $M(X,\ul{x})$ is smooth then $M(X,\ul{x})$ is monotone with
 monotonicity constant $\tau^{-1} = 2r$.
\end{theorem}  

Finally we note that the natural action of the diffeomorphism group
acts by symplectomorphisms:

\begin{definition} \label{def:mcg}  {\rm (Marking-preserving mapping class group)}  
Let $\on{Diff}_+(X,\ul{x})$ be the subgroup of orientation-preserving
diffeomorphisms $\varphi \in \on{Diff}_+(X)$ that preserve the marked
points, orientations, and labels:
\begin{equation} \label{eq:diffplus} 
 \on{Diff}_+(X,\ul{x}) = \left\{ \varphi \in \on{Diff}_+(X) \,\big|\,
 \varphi(\ul{x} ) = \ul{x}, \ \varphi^*\ul{\eps} =\ul{\eps},
 \ \varphi^* \ul{\mu} = \ul{\mu} \right\} .\end{equation}
Here we denote the maps 
\[ \ul{\eps}: \ul{x} \to \{\pm 1\}, x_i \mapsto \eps_i, \quad
\ul{\mu}: \ul{x} \to \Alc, x_i \mapsto \mu_i .\]
So the conditions in \eqref{eq:diffplus} are 
\[ ( \varphi(x_i) = x_j) \implies ( \eps_i = \eps_j \ \text{and} \  \mu_i
= \mu_j) , \quad \forall i,j  = 1,\ldots, n.\]  
Let $\Map_+(X,\ul{x})$ be the quotient of $\on{Diff}_+(X,\ul{x})$ by
isotopy.
\end{definition} 

\begin{remark} \label{rem:sbraid} {\rm (Spherical braid group action)} 
The action of $\Map_+(X,\ul{x})$ on $\pi_1(X \ssm \ul{x})$ induces an
action on $M(X,\ul{x})$ by symplectomorphisms on the smooth stratum.
See for example \cite[Section 9.4]{al:mom} for a proof from the
holonomy point of view.  In particular, if $X$ is a sphere and all
labels are equal, $\ul{\mu}=(\mu,\ldots,\mu)$, then the spherical
braid group $\Map_+(X,\ul{x})$ acts on $M(X,\ul{x})$.  Explicitly let
$\sigma_{i(i+1)} \in \Map_+(X,\ul{x})$ is the half-twist of $x_i$ and
$x_{i+1}$.  For suitable choice of presentation of $\pi_1(X \ssm
\ul{x})$ we have
\begin{multline}  \sigma_{i(i+1)}: M(X,\ul{x}) \to M(X,\ul{x}), \\ \quad
     [b_1,\ldots,b_n] \mapsto [b_1,\ldots,b_{i-1},b_{i+1},b_{i+1}^{-1}
       b_i b_{i+1}, b_{i+2}, \ldots,b_n ]. \end{multline}
\end{remark} 

\subsection{Moduli spaces for tangles}

Given a tangle we construct Lagrangian correspondences as follows.

\begin{definition}  Let $(Y,K,\phi)$ be a labelled tangle
from $(X_-,\ul{x}_-)$ to $(X_+,\ul{x}_+)$.
\begin{enumerate} 
\item {\rm (Restriction to the boundary)} Let $K_1,\ldots,K_p$ be the
  connected components of $K$ and fix labels $\ul{\nu} =
  (\nu_1,\ldots,\nu_p) \in \Alc^p$ for $K$.  In Section \ref{sec:moduli of flat G bundles} we defined the moduli space $ M(Y,K)$ of flat
  $G$-bundles on $Y\ssm K$ with holonomy around $K_j$ in
  $\cC_{\nu_j}$.  On the boundary the labels $\ul{\nu}$ induce labels
  $\ul{\mu}_\pm \in \Alc^{n_\pm}$ defined by $\nu_j$ for $\partial
  K_j$.  Restriction to the boundary and pull-back under $\phi$ define
  a map
\begin{equation} \label{eq:restr} M(Y,K) \to M(X_-,\ul{x}_-)^-
\times M(X_+,\ul{x}_+) .\end{equation} 
More precisely, the inclusion of the boundary and a choice of paths
between base points induces a map of fundamental groups 
\[ \pi_1(X_- \ssm \ul{x}_-) \sqcup \pi_1(X_+ \ssm \ul{x}_+) \to
\pi_1(Y \ssm K) \]
that is well-defined up to conjugacy.  The dual maps the
representation variety of the bordism to the product of representation
varieties of its boundary components, independent of the choice of
path.
\item 
{\rm (Correspondences for tangles)} For any labelled
tangle $(Y,K,\phi)$ we denote the image of \eqref{eq:restr} by
\[L(Y,K,\phi)\subset M(X_-,\ul{x}_-)^-
\times M(X_+,\ul{x}_+) .\] 
\end{enumerate} 
\end{definition}  

\begin{lemma} \label{lem:compcor} {\rm (Correspondences for compositions)}
Let $(X_i,\ul{x}_i)$ be marked surfaces for $i=0,1,2$, and let
$(Y_{01},K_{01},\phi_{01})$ resp.\ $(Y_{12},K_{12},\phi_{12})$ be
tangle from $(X_0,\ul{x}_0)$ to $(X_1,\ul{x}_1)$ resp.\ from
$(X_1,\ul{x}_1)$ to $(X_2,\ul{x}_2)$.  Let $\ul{\nu}_{01}$ and
$\ul{\nu}_{12}$ be labels for the bordisms with tangles such that they
induce the same label $\ul{\mu}_1$ for $(X_1,\ul{x}_1)$.  Gluing
provides a bordism with tangle $(Y_{01}\circ Y_{12},K_{01}\circ
K_{12})$ from $(X_0,\ul{x}_0)$ to $(X_2,\ul{x}_2)$ with labels
$\ul{\nu}_{01}\circ \ul{\nu}_{12}$. The induced labels $\ul{\mu}_0$
for $\ul{x}_0$ and $\ul{\mu}_2$ for $\ul{x}_2$ are the same as the
ones induced from $\ul{\nu}_{01}$ and $\ul{\nu}_{12}$, and we have the
equality of subsets of $M(X_0,\ul{x}_0,\ul{\mu}_0) \times
M(X_2,\ul{x}_2,\ul{\mu}_2)$
\[ L( (Y_{01},K_{01},\phi_{01}) \circ (Y_{12}, K_{12}, \phi_{12})) =
L(Y_{01},K_{01},\phi_{01}) \circ L(Y_{12},K_{12},\phi_{12}) .
\]
This is an equality of sets; whether the geometric composition on the
right is transverse and embedded is a separate question, addressed
below.
\end{lemma}

\begin{proof}   By the Seifert-van Kampen theorem we have an isomorphism (using a base
point on $X_1$)
\[
\pi_1(Y_{01}\circ Y_{12}\setminus K_{01}\circ K_{12})
\cong
\pi_1(Y_{01}\setminus K_{01})
\star_{\pi_1(X_1\setminus\ul{x}_1)} \pi_1(Y_{12}\setminus K_{12}) .
\] In particular, any representation $\varphi: \pi_1(Y_{01}\circ
Y_{12}\setminus K_{01}\circ K_{12}) \to G$ induces representations on
both sides $\varphi_{j(j+1)} : {\pi_1(Y_{j(j+1)} \setminus
  K_{j(j+1)})} \to G, j = 0,1$, whose restriction to
$X_1\setminus\ul{x}_1$ agree.  Conversely, any pair of representations
on the two sides, whose restrictions to $X_1\setminus\ul{x}_1$ are
conjugate, we can conjugate one of the sides so that the restrictions
agree.  Patching induces a representation on the glued space.
\end{proof}

\begin{lemma}
\label{lem:Lex} {\rm (Correspondences for elementary tangles)} 
\noindent
\ben
\item \label{cbord} {\rm (Cylindrical bordisms)} Suppose that
  $(Y,K,\phi)$ admits a surjective Morse function $f: Y \to [-1,1]$
  with no critical points on $Y$ or $K$.  The map $f: Y \to [-1,1]$ is
  a fiber bundle containing $K$ that admits a simultaneous
  trivialization
\[{\cT}:([-1,1] \times X_-,[-1,1] \times \ul{x}_-)\to (Y,K) ,\]
where 
\[{\cT}|_{f^{-1}(-1)}={\rm Id}_{X_-}, \quad \psi:= \phi \circ
{\cT}|_{f^{-1}(1)}:(X_-,\ul{x}_-)\to (X_+,\ul{x}_+)\]
are isomorphisms of marked surfaces.  The correspondence associated to
$(Y,K,\phi)$ is then the graph:
\[L(Y,K,\phi) = \on{graph}((\psi^{-1})^*) \subset M(X_-,\ul{x}_-)^-
\times M(X_+,\ul{x}_+) .\]
\item \label{etangle} {\rm (Elementary tangles)} Let $X$ be a compact
  oriented surface, $Y = [-1,1] \times X$ the trivial bordism and
  suppose that $Y$ admits a cylindrical Morse function such that $K$
  contains a single critical point that is a maximum.  Thus $K$  consists
  of $n-2$ strands meeting both the incoming and outgoing boundary,
  and one strand that connects two incoming markings $x_i,x_j \in
  \ul{x}_-$, as in Figure \ref{fig:capfig}; here $n$ is the number of
  incoming markings, so that $\ul{x}_- = \ul{x}_+ \cup \{ x_i, x_j
  \}$.  The map $L(Y,K,\phi) \to M(X,\ul{x}_-)$ induced by pullback is
  a coisotropic embedding, and $ L(Y,K,\phi) \to M(X,\ul{x}_+)$ is a
  fiber bundle with fiber $\cC_{\mu_i} \cong \cC_{\mu_{j}}$.  If
  instead the critical point is a minimum, so that the strand connects
  two outgoing markings, the same holds with the roles of $\ul{x}_-$
  and $\ul{x}_+$ exchanged.
\item \label{qcups} {\rm (Quilted cups and caps)} More generally,
  suppose that $X$ is a compact oriented surface and $Y = [-1,1]
  \times X$ a product bordism.  Suppose that $Y$ admits a cylindrical
  Morse function such that all critical points in $K$ are maxima.
  Then the pull-back map $L(Y,K,\phi) \to M(X,\ul{x}_-)$ is a
  coisotropic embedding and $ L(Y,K,\phi) \to M(X,\ul{x}_+)$ is a
  fiber bundle with fiber a product of conjugacy classes.  If all
  critical points in $K$ are minima, the same holds with the roles of
  $\ul{x}_-$ and $\ul{x}_+$ exchanged.  \een
\noindent Furthermore, in each case $L(Y,K,\phi)$ is a smooth Lagrangian
correspondence in $M(X_-,\ul{x}_-)^- \times
M(X_+,\ul{x}_+) .$
\begin{figure}[ht]
\setlength{\unitlength}{0.00047489in}
\begingroup\makeatletter\ifx\SetFigFont\undefined
\def\x#1#2#3#4#5#6#7\relax{\def\x{#1#2#3#4#5#6}}%
\expandafter\x\fmtname xxxxxx\relax \def\y{splain}%
\ifx\x\y   
\gdef\SetFigFont#1#2#3{%
  \ifnum #1<17\tiny\else \ifnum #1<20\small\else
  \ifnum #1<24\normalsize\else \ifnum #1<29\large\else
  \ifnum #1<34\Large\else \ifnum #1<41\LARGE\else
     \huge\fi\fi\fi\fi\fi\fi
  \csname #3\endcsname}%
\else
\gdef\SetFigFont#1#2#3{\begingroup
  \count@#1\relax \ifnum 25<\count@\count@25\fi
  \def\x{\endgroup\@setsize\SetFigFont{#2pt}}%
  \expandafter\x
    \csname \romannumeral\the\count@ pt\expandafter\endcsname
    \csname @\romannumeral\the\count@ pt\endcsname
  \csname #3\endcsname}%
\fi
\fi\endgroup
{\renewcommand{\dashlinestretch}{30}
\begin{picture}(3264,1839)(0,-10)
\dashline{60.000}(57,1812)(3252,1812)
\dashline{60.000}(12,12)(3207,12)
\path(507,1812)(507,12)
\path(957,1812)(957,12)
\path(1407,12)(1407,15)(1407,22)
	(1408,33)(1408,51)(1409,75)
	(1410,105)(1412,140)(1414,178)
	(1416,218)(1418,259)(1421,300)
	(1424,339)(1428,377)(1432,413)
	(1436,446)(1441,477)(1446,506)
	(1452,533)(1458,558)(1465,583)
	(1472,606)(1480,628)(1490,650)
	(1499,671)(1510,692)(1522,713)
	(1534,734)(1547,754)(1562,774)
	(1577,794)(1593,814)(1610,833)
	(1627,852)(1645,869)(1664,886)
	(1683,902)(1702,916)(1721,930)
	(1741,942)(1760,952)(1779,962)
	(1798,969)(1816,976)(1834,981)
	(1852,984)(1870,986)(1887,987)
	(1904,986)(1921,984)(1938,981)
	(1956,976)(1973,969)(1990,962)
	(2008,952)(2025,942)(2042,930)
	(2060,916)(2077,902)(2093,886)
	(2109,869)(2125,852)(2140,833)
	(2154,814)(2167,794)(2180,774)
	(2192,754)(2203,734)(2213,713)
	(2223,692)(2231,671)(2240,650)
	(2247,628)(2254,606)(2260,583)
	(2265,558)(2271,533)(2275,506)
	(2279,477)(2283,446)(2287,413)
	(2290,377)(2293,339)(2295,300)
	(2298,259)(2300,218)(2302,178)
	(2303,140)(2304,105)(2305,75)
	(2306,51)(2307,33)(2307,22)
	(2307,15)(2307,12)
\path(2757,1812)(2757,12)
\put(3290,-30){\makebox(0,0)[l]{$X_-$}}
\put(3290,1770){\makebox(0,0)[l]{$X_+$}}
\end{picture}
}
\caption{A cap: a strand joining two incoming markings}
\label{fig:capfig}
\end{figure}
\end{lemma}

\begin{proof}  
\eqref{cbord} First we construct a simultaneous trivialization.
Suppose that $(Y,K,\phi)$ admits a Morse function $f: Y \to [-1,1]$
with no critical points on $Y$ or $K$.  Choose a vector field $v$ on
$K$ with flow $\psi_t$ so that that $\ddt f \circ \psi_t = 1$; for
example, a normalized gradient vector field.  The vector field $v$
extends to a vector field on $Y$ with $L_v f = 1$: extend $v$ to a
neighborhood of $K$ with $L_v f = 1$ using a tubular neighborhood, and
interpolate, using a bump function $\rho$ equal to one near $K$, with
the normalized gradient vector field $|\nabla f|^{-2} \nabla f$ of
$f$, which is defined since $f$ has no critical points on $Y$, setting
\[ v' = \rho v + (1 - \rho) |\nabla f|^{-2} \nabla f \in \Vect(Y)
.\]
Replacing $v$ by $v'$, we have $L_vf = 1$, in case the flow
${\cT}(x_-,-1 + t)=\psi_t(x_-)$ gives the claimed trivialization.

Next we identify the correspondence associated to the bordism as a
graph of a symplectomorphism.  With notation from the previous
paragraph any representation $\rho'$ of $\pi_1(Y\setminus K)$ is the
pullback $({\cT}^{-1})^*\rho$ of a representation $\rho$ of the
trivial cylinder over $X_-\setminus\ul{x}_-$.  Since the restrictions
of $\rho$ to the two boundary components are conjugate, the boundary
restriction of $({\cT}^{-1})^*\rho$ will be the graph of pullback
under $\psi^{-1}=( \phi \circ {\cT})^{-1}|_{X_+}:X_+\to X_-$.
Moreover, $\psi^{-1}$ preserves the orientations and labels and hence
induces a symplectomorphism of moduli spaces $M(X_-,\ul{x}_-) \to
M(X_+,\ul{x}_+)$.  So $L(Y,K,\phi)$ is the graph of a
symplectomorphism and hence a smooth Lagrangian correspondence.

Next we consider the case \eqref{etangle} of an elementary tangle.
Choose a system of generators of the fundamental groups of the
punctured surfaces such that in the holonomy description
\begin{align*}
L(Y,K,\phi) &= \bigl\{ \bigl( [a_1,\ldots,a_{2g},h_1,\ldots, h_{n}],
\\ &\qquad\quad [a_1,\ldots,a_{2g},h_1,\ldots,\hat{h}_i,\ldots,
  \hat{h}_j, \ldots, h_{n}] \bigr) \,\big|\, h_i = h_j \bigr\} .
\end{align*}
\noindent Such a system of generators can be found as follows.  Let
$k_0 \in K$ denote the unique critical point, by assumption a maximum.
Choose
\[b_+:=f(X_+)> c_+ > f(k_0) > c_-> f(X_-)=:b_-\]
such that the bordisms $f^{-1}([b_-,c_-])$ and $f^{-1}([b_+,c_+])$ are
cylindrical.  Any choice of generators for the fundamental groups of
$f^{-1}(c_\pm)$ induces generators for the fundamental groups of
$f^{-1}(b_\pm)$.  Consider the bordism $f^{-1}([c_-,c_+])$.  For
$c_\pm$ sufficiently close to $f(k_0)$, the level set $f^{-1}(c_+)$ is
obtained from $f^{-1}(c_-)$ by replacing a twice punctured disk $D_-$
by a disk $D_+$ without punctures: 
\[ f^{-1}(c_+) = f^{-1}(c_-) - D_- \cup D_+ . \]
The two punctures labelled $i,j$ in $D_-$ are connected in
$f^{-1}([c_-,c_+])$ by a small cap.  Choose a system of
representatives for generators for $\pi_1(f^{-1}(c_-))$ that except
for the generators $\gamma_i,\gamma_j$ around the $i$-th and $j$-th
markings do not meet $D_-$.  Removing the two generators
$\gamma_i,\gamma_j$ produces a system of generators for $f^{-1}(c_+)$.
Since the $i$-th and $j$-th markings are connected by a cap, the
holonomies around the punctures $x_i,x_{j}\in X_-$ are inverses, up to
conjugacy.  By \eqref{eq:inv} we have $\mu_i = *\mu_{j}$.  We may choose
the generators so that $x_i,x_j$ are consecutive; since they carry
opposite orientations, $\eps_j = -\eps_i$, and the factor $h_i^{\eps_i}
h_j^{\eps_j}$ in \eqref{eq:phi} equals $1$ whenever $h_i = h_j$.

The presentation above implies that the Lagrangian correspondence is
smooth.  Indeed, on the locus $h_i = h_j$ the map \eqref{eq:phi} for
$\ul{x}_-$ agrees with the map \eqref{eq:phi} for $\ul{x}_+$ applied to
the remaining holonomies.  Since $M(X_+,\ul{x}_+)$ is a smooth free
quotient, the level set
\[\prod_{l=1}^{g} [a_{2l-1},a_{2l}] \prod_{k \in \{1,\ldots,n\} \ssm \{i,j\}}
h_k^{\eps_k} = 1, \quad h_i = h_j \]
is transversally cut out, and is isomorphic to the product of the
identity level set for $\ul{x}_+$ with $\cC_{\mu_i}$.  Thus
$L(Y,K,\phi)$ is the free quotient of a smooth manifold by the
projectivized unitary group, and so smooth.  The map
$\pi_+:L(Y,K,\phi) \to M(X_+,\ul{x}_+)$ forgets the holonomies $h_i =
h_j$, and so is a smooth fibration with fiber $\cC_{\mu_i}$.  The map
$\pi_-: L(Y,K,\phi) \to M(X_-,\ul{x}_-)$ is injective, since the
holonomies for $\ul{x}_+$ are among those for $\ul{x}_-$, and similarly
an immersion; hence it is an embedding.  Using the identification
$\cC_{\mu_j} \to \cC_{*\mu_i}, g \mapsto g^{-1}$, the pair
$(h_i,h_j)$ ranges over the antidiagonal
\[\Delta_i:=\{(h,h^{-1}) \,|\, h \in \cC_{\mu_i}\} \subseteq
\cC_{\mu_i} \times \cC_{*\mu_i},\]
which is the level set $\Phi_{0,\{\mu_i,*\mu_i\}}^{-1}(1)$ and a
single $G$-orbit.

It remains to show that $L(Y,K,\phi)$ is Lagrangian in
$M(X_-,\ul{x}_-)^- \times M(X_+,\ul{x}_+)$.  We show that the
symplectic form on the larger space $M(X_-,\ul{x}_-)$ restricts on
$\pi_-(L(Y,K,\phi))$ to the pullback of the symplectic form on the
smaller space $M(X_+,\ul{x}_+)$ under $\pi_+ \circ \pi_-^{-1}$.  The
symplectic form on $M(X_-,\ul{x}_-)$ is given by reduction from the
2-form \eqref{eq:fus2} on $G^{2g} \times \cC_{\ul{\mu}_-} \cong G^{2g}
\times \cC_{\ul{\mu}_+} \times \cC_{\mu_i} \times \cC_{*\mu_i}$.  In
the latter splitting the 2-form is
\begin{equation}
\omega_{g,\ul{\mu}_+} + \omega_{0, \{ \mu_i, *\mu_i \}} + \langle
\Phi_{g,\ul{\mu}_+}^* \theta \wedge \Phi_{0, \{ \mu_i, *\mu_i \}}^*
\ol{\theta} \rangle/2 .
\end{equation}
The third term vanishes on $G^{2g} \times\cC_{\ul{\mu}_+} \times
\Delta_i$, since $\Phi_{0,\{\mu_i,*\mu_i\}}$ is constant there.  The
second term vanishes on $\Delta_i$, since $\Delta_i$ is an identity
level set whose quotient by $G$ is a point, so the restriction of
$\omega_{0,\{\mu_i,*\mu_i\}}$ to it is basic and pulled back from a
point.  So on $G^{2g} \times \cC_{\ul{\mu}_+} \times \Delta_i$ the
2-form is the pullback of $\omega_{g,\ul{\mu}_+}$, and the same holds
after taking quotients.  Hence on $L(Y,K,\phi)$
\[ - \pi_-^* \omega_{M(X_-,\ul{x}_-)} + \pi_+^*
\omega_{M(X_+,\ul{x}_+)} = 0 .\]
For the dimension count, note that $\dim M(X_-,\ul{x}_-) = \dim
M(X_+,\ul{x}_+) + 2 \dim \cC_{\mu_i}$, since $\dim \cC_{*\mu_i} = \dim
\cC_{\mu_i}$, while $\dim L(Y,K,\phi) = \dim M(X_+,\ul{x}_+) + \dim
\cC_{\mu_i}$ by the fibration $\pi_+$.  So $L(Y,K,\phi)$ has half the
dimension of $M(X_-,\ul{x}_-) \times M(X_+,\ul{x}_+)$, and is
therefore Lagrangian.  In particular the fibers of $\pi_+$ are
isotropic, and $\pi_-$ is a coisotropic embedding.

The case of a minimum follows by reversing the direction of $Y$, which
exchanges the roles of $\ul{x}_-$ and $\ul{x}_+$.  Case
\eqref{qcups} follows by the same argument applied to one pair of
holonomies $h_i = h_j$ for each critical point.
\end{proof} 

\begin{corollary}  \label{cor:LYsmooth} {\rm (Lagrangian correspondences for elementary tangles)} 
If $(Y,K,\phi)$ is an elementary cylindrical tangle as in Definition
\ref{def:handletangle} from $(X_-,\ul{x}_-)$ to $(X_+,\ul{x}_+)$ and the
labels $\ul{\nu}$ for the components of $K$ are such that the moduli
spaces $M(X_\pm,\ul{x}_\pm)$ are smooth manifolds (see Proposition
\ref{prop:smooth}) then the moduli space $L(Y,K,\phi)$ is a smooth
Lagrangian correspondence from $M(X_-, \ul{x}_-)$ to $M(X_+,
\ul{x}_+)$.
\end{corollary}  

\begin{proof}
Since the tangle is cylindrical, $f$ has no critical points on $Y$, so
there are two cases to consider, depending on whether a critical point
occurs in the tangle, so that $ \# \Crit(f|_K) = 1 $, or not at all,
so that $ \# \Crit(f|_K) = 0 .$ In the first case, the critical point
must be a maximum or minimum, and the claim is
Lemma~\ref{lem:Lex}~\eqref{etangle}.  The second case follows from
Lemma~\ref{lem:Lex}~\eqref{cbord}.
\end{proof}

\begin{remark} 
 Relative spin structures will be needed later to provide orientations
 on moduli spaces of holomorphic quilts.  Recall from \cite{fooo},
 \cite{orient} that a relative spin structure with background class $b
 \in H^2(M;\Z_2)$ is a trivialization of $TL \oplus E_b|L$ over the
 $2$-skeleton of $L$ with respect to some triangulation, where $E_b
 \to M$ is an orientable  bundle with $w_2(E_b) =b$.  Two
 relative spin structures are {\em equivalent mod $w_2(TM)$} if their
 background classes are related by $b - b' = w_2(TM)$, that is, by
 adding Stiefel-Whitney classes $w_2(TM) \in H^2(M,\Z_2)$, and the
 corresponding trivializations of $L \oplus E_b$ are related by adding
 the canonical trivialization of $TM|L$ on the $2$-skeleton.
\end{remark}

\begin{remark}  \label{rem:relspin} 
{\rm (Background classes for moduli of bundles)} Suppose that
$\ul{x}_i$ consists of $n_i^\pm$ markings with positive resp. negative
orientation, and $\mu_i^\pm$ are the labels of the points with
positive resp. negative orientation.  We take as background classes
for $M(X_i,\ul{x}_i)$ the Stiefel-Whitney classes for conjugacy
classes associated to the positively or negatively oriented markings
\begin{equation} \label{eq:bpm}
 b_\pm(X_i,\ul{x}_i) := w_2 (
T (\Pi_{j=1}^{n_i^\pm} \cC_{\mu_i^\pm})
 \qu G ) \in H^2( M(X_i,\ul{x}_i),\Z_2) .\end{equation}
Here $ T (\Pi_{j=1}^{n_i^\pm} \cC_{\mu_i^\pm}) \qu G $ denotes the
bundle obtained by pulling back $T (\Pi_{j=1}^{n_i^\pm}
\cC_{\mu_i^\pm})$ to $G^{2g} \times \cC_{\ul{\mu}}$, restricting to
$\Phinv(e)$, and quotienting by the action of $G$.  The classes $
b_\pm(X_i,\ul{x}_i)$ are equivalent modulo $w_2 ( TM(X_i,\ul{x}_i))$.
Indeed, the quotient construction gives an isomorphism
\[ TM(X_i,\ul{x}_i) \oplus \ad \oplus \ad \cong (TG^{2g} \oplus T
\cC_{\ul{\mu}}) \qu G, \quad TG^{2g} \qu G \cong \ad^{\oplus 2g} \]
where $\ad$ denotes the bundle associated to the adjoint
representation.  Since $G$ is connected, $w_1(\ad) = 0$, and so any
even number of copies of $\ad$ has vanishing second Stiefel-Whitney
class.  Hence $w_2(TM(X_i,\ul{x}_i)) = b_+(X_i,\ul{x}_i) +
b_-(X_i,\ul{x}_i)$.
\end{remark}

\begin{lemma} \label{lem:Kbrane} {\rm (Relative spin structures)} 
Let $(Y,K,\phi)$ be an oriented elementary tangle from
$(X_-,\ul{x}_-)$ to $(X_+, \ul{x}_+)$, where $X_- = X_+ = X$ and the
tangle is cylindrical as in Definition \ref{def:tancat}, such that
$\ul{x}_+$ has at least as many elements as $\ul{x}_-$.  Let
$\ul{\nu}$ be a labelling of $K$ such that the moduli spaces
$M(X_\pm,\ul{x}_\pm)$ are smooth manifolds.  Then $L(Y,K,\phi)$ is
simply-connected, so oriented.  There is a unique relative spin
structure on $L(Y,K,\phi)$ with background class
$(b_\pm(X_-,\ul{x}_-),b_\mp(X_+,\ul{x}_+))$ of \eqref{eq:bpm}, for
either choice of sign.  These
relative spin structures are compatible under composition in the sense
that the relative spin structures on $L(Y_{01},K_{01}) \circ
L(Y_{12},K_{12})$ and $L(Y_{02},K_{02})$ agree up to shifts by $w_2 (
TM(X_i,\ul{x}_i))$.
\end{lemma} 

\begin{proof}  It suffices to consider the case of a single critical point, since
the case of no critical points is trivial.  Suppose first that $K$
contains a critical point of index $0$, with a strand connecting the
markings $x_i,x_j \in X_+$. Suppose that the orientation of $x_j$
resp. $x_i$ is the same resp. opposite of the standard orientation of
a point.  By Lemma \ref{lem:Lex}, $L(Y,K,\phi)$ is diffeomorphic to a
$\cC_{\mu_j}$-bundle over $ M(X_-,\ul{x}_-).$ The base
$M(X_-,\ul{x}_-)$ is simply-connected because $M(X_-,\ul{x}_-)$ is
homeomorphic to the moduli space of parabolic bundles \cite{ms:pb} and
the moduli space of parabolic bundles is simply-connected
\cite{ni:co}.  The conjugacy classes of $G$ are simply-connected,
since they are partial flag varieties.  Hence $L(Y,K,\phi)$ is simply
connected as well.

An orientation on the Lagrangian correspondence is defined as follows.
Since the base $M(X_-,\ul{x}_-)$ is simply-connected and the structure
group of the bundle $SU(r)$ is connected, an orientation $L(Y,K,\phi)$
is induced by the symplectic orientation on the base $M(X_-,\ul{x}_-)$
and the orientation on the fiber $\cC_{\mu_j}$.

Relative spin structures are defined as follows.  By \cite[Remark 3.1.10]{orient}, an
oriented submanifold $L \subset N$ admits a relative spin structure
with background class $b \in H^2(N,\Z_2)$ if and only if $w_2(TL)$
equals the restriction of $b$ to $L$.  By \cite[Proposition
3.3.1(b)]{orient}, the isomorphism classes of such structures
correspond to spin structures on $TL \oplus E_b|_L$, and so, if
non-empty, form a torsor over $H^1(L;\Z_2)$.  Since $L(Y,K,\phi)$ is
simply connected, $H^1(L(Y,K,\phi);\Z_2) = 0$.  Writing $\pi_\pm :
L(Y,K,\phi) \to M(X_\pm,\ul{x}_\pm)$ for the projections, existence
and uniqueness of the relative spin structure therefore both follow
from the identity
\begin{equation} \label{eq:w2L}
w_2(TL(Y,K,\phi)) = \pi_-^* b_\pm(X_-,\ul{x}_-) + \pi_+^*
b_\mp(X_+,\ul{x}_+) \in H^2(L(Y,K,\phi);\Z_2)
\end{equation}
for either choice of sign.  (Reversing the symplectic form on the
first factor does not affect classes mod $2$.)

The tangent bundle of the Lagrangian is computed from its holonomy
description.  Write $\ul{x}_+ = \ul{x}_- \cup \{ x_i, x_j \}$, where
$x_j$ resp. $x_i$ is positively resp. negatively oriented, and let
$\cC := \cC_{\mu_j} = \cC_{\mu_i}$ be the conjugacy class of the label
$\nu_k$ of the strand joining them.  Let
\[ Z_- := \Phi_{\ul{x}_-}^{-1}(e) \subset G^{2g} \times
\cC_{\ul{\mu}_-} \]
be the identity level set for $\ul{x}_-$.  As in the proof of Lemma
\ref{lem:Lex}, with the roles of $\ul{x}_\pm$ exchanged, the locus
$h_i = h_j$ in the identity level set for $\ul{x}_+$ is identified with
$Z_- \times \cC$, and
\[ L(Y,K,\phi) \cong (Z_- \times \cC)/G , \]
where $G$ acts diagonally by conjugation.  Here $\pi_-$ is induced by
the projection onto $Z_-$, and $\pi_+$ by the map $(\ul{a},\ul{h}_-,c)
\mapsto (\ul{a},\ul{h}_-,c,c)$ that places $c$ in both of the factors
for $x_i$ and $x_j$.  By admissibility, $G/Z(G)$ acts freely on $Z_-$,
hence on $Z_- \times \cC$.  For a $G$-equivariant vector bundle $E$ over
$Z_- \times \cC$, denote by $E \qu G$ the induced bundle over
$L(Y,K,\phi)$, as in Remark \ref{rem:relspin}.  Note the following isomorphisms: We have 
\[ T(Z_-
  \times \cC) \qu G \cong TL(Y,K,\phi) \oplus \ad \] 
  (with notation from Remark  \ref{rem:relspin} )
from the description of the quotient and the coisotropic property.   Also, 
\[ TZ_- \oplus \g \cong T(G^{2g}
  \times \cC_{\ul{\mu}_-})|_{Z_-} \] equivariantly, 
  since
$\Phi_{\ul{x}_-}$ is a submersion along $Z_-$ and is the normal bundle of $Z_-$ is the trivial bundle with fiber $\g$.  Finally, 
\[ TG^{2g} \qu G \cong \ad^{\oplus 2g} \] 
as in Remark
  \ref{rem:relspin}. Combining these,
\begin{equation} \label{eq:TLsplit}
TL(Y,K,\phi) \oplus \ad^{\oplus 2} \cong \ad^{\oplus 2g} \oplus
\pi_-^*(T\cC_{\ul{\mu}_-} \qu G) \oplus T\cC \qu G .
\end{equation}
Each conjugacy class is a partial flag variety with a $G$-invariant
complex structure, so the bundles $T\cC_{\nu} \qu G$ are complex and in
particular oriented.  The bundle $\ad$ is oriented since $G$ is
connected, and 
\[ w_2(\ad \oplus \ad) = w_1(\ad)^2 = 0 . \] 
Since $w_2$ is
additive on oriented bundles, splitting $\cC_{\ul{\mu}_-}$ into the
factors for positively and negatively oriented markings and using
\eqref{eq:bpm} gives
\[ w_2(TL(Y,K,\phi)) = \pi_-^* \bigl( b_+(X_-,\ul{x}_-) +
b_-(X_-,\ul{x}_-) \bigr) + w_2(T\cC \qu G) . \]

The background classes restrict to $L(Y,K,\phi)$ in the same way.
Under $\pi_+$ the holonomies at the markings of $\ul{x}_-$ are
unchanged, and the holonomies at $x_i$ and $x_j$ are both equal to
$c$.  So the factors $T\cC_{\mu_i}$ and $T\cC_{\mu_j}$ of
\eqref{eq:bpm} both pull back to $T\cC \qu G$.  Since $x_j$ is
positively and $x_i$ negatively oriented,
\[ \pi_+^* b_\pm(X_+,\ul{x}_+) = \pi_-^* b_\pm(X_-,\ul{x}_-) + w_2(T\cC
\qu G) . \]
Adding $\pi_-^* b_\mp(X_-,\ul{x}_-)$ to both sides gives
\eqref{eq:w2L}.  In words: The fiber direction $\cC$ of $L(Y,K,\phi)$
contributes to $w_2(TL(Y,K,\phi))$ exactly once, and this contribution
is accounted for by whichever of $x_i,x_j$ enters the background class
on $M(X_+,\ul{x}_+)$.

If $K$ has no critical points, then $L(Y,K,\phi)$ is the graph of the
symplectomorphism $\psi: M(X_-,\ul{x}_-) \to M(X_+,\ul{x}_+)$ induced
by a diffeomorphism of $X$ mapping $\ul{x}_-$ to $\ul{x}_+$ and
preserving labels and orientations.  Then $\psi^*
b_\pm(X_+,\ul{x}_+) = b_\pm(X_-,\ul{x}_-)$.  It follows that the restriction of the
background class to the graph is 
\[ b_+(X_-,\ul{x}_-) +
b_-(X_-,\ul{x}_-) = w_2(TM(X_-,\ul{x}_-)) \] 
by Remark
\ref{rem:relspin}, which is $w_2$ of the tangent bundle of the graph.

Compatibility under composition follows from uniqueness.  Suppose
that $L(Y_{01},K_{01})$ and $L(Y_{12},K_{12})$ carry the relative spin
structures above with background classes
$(b_+(X_0,\ul{x}_0),b_-(X_1,\ul{x}_1))$ and
$(b_+(X_1,\ul{x}_1),b_-(X_2,\ul{x}_2))$.  By Remark \ref{rem:relspin},
$b_+(X_1,\ul{x}_1) = b_-(X_1,\ul{x}_1) + w_2(TM(X_1,\ul{x}_1))$, which
is the matching condition on the middle factor in \cite[Theorem
6.5.5]{orient}.  If the composition is transverse and embedded,
\cite[Remark 4.6.2(a)]{orient} gives a relative spin structure on
$L(Y_{01},K_{01}) \circ L(Y_{12},K_{12})$ with background class
$(b_+(X_0,\ul{x}_0),b_-(X_2,\ul{x}_2))$.  Whenever $L(Y_{02},K_{02})$
is simply connected, for example when it is an elementary
correspondence or the graph of a symplectomorphism as in the Cerf
moves, this structure agrees with the one constructed above, by
uniqueness.  The other choices of signs differ from these by shifting
background classes by $w_2(TM(X_i,\ul{x}_i))$ as in \cite[Remark
5.2.2]{orient}.
\end{proof} 

\subsection{Symplectic-valued field theory}
\label{sec:invariance2}

Putting everything together we construct a functor from the tangle
category to the category of (symplectic manifolds, equivalence classes
of generalized Lagrangian correspondences.)  We denote by 
\[\B = \{ \omega_1/2,\ldots , \omega_{r-1}/2 \}\] 
the set of midpoints on the edges connecting the origin with the
vertices in the Weyl alcove, and restrict to marked surfaces with
labels in this set.  For example, in the rank two case this assumption
means that all labels are contained in the set $\B = \{ {1/4} \} \subset
[0,{1/2}] \cong \Alc $ corresponding to traceless holonomies.

\begin{definition} {\rm (Admissible tangle category)} 
Fix coprime integers $r,d >0$.  Let $X$ be a compact oriented surface.
Denote by $\Tan(X,r,d)$ the labelled version of the cylindrical
tangle category of Definition \ref{def:cobtangles}: its objects are
markings in $X$ with labels in $\B$ such that the labels are
admissible as in Proposition \ref{prop:smooth}, and its morphisms are
isotopy classes of $X$-cylindrical tangles with labels in $\B$.
\end{definition} 

\begin{example}
In the simplest case $r =2,
d= 1$ the category $\Tan(X,r,d)$ is the category of tangles in $X$
whose objects are markings of odd order.
\end{example} 

The proof that the Cerf relations hold uses the following holonomy
description of cups and caps, valid for surfaces of arbitrary genus.
For a marking $\ul{x}$ of $X$ with labels $\ul{\mu}$ define the 
{\em framed moduli space} resp. {\em identity level set}
\[ M^{\on{fr}}(X,\ul{x}) = G^{2g} \times \cC_{\ul{\mu}}, \quad Z_{\ul{x}} =
\Phi_{\ul{x}}^{-1}(1) \subset M^{\on{fr}}(X,\ul{x}), \quad M(X,\ul{x}) =
Z_{\ul{x}}/G ,\]
where $\Phi_{\ul{x}}: M^{\on{fr}}(X,\ul{x}) \to G$ is the map \eqref{eq:phi}.  If
the labels are admissible then the stabilizer of every point of
$Z_{\ul{x}}$ is the center $Z(G)$, by the proof of Proposition
\ref{prop:smooth}.  Hence $G/Z(G)$ acts freely on $Z_{\ul{x}}$, and the
image of $\d \Phi_{\ul{x}}$ at any point of $Z_{\ul{x}}$ is all of
$T_1 G$, since by \cite[Proposition 4.1]{al:mom} this image is the
annihilator of the Lie algebra of the stabilizer.  So $\Phi_{\ul{x}}$
is a submersion along $Z_{\ul{x}}$.

\begin{lemma} \label{lem:elemhol} {\rm (Holonomy description of cups and caps)}
Let $(Y,K,\phi)$ be an elementary cylindrical tangle from
$(X,\ul{x}_-)$ to $(X,\ul{x}_+)$ with admissible labels, such that the
unique critical point of $f|_K$ is a minimum.  Thus $\ul{x}_+ =
\ul{x}_- \cup \{ x_p, x_{p'} \}$, where $x_p,x_{p'}$ are the endpoints
of the strand $K_0$ containing the critical point.  Choose a standard
presentation of $\pi_1(X \ssm \ul{x}_+)$ in which $x_p,x_{p'}$ are
consecutive, and in which the generators other than
$\gamma_p,\gamma_{p'}$ avoid a disk $D \subset X$ containing
$x_p,x_{p'}$; deleting $\gamma_p,\gamma_{p'}$ gives a standard
presentation of $\pi_1(X \ssm \ul{x}_-)$.  Let
\[ C(p,p') = \{ (\ul{a},\ul{b}) \in M^{\on{fr}}(X,\ul{x}_+) \, | \, b_p = b_{p'} \}
\quad \text{and} \quad \pi: C(p,p') \to M^{\on{fr}}(X,\ul{x}_-) \]
the map forgetting $b_p$ and $b_{p'}$.  Then:
\begin{enumerate}
\item \label{elemhol1} $\pi$ is a trivial fiber bundle with fiber
  $\cC_\nu$, where $\nu$ is the label of $K_0$, and $\Phi_{\ul{x}_+}
  |_{C(p,p')} = \Phi_{\ul{x}_-} \circ \pi$;
\item \label{elemhol2} $L(Y,K,\phi) = \bigl\{ \bigl( [\pi(\ul{a},\ul{b})],
  [\ul{a},\ul{b}] \bigr) \, \big| \, (\ul{a},\ul{b}) \in {C(p,p')} \cap
  Z_{\ul{x}_+} \bigr\}$;
\item \label{elemhol3} ${C(p,p')} \cap Z_{\ul{x}_+} = \pi^{-1}(Z_{\ul{x}_-})$
  is a $G$-invariant submanifold of $Z_{\ul{x}_+}$.  The projection
  $L(Y,K,\phi) \to M(X,\ul{x}_+)$ is an embedding with image $({C(p,p')} \cap
  Z_{\ul{x}_+})/G$, and the projection $L(Y,K,\phi) \to M(X,\ul{x}_-)$
  is the associated fiber bundle $Z_{\ul{x}_-} \times_G \cC_\nu \to
  M(X,\ul{x}_-)$.  In particular the projection is a submersion.
\end{enumerate}
The same holds for an elementary tangle whose critical point is a
maximum, with the roles of $\ul{x}_-$ and $\ul{x}_+$ exchanged.
\end{lemma}

\begin{proof}  The endpoints $x_p,x_{p'}$ of $K_0$ carry the label $\nu$
of $K_0$ and opposite orientations, $\eps_{p'} = - \eps_p$.  The loop
$\gamma_p^{\eps_p} \gamma_{p'}^{\eps_{p'}}$ bounds a disk in $Y \ssm K$
lying just below the critical value.  By the Seifert-van Kampen theorem
as in the proof of Lemma \ref{lem:compcor}, the representations of
$\pi_1(Y \ssm K)$ are exactly the representations of $\pi_1(X \ssm
\ul{x}_+)$ with $b_p = b_{p'}$.  This proves \eqref{elemhol2}.  Since
$\gamma_p,\gamma_{p'}$ are consecutive, the factor
$b_p^{\eps_p}b_{p'}^{\eps_{p'}}$ in \eqref{eq:phi} equals $1$ on ${C(p,p')}$,
which proves \eqref{elemhol1}.  Since $\Phi_{\ul{x}_-}$ is a
submersion along $Z_{\ul{x}_-}$ and $\pi$ is a submersion,
$\pi^{-1}(Z_{\ul{x}_-}) \cong Z_{\ul{x}_-} \times \cC_\nu$ is a
submanifold, on which $G/Z(G)$ acts freely.  Its quotient is the
associated bundle $Z_{\ul{x}_-} \times_G \cC_\nu$, and by
\eqref{elemhol2} $L(Y,K,\phi)$ is the graph of the bundle projection.
This proves \eqref{elemhol3}.  The case of a maximum follows by
reversing the direction of $Y$.
\end{proof}

\begin{lemma} \label{lem:transcrit} {\rm (Composing a cup with a cap)}
Let $(Y_1,K_1,\phi_1)$ from $(X,\ul{x}_0)$ to $(X,\ul{x}_1)$ and
$(Y_2,K_2,\phi_2)$ from $(X,\ul{x}_1)$ to $(X,\ul{x}_2)$ be elementary
cylindrical tangles with admissible labels, such that $f|_{K_1}$ has a
minimum on a strand with endpoints $x_p,x_{p'} \in \ul{x}_1$, and
$f|_{K_2}$ has a maximum on a strand with endpoints $x_q,x_{q'} \in
\ul{x}_1$.  Suppose that either
\begin{enumerate}
\item \label{tc1} $\{p,p'\} \cap \{q,q'\} = \emptyset$, or
\item \label{tc2} $p' = q'$ and $p \neq q$, so that the two critical
  points lie on the same strand of $K_1 \cup K_2$.
\end{enumerate}
Then the geometric composition $L(Y_1,K_1,\phi_1) \circ
L(Y_2,K_2,\phi_2)$ is transverse and embedded, and equal to
$L((Y_1,K_1,\phi_1) \circ (Y_2,K_2,\phi_2))$.  In case \eqref{tc2} it
is the diagonal in $M(X,\ul{x}_0)^- \times M(X,\ul{x}_2)$, under the
identification $\ul{x}_0 \cong \ul{x}_2$ given by the strands of $K_1
\cup K_2$.
\end{lemma}

\begin{proof}   Choose a standard presentation of $\pi_1(X \ssm \ul{x}_1)$
adapted, in the sense of Lemma \ref{lem:elemhol}, to both pairs.  In case
\eqref{tc1} the two pairs lie in disjoint disks; in case \eqref{tc2}
the points $x_p,x_{p'},x_q$ are consecutive, in this order, in a
single disk.  Let
\[ {C(p,p')} = \{ b_p = b_{p'} \}, \quad {C(q,q')} = \{ b_q = b_{q'} \} \subset M^{\on{fr}} :=
M^{\on{fr}}(X,{\ul{x}_1}), \quad Z := Z_{\ul{x}_1} .\]
By Lemma \ref{lem:elemhol} the projections of $L(Y_1,K_1,\phi_1)$ and
$L(Y_2,K_2,\phi_2)$ to $M(X,\ul{x}_1)$ are embeddings with images $({C(p,p')}
\cap Z)/G$ and $({C(q,q')} \cap Z)/G$.  

The fiber product of two embeddings is
transversally cut out if and only if their images intersect
transversally, so it suffices to show that ${C(p,p')} \cap Z$ and ${C(q,q')} \cap Z$
intersect transversally in $Z$.  Since 
\[ p \notin \{q,q'\}, \quad q \notin \{p,p'\}, \] 
the intersection ${C(p,p')} \cap {C(q,q')}$ is the graph of a map from the product of the factors of $M^{\on{fr}}$
other than those of $b_p$ and $b_q$.  Hence ${C(p,p')}$ and ${C(q,q')}$ intersect
transversally in $M^{\on{fr}}$.  On ${C(p,p')} \cap {C(q,q')}$, the map $\Phi_{\ul{x}_1}$ factors
as $\Phi_{\ul{x}'} \circ \pi'$, where $\pi'$ is a submersion and
$\ul{x}'$ is admissible.  Indeed, in case \eqref{tc1} let $\ul{x}' =
\ul{x}_1 \ssm \{x_p,x_{p'},x_q,x_{q'}\}$ and let $\pi'$ forget
$b_p,b_{p'},b_q,b_{q'}$.  Then $\ul{x}'$ is obtained from $\ul{x}_0$
by removing the two endpoints of a single strand, which carry the same
label with opposite orientations.  This does not change the parity
condition \eqref{eq:smooth3}, in which a negatively oriented marking with
label $\nu$ contributes $*\nu$.  In case \eqref{tc2} let $\ul{x}' =
\ul{x}_0$ and let $\pi'$ forget $b_p$ and $b_q$, so that the common
value $b_{p'} = b_p = b_q$ is the holonomy around the through strand;
since $\eps_p = \eps_q = -\eps_{p'}$, the three consecutive factors in
\eqref{eq:phi} multiply to $b_{p'}^{-\eps_{p'}}$, which is the factor for
the corresponding marking of $\ul{x}_0$.  In both cases
$\Phi_{\ul{x}'}$ is a submersion along $Z_{\ul{x}'}$, so ${C(p,p')} \cap {C(q,q')}$
meets $Z$ transversally, and hence so do ${C(p,p')}$ and ${C(q,q')}$.  Therefore ${C(p,p')}
\cap Z$, ${C(q,q')} \cap Z$ and ${C(p,p')} \cap {C(q,q')} \cap Z$ are submanifolds of $Z$
whose tangent spaces are the intersections of the corresponding
tangent spaces, and
\begin{eqnarray*} 
\codim_Z({C(p,p')} \cap {C(q,q')} \cap Z) &=& \codim_{M^{\on{fr}}} {C(p,p')} + \codim_{M^{\on{fr}}} {C(q,q')} \\
&=& \codim_Z({C(p,p')}
\cap Z) + \codim_Z({C(q,q')} \cap Z) .\end{eqnarray*}
So ${C(p,p')} \cap Z$ and ${C(q,q')} \cap Z$ intersect transversally in $Z$.  Since
$G/Z(G)$ acts freely on $Z$, the transversality descends to
$M(X,\ul{x}_1)$.

The fiber product is thus identified with $({C(p,p')} \cap {C(q,q')} \cap Z)/G$,
mapping to $M(X,\ul{x}_0) \times M(X,\ul{x}_2)$ by restricting
holonomies.  In case \eqref{tc2}, $\pi'$ identifies $({C(p,p')} \cap {C(q,q')} \cap
Z)/G$ with $M(X,\ul{x}_0)$, and the map is the diagonal embedding.  In
case \eqref{tc1}, suppose two points $(\ul{a},\ul{b}), (\ul{a}',
\ul{b}') \in {C(p,p')} \cap {C(q,q')} \cap Z$ have the same image.  Then there exist
$g,h \in G$ with $g (\ul{a},\ul{b}|_{\ul{x}_0}) =
(\ul{a}',\ul{b}'|_{\ul{x}_0})$ and $h (\ul{a},\ul{b}|_{\ul{x}_2}) =
(\ul{a}',\ul{b}'|_{\ul{x}_2})$.  Then $g^{-1}h$ fixes
$\pi'(\ul{a},\ul{b}) \in Z_{\ul{x}'}$, so $g^{-1}h \in Z(G)$, and the
two points define the same element of $({C(p,p')} \cap {C(q,q')} \cap Z)/G$.  The same
argument at the level of Lie algebras, using that $Z(G)$ is discrete,
shows that the map is an immersion.  Since the fiber product is
compact, the composition is embedded.  It equals $L((Y_1,K_1,\phi_1)
\circ (Y_2,K_2,\phi_2))$ by Lemma \ref{lem:compcor}.
\end{proof}

\begin{theorem}
\label{thm:extendt} {\rm (Symplectic-valued field theory for admissible tangles)}
For $r \ge 2,d > 0$ coprime, partially define a functor $\Phi:\Tan(X,r,d)
\to \Symp^\#_{1/2r}$ by mapping
\begin{itemize}
\item an object $(\ul{x},\ul{\mu})$ to the moduli space $M(X,\ul{x})$
  of flat $SU(r)$-bundles with fixed holonomies;
\item an elementary morphism $[(Y,K,\phi)]$ to the Lagrangian
  correspondence $L(Y,K,\phi)$.
\end{itemize} 
Then $\Phi$ extends to a $\Symp^\#_{1/2r}$-valued field theory.
\end{theorem} 

\begin{proof} 
We first check that the partial map is well-defined.  Any equivalence
of bordisms $\psi: (Y,K,\phi) \to (Y',K', \phi')$ induces an equality
of Lagrangians $L(Y,K,\phi) = L(Y',K',\phi')$ by pull-back.  So any
equivalence class of elementary bordisms gives a Lagrangian
correspondence.  

It remains to check the Cerf relations in Theorem \ref{thm:suffices}.
These relations follow from suitable equivariant versions of the
results of \cite{field}.  However, we prefer to give an explicit
argument, based on Lemmas \ref{lem:elemhol} and \ref{lem:transcrit}, which
applies to surfaces $X$ of arbitrary genus.  Recall that two
generalized correspondences define the same morphism of
$\Symp^\#_{1/2r}$ if they are related by replacing two consecutive
correspondences by their geometric composition, whenever that
composition is transverse and embedded and the composite, with its
induced brane structure, again lies in the class of admissible
correspondences.

We first note that all correspondences arising below have the required
brane structures.  Each is either the graph of a symplectomorphism,
an elementary correspondence as in Lemma \ref{lem:Kbrane}, or the
correspondence $L$ of a composition of two elementary tangles with
critical points on disjoint strands.  By Lemma \ref{lem:elemhol}, applied
twice, such an $L$ is a fiber bundle over the moduli space of an
admissible marked surface, with fiber a product of two conjugacy
classes.  So $L$ is simply connected, by \cite{ms:pb,ni:co} and since
conjugacy classes are simply connected.  Since $\pi_2(M,L)$ is then the
image of $\pi_2(M)$, each such $L$ is monotone with the constant $\tau
= 1/2r$ of Theorem \ref{thm:monthm}.  The construction of relative spin
structures in the proof of Lemma \ref{lem:Kbrane} applies verbatim.  By
simple connectivity the relative spin structure with background class
\eqref{eq:bpm} is unique, so the relative spin structure induced on a
composite agrees with the one constructed directly.  It therefore
suffices to check that the relevant compositions are transverse and
embedded, with the expected composites.

\vskip .1in \noindent {\em Trivial tangles.}  By Lemma \ref{lem:Lex}
\eqref{cbord}, the correspondence of a trivial tangle is the graph of
the identity, that is, the diagonal, which is the identity morphism of
$\Symp^\#_{1/2r}$.

\vskip .1in \noindent {\em Braid gluing.}  Suppose that
$(Y_2,K_2,\phi_2)$ is a braid; the case that $(Y_1,K_1,\phi_1)$ is
a braid is the same.  By Lemma \ref{lem:Lex} \eqref{cbord},
$L(Y_2,K_2,\phi_2)$ is the graph of a symplectomorphism $\psi$.  The
geometric composition of any Lagrangian correspondence $L_{01}$ with
the graph of a symplectomorphism $\psi$ is transverse and embedded,
with composite $(1 \times \psi)(L_{01})$.  By Lemma \ref{lem:compcor} the
composite equals $L((Y_1,K_1,\phi_1) \circ (Y_2,K_2,\phi_2))$.

\vskip .1in \noindent {\em Critical point cancellation.}  In the
situation of Theorem \ref{thm:cerftangle} \eqref{cancel}, $f|_{K_1}$ has a minimum
and $f|_{K_2}$ a maximum on the same strand.  Label the markings of
$\ul{x}_1$ so that the minimum joins $x_{j-1}$ and $x_j$ and the
maximum joins $x_j$ and $x_{j+1}$.  This is case \eqref{tc2} of Lemma
\ref{lem:transcrit}, with $(p,p',q) = (j-1,j,j+1)$.  Hence
$L(Y_1,K_1,\phi_1) \circ L(Y_2,K_2,\phi_2)$ is transverse and embedded
and equal to the diagonal.  This is the correspondence of the
cylindrical composite $(Y_1,K_1,\phi_1) \circ (Y_2,K_2,\phi_2)$.

\vskip .1in \noindent {\em Critical point reversal.}  Suppose that
$(Y_1,K_1,\phi_1)$ and $(Y_2,K_2,\phi_2)$ have critical points $y_1 \in
K_1$ and $y_2 \in K_2$ on strands meeting $\ul{x}_1$ in disjoint sets.
We claim that $L(Y_1,K_1,\phi_1) \circ L(Y_2,K_2,\phi_2)$ is
transverse and embedded with composite $L((Y_1,K_1,\phi_1) \circ
(Y_2,K_2,\phi_2))$.  Applying the claim also to the reversed pair
$(Y_1',K_1',\phi_1'), (Y_2',K_2',\phi_2')$, whose composition is
equivalent to the same tangle, proves the relation.  Since $K$ is
one-dimensional, each $y_i$ is a minimum or a maximum of $f|_K$.  The
four possibilities are covered by the following three cases:
\begin{enumerate}
\item If $y_1$ is a minimum and $y_2$ a maximum, the claim is case
  \eqref{tc1} of Lemma \ref{lem:transcrit}.
\item If $y_1$ is a maximum ($y_2$ arbitrary), then by Lemma \ref{lem:elemhol} (for a
  maximum) the projection $L(Y_1,K_1,\phi_1) \to M(X,\ul{x}_1)$ is a
  submersion, so the fiber product $L(Y_1,K_1,\phi_1) \times_{M(X,\ul{x}_1)}
  L(Y_2,K_2,\phi_2)$ is transversally cut out.  Moreover the projection
  $L(Y_1,K_1,\phi_1) \to M(X,\ul{x}_0)$ is an embedding.  If two points
  $(m_0,m_1,m_2),(m_0,m_1',m_2)$ of the fiber product have the same
  image, then $(m_0,m_1),(m_0,m_1') \in L(Y_1,K_1,\phi_1)$ and so $m_1
  = m_1'$.  The same argument for tangent vectors shows that the map
  to $M(X,\ul{x}_0) \times M(X,\ul{x}_2)$ is an injective immersion,
  hence an embedding by compactness.
\item If $y_1$ and $y_2$ are both minima, the same argument applies
  with the roles of the two factors exchanged: by Lemma \ref{lem:elemhol}
  the projection $L(Y_2,K_2,\phi_2) \to M(X,\ul{x}_1)$ is a submersion
  and $L(Y_2,K_2,\phi_2) \to M(X,\ul{x}_2)$ is an embedding.
\end{enumerate}
In each case the composite equals $L((Y_1,K_1,\phi_1) \circ
(Y_2,K_2,\phi_2))$ by Lemma \ref{lem:compcor}, which proves the claim.
\end{proof}

The field theory of Theorem \ref{thm:extendt} is defined on isotopy
classes of tangles, but in fact depends only on the coarser
equivalence classes of Definition \ref{def:tancat} \eqref{tangdef}; see Remark
\ref{rem:isotopy}.

\begin{proposition} \label{prop:descent} {\rm (Invariance under
    equivalence)} Suppose that $X$ is closed and connected, and let
  $\Phi: \Tan(X,r,d) \to \Symp^\#_{1/2r}$ be the field theory of
  Theorem \ref{thm:extendt}.  If $K_0$ and $K_1$ are labelled
  $X$-cylindrical tangles from $\ul{x}_-$ to $\ul{x}_+$ such that
  $([-1,1] \times X, K_0,\phi)$ and $([-1,1] \times X,K_1,\phi)$ are
  equivalent by an equivalence preserving the labels, then
\[ \Phi([K_0]) = \Phi([K_1]) .\]
\end{proposition}

\begin{proof}
The proof reduces an arbitrary equivalence to a level-preserving one,
which acts trivially on the generalized correspondences.  Let $Y =
[-1,1] \times X$ and let $\psi: Y \to Y$ be an equivalence with
$\psi(K_0) = K_1$, so that $\psi$ is the identity on $\partial Y$.
By uniqueness of collars we may assume, after an isotopy of $\psi$
through equivalences from $K_0$ to tangles isotopic to $K_1$, that
$\psi$ is the identity near $\partial Y$.  Call a diffeomorphism of
$Y$ {\em level-preserving} if it has the form
\[ g: Y \to Y, \quad (t,x) \mapsto (t, g_t(x)) \]
for a smooth family $g_t, t \in [-1,1]$ of diffeomorphisms of $X$.

\vskip .1in \noindent  {\em Step 1: The diffeomorphism $\psi$ is isotopic, through diffeomorphisms equal to the
  identity near $\partial Y$, to a level-preserving diffeomorphism $g$
  with $g_t$ equal to the identity for $t$ near $\pm 1$.} 
  
  Suppose first
that $X$ has positive genus.  By Waldhausen \cite[Lemma
  3.5]{waldhausen}, a homeomorphism of $Y$ equal to the identity on $\{
-1 \} \times X$ is isotopic, by an isotopy constant on $\partial Y$,
to a level-preserving homeomorphism.  Waldhausen's lemma is stated in
the piecewise-linear category, and we pass to the smooth category as
follows.  Waldhausen's isotopy is in particular an isotopy through
homeomorphisms equal to the identity on $\partial Y$, and the
level-preserving homeomorphism it provides has the form
\[ (t,x) \mapsto (t, f_t(x)) \]
for a loop $f_t, t \in [-1,1]$ in the group $\on{Homeo}(X)$ of
homeomorphisms of $X$, based at the identity.  For a closed surface
$X$ of positive genus, the identity components of $\Diff(X)$ and
$\on{Homeo}(X)$ are both contractible if the genus is at least two,
and both homotopy equivalent to the torus $X$ acting on itself by
translations if the genus is one, by Earle-Eells \cite{ee:fibre} and
Hamstrom \cite{hamstrom:homotopy}; in particular the inclusion
$\Diff(X) \to \on{Homeo}(X)$ induces an isomorphism on fundamental
groups.  So the loop $f_t$ is homotopic, through loops based at the
identity, to a loop $g_t$ of diffeomorphisms.  The homotopy defines an
isotopy, through level-preserving homeomorphisms equal to the identity
on $\partial Y$, to the diffeomorphism
\[ g: Y \to Y, \quad (t,x) \mapsto (t, g_t(x)) .\]
Thus $\psi$ and $g$ are diffeomorphisms of $Y$ equal to the identity on
$\partial Y$ that are isotopic rel $\partial Y$ through homeomorphisms.
By Cerf \cite[Corollaire 1]{ce:groupes}, the inclusion of the group of
diffeomorphisms of a compact $3$-manifold equal to the identity on the
boundary into the corresponding group of homeomorphisms is a homotopy
equivalence; the proof uses the contractibility of the group of
diffeomorphisms of the $3$-ball equal to the identity on the boundary,
which was established by Hatcher \cite{ha:smale}.  Hence $\psi$ and $g$
are smoothly isotopic rel $\partial Y$.
After a further isotopy, $g_t$ is the identity for $t$ near $\pm 1$.  

Suppose next that $X$ has genus zero.  Then $Y = [-1,1] \times S^2$
is diffeomorphic to the complement of two disjoint open balls in
$S^3$.  By Brendle-Broaddus-Putman \cite[Lemma 2.5]{bbp:mcg}, the
group of isotopy classes of orientation-preserving diffeomorphisms of
$Y$ equal to the identity on $\partial Y$ is generated by sphere twists
about spheres parallel to the components of $\partial Y$; the proof
uses only Cerf's theorem that the mapping class group of $S^3$ is
trivial.  A sphere twist about a sphere parallel to $\{ \pm 1 \} \times
S^2$ is isotopic, rel $\partial Y$, to the level-preserving
diffeomorphism
\[ (t,x) \mapsto (t, g_t(x)) \]
where $g_t \in SO(3)$ is a loop of rotations representing the
generator of $\pi_1(SO(3))$, equal to the identity outside a small
interval of values of $t$ near $\pm 1$, as in Remark \ref{rem:isotopy}.
Since compositions of level-preserving diffeomorphisms are
level-preserving, every element of $\Diff(Y \ \text{rel} \ \partial Y)$
is isotopic rel $\partial Y$ to a level-preserving diffeomorphism $g$
with $g_t$ the identity for $t$ near $\pm 1$.

\vskip .1in \noindent {\em Step 2: Level-preserving diffeomorphisms induce the same functors.}  By Step
1 and Remark \ref{rem:isotopy}, the tangle $K_1 = \psi(K_0)$ is
isotopic to $g(K_0)$, so it suffices to show $\Phi([g(K)]) =
\Phi([K])$ for any labelled $X$-cylindrical tangle $K$ and any $g$ as in
Step 1.  After an isotopy of $K$, the coordinate $t$ restricts to a
Morse function on $K$ with distinct critical values; choose values
$\ul{b}$ separating them as in the proof of Theorem
\ref{thm:suffices}.  After reparametrizing the family $g_t$ in $t$,
which changes $g$ by an isotopy through level-preserving
diffeomorphisms and $g(K)$ by an isotopy, we may assume that $g_t$ is
independent of $t$ near each $b_i$, so that $g$ maps tangles vertical
near the levels $\{ b_i \} \times X$ to tangles vertical near these
levels.  Since $t \circ g = t$, the restriction of $t$ to
$g(K)$ has the same critical values, and $(t,\ul{b})$ is a cylindrical
Morse datum for $g(K)$ whose pieces are the images $g(K_i)$ of the
pieces $K_i$ of $K$, with intermediate markings $g_{b_i}(\ul{x}_i)$.
Let
\[ \gamma_i : M(X,\ul{x}_i) \to M(X, g_{b_i}(\ul{x}_i)), \quad i =
0,\ldots, m \]
be the symplectomorphisms induced by the diffeomorphisms $g_{b_i}$ of
marked surfaces, which preserve labels and orientations of the
markings.  Since $g$ restricts to a diffeomorphism of pairs from
$([b_{i-1},b_i] \times X, K_i)$ to $([b_{i-1},b_i] \times X, g(K_i))$
covering $g_{b_{i-1}}$ and $g_{b_i}$ on the boundary, pullback of flat
bundles gives
\[ L(g(K_i)) = (\gamma_{i-1} \times \gamma_i) (L(K_i)), \quad i =
1,\ldots, m .\]
The graph $\on{graph}(\gamma_i)$ is a simply connected monotone
Lagrangian correspondence, and since $\gamma_i$ pulls back the
background classes \eqref{eq:bpm} of $M(X,g_{b_i}(\ul{x}_i))$ to those
of $M(X,\ul{x}_i)$, it admits a relative spin structure with the
required background classes, unique by simple connectivity.  The
geometric composition of any Lagrangian correspondence with the graph
of a symplectomorphism is transverse and embedded, so in
$\Symp^\#_{1/2r}$
\[ [L(g(K_i))] = [\on{graph}(\gamma_{i-1})^t, L(K_i),
\on{graph}(\gamma_i)] \]
where the superscript $t$ denotes the transpose.  Concatenating over
$i$, each consecutive pair $(\on{graph}(\gamma_i),
\on{graph}(\gamma_i)^t)$ composes to the diagonal, which is equivalent
to the empty sequence.  Since $g$ is the identity near $t = \pm 1$, the
maps $\gamma_0$ and $\gamma_m$ are the identity.  Hence
\[ \Phi([g(K)]) = [L(K_1), \ldots, L(K_m)] = \Phi([K]) ,\]
which proves the proposition.
\end{proof}

Composing with the categorification functor of Corollary
\ref{cor:categorify}, the same invariance holds for the
category-valued field theory of Theorem \ref{thm:extendtc} below.

It seems likely that, by a more detailed examination of Cerf theory,
one can allow simultaneously tangles and non-trivial bordisms, but we
have not checked the details.

\section{Functors for Lagrangian correspondences}
\label{sec:derived fh}

In this section we recall results of Oh \cite{oh:fl1} on Floer theory
in the presence of Maslov index two disks, and their quilted versions.
We also recall results from \cite{Ainfty} on \ainfty functors for
generalized Lagrangian correspondences.

\subsection{Monotone Floer theory}

We begin by recalling the definition of quilted Floer cohomology from
\cite{ww:quilts}.

\begin{definition} {\rm (Moduli of Maslov-index-two pseudoholomorphic disks)}   
Let $D\subset\C$ be the unit disk and fix the base point $1\in\pd D$.
Let $(M,\omega)$ be a compact monotone symplectic manifold and
$L\subset M$ an oriented monotone Lagrangian submanifold.  For any $J
\in \J(M,\omega)$ and submanifold $X\subset L$, let $\M_1^2(L,J,X)$
denote the moduli space of $J$-holomorphic disks $u:(D,\partial D) \to
(M,L)$ with Maslov number $2$ and one marked point satisfying $u(1)\in
X$, modulo the action of the group $\Aut(D,\partial D,1)$ of
automorphisms of the disk fixing $1 \in \partial D$.
\end{definition} 

Oh \cite{oh:fl1} proves that the moduli space of disks above gives
rise to a well-defined number:

\begin{proposition}  \label{prop:disk number} {\rm (Disk invariant of a Lagrangian)}  
For any $\ell \in L$ there exists a comeager subset $\J^{\rm
  reg}(\ell) \subset \J(M,\omega)$ such that $\M_1^2(L,J,\{\ell\})$ is
a finite set.  Any relative spin structure on $L$ induces an
orientation on $\M_1^2(L,J,\{ \ell \})$.  Letting $\eps: \M_1^2(L,J,\{
\ell \}) \to \{ \pm 1 \} $ denote the map comparing the given
orientation to the canonical orientation of a point, the disk number
of $L$,
\[ w(L) := \sum_{[u] \in \M_1^2(L,J,\{ \ell \})} \eps([u]) ,\]
is independent of $J \in J^{\reg}(\ell)$ and $\ell \in L$.
\end{proposition}  

We will now extend the definition of quilted Floer cohomology, using
the setup of \cite{we:co} but dropping the assumption on minimal
Maslov number at least three.

\begin{definition} 
\begin{enumerate} 
\item {\rm (Symplectic backgrounds)} Fix a monotonicity constant $\tau
  > 0$ and an even integer $N > 0$.  A {\em symplectic background} is
  a tuple $(M,\omega,b,\Lag^N(M))$ as follows.
\begin{enumerate}
\item {\rm (Bounded geometry)} $M$ is a smooth compact manifold;
\item {\rm (Monotonicity)} $\omega$ is a symplectic form on $M$ that
  is monotone, i.e.\ $[\omega] = \tau c_1(TM)$;
\item {\rm (Background class)} $b \in H^2(M,\Z_2)$ is a {\em
  background class}, which will be used for the construction of
  orientations; and
\item {\rm (Maslov cover)} $\Lag^N(M) \to \Lag(M)$ is an $N$-fold
  Maslov cover such that the induced $2$-fold Maslov covering
  $\Lag^2(M)$ is the oriented double cover.
\end{enumerate}
We often refer to a symplectic background $(M,\omega,b,\Lag^N(M))$ as
$M$.
\item {\rm (Lagrangian branes)} A {\em brane structure} on a compact
  Lagrangian $L$ consists of an orientation, a relative spin
  structure, and a grading.  An {\em admissible Lagrangian brane} is a
  compact, oriented, simply connected Lagrangian with brane
  structure.  (Torsion fundamental group would suffice for
  monotonicity, but simple connectivity is assumed for consistency
  with the composition theorem, Theorem \ref{thm:composethm}.  One can also assume other conditions that give
  monotonicity for pseudoholomorphic curves with boundary in these
  Lagrangians, or work with Novikov rings etc.)
\end{enumerate} 
\end{definition}

We recall the definition of quilted Floer cohomology from
Wehrheim-Woodward \cite{ww:quilts} and Ma'u-Wehrheim-Woodward
\cite{Ainfty}.  Strip shrinking and figure eight bubbling were
studied further by Bottman-Wehrheim \cite{bw:squiggly} and Bottman
\cite{bottman:fig8}, and geometric composition in quilted Floer theory
by Lekili-Lipyanskiy \cite{ll:geomcomp}.  Approaches to functoriality
beyond the monotone, embedded setting include Fukaya's immersed
correspondences \cite{fukaya:immersed}; see Abouzaid-Bottman
\cite{ab:functoriality} for a survey.

Let $\J_t(\ul{L})$ denote the space of {\em quilted
  time-dependent almost complex structures}
\[ \J_t(\ul{L}) = \prod_{j=0}^r
C^\infty([0,\delta_j],\J(M_j,\omega_j)) .\]
Fix a {\em quilted Hamiltonian perturbation}
\[\ul{H}\in \prod_{j=0}^r
 C^\infty([0,\delta_j] \times M_j) .\]
The space of {\em quilted Floer cochains} is generated by 
generalized trajectories of $\ul{H}$, 
\begin{equation*} \label{eq:cI}
\cI(\ul{L}) := \left\{ \ul{x}=\bigl(x_j: [0,\delta_j] \to M_j\bigr)_{j=0,\ldots,r} \, \left|
\begin{aligned}
\dot x_j(t) = X_{H_j}(x_j(t)), \\
(x_{j}(\delta_j),x_{j+1}(0)) \in L_{j(j+1)} 
\end{aligned} \right.\right\} .
\end{equation*}
Define the {\em Floer operator}
\[ \ul{u} \mapsto \overline{\partial}_{\ul{J},\ul{H}} \ul{u} =(
\partial_{J_j,H_j} u_j = \partial_s u_j + J_j \bigl(\partial_t u_j -
X_{H_j}(u_j) \bigr) )_{j=0}^ r .\]
Counting solutions to the equation
$\overline{\partial}_{\ul{J},\ul{H}} \ul{u} = 0 $ with boundary and
seam conditions in $\ul{L}$ defines a {\em quilted Floer operator}
\[ \partial: CF(\ul{L}) \to CF(\ul{L}), \quad CF(\ul{L}) =
\bigoplus_{\ul{x} \in \cI(\ul{L})} \Z x .\]

\begin{theorem} \label{thm:d2} 
{\rm (Quilted Floer cohomology)} Let
$\ul{L}=(L_{j(j+1)})_{j=0,\ldots,r}$ be a cyclic generalized
Lagrangian brane between symplectic backgrounds $M_j, j = 0,\ldots, r$
with the same monotonicity constant $\tau > 0$. Then, for any
collection $\ul{H}$ of Hamiltonian perturbations, widths
$\ul{\delta}=(\delta_j>0)_{j=0,\ldots,r}$, and for $\ul{J}$ in a
comeager subset $\J^{\reg}_t(\ul{L},\ul{H}) \subset \J_t(\ul{L})$, the
Floer differential $\partial: CF(\ul{L}) \to CF(\ul{L})$ satisfies
\[\partial^2= w(\ul{L}) \Id , \qquad w(\ul{L})= \sum_{j=0}^r
w(L_{j(j+1)}) .\]
The pair $( CF(\ul{L}),\partial)$ is independent of the choice of
$\ul{H}$ and $\ul{J}$, up to cochain homotopy (that is, if $w(\ul{L}) \neq 0$, up to homotopy equivalence in the category of matrix factorizations as in \cite{kr:mf}).
\end{theorem}

\begin{remark} \label{rem:cancel} {\rm (Floer theory of a pair of Lagrangians)} 
In the special case $\ul{L}=(L_0,L_1)$ of a cyclic correspondence
consisting of two Lagrangian submanifolds $L_0,L_1\subset M$ we have
$w(\ul{L})=w(L_0)-w(L_1)$.  Indeed the $-J_1$-holomorphic discs with
boundary on $L_1\subset M^-\times \{ \on{pt} \}$ are identified with
$J_1$-holomorphic discs with boundary on $L_1\subset M$ via an
anti-holomorphic involution of the domain, which is orientation
reversing for the moduli spaces of Maslov index two disks.  In
particular, the differential for a monotone pair $\ul{L}=(L,\psi(L))$
with any symplectomorphism $\psi\in\Symp(M)$ always squares to zero,
since $w(L) = w(\psi(L))$.
\end{remark}

\begin{theorem} \label{thm:derived main} {\rm (Behavior of Floer theory under
geometric composition)} Let $\ul{L}= (L_{01},\ldots,L_{r(r+1)})$ be a
  cyclic generalized Lagrangian correspondence with admissible brane
  structure.  Suppose that for some $1\leq j\leq r$ the composition
  $L_{(j-1)j} \circ L_{j(j+1)}$ is transverse and embedded and the modified sequence
  $\ul{L}':=(L_{01},\ldots,L_{(j-1)j} \circ
  L_{j(j+1)},\ldots,L_{r(r+1)})$ is monotone.  Then, with respect to
  the induced brane structure, the canonical identification of
  generalized intersection points induces an isomorphism of groups
  $CF(\ul{L}) \cong CF(\ul{L}')$ intertwining the Floer operators.
  In particular, if these groups are non-zero then
  $w(\ul{L})=w(\ul{L}')$, and if the disk invariants vanish then there
  is a canonical isomorphism $HF(\ul{L}) \cong HF(\ul{L}')$.
\end{theorem}

\begin{proof}
The bijection between the trajectory spaces for small widths and for
the composed Lagrangian correspondence in \cite{ww:isom} only requires
that the minimal Maslov number of the Lagrangians is at least two
(which is automatic in the monotone orientable case; see for example the remarks after condition (L3) in \cite{quiltfloer}).
Indeed, consider a family of holomorphic quilts of index one
contributing to the differential, with limits $x_\pm$.  Small-width
quilts for $\ul{L}$ and quilts for $\ul{L}'$ with these limits satisfy
the same relation $A = \tau I + c(x_-,x_+)$ between area $A$ and index
$I$: The limits are the same generalized intersection points, so the
action differences agree, and $\ul{L}'$ is monotone by assumption.
Any figure eight bubbling would lower the area by a positive amount,
as explained in \cite[Lemma 3.3.2]{ww:isom}, so the configuration that
remains after removing the bubbles is a quilt for $\ul{L}'$ of index at
most zero.  This configuration is not constant, since an index one
trajectory has $x_- \neq x_+$, its degree difference being $1$ mod
$N$ with $N \ge 2$.  So it would be a non-constant element of a moduli
space of expected dimension at most zero, which contradicts
regularity.  Disk and sphere bubbles are excluded by the same
count.  Conversely, the gluing argument of \cite{ww:isom}, which
produces small-width quilts for $\ul{L}$ from rigid quilts for
$\ul{L}'$, uses no assumption on Maslov indices, so the two sets of
index one trajectories are in bijection.   The comparison
of orientations in \cite{orient} is also independent of Maslov
indices.  Hence the morphism given by the canonical identification of
intersection points is a cochain map:
\[f:CF(\ul{L})\to CF(\ul{L}'), \quad f\partial = \partial' 
f ,\]
where $\partial$ and $\partial'$ are the Floer differentials on
$CF(\ul{L})$ resp.\ $CF(\ul{L}')$.  Similarly, the inverse is a
cochain map 
\[f^{-1}:CF(\ul{L}')\to CF(\ul{L}), \quad \partial  f^{-1} = f^{-1}
 \partial'. \]
So $f$ is an isomorphism intertwining the Floer operators.  Since
\[\partial^2 = w(\ul{L}) \on{Id}, \quad 
(\partial')^2 = w(\ul{L}') \on{Id},\] 
we have $w(\ul{L}) f = f \partial^2 = (\partial')^2 f = w(\ul{L}')
f$, so $w(\ul{L}) = w(\ul{L}')$ whenever $CF(\ul{L})$ is non-zero.
\end{proof}

Disk invariants are additive under embedded geometric composition.

\begin{corollary} \label{cor:additive} {\rm (Additivity of disk
    invariants)} Let $M_0,M_1,M_2$ be symplectic backgrounds with the
  same monotonicity constant $\tau > 0$, and let
\[ L_{01} \subset M_0^- \times M_1, \quad L_{12} \subset M_1^- \times
M_2 \]
be compact, oriented, simply connected Lagrangian correspondences with
admissible brane structures, such that the geometric composition
$L_{02} := L_{01} \circ L_{12}$ is transverse, embedded and simply
connected.  Then
\[ w(L_{01}) + w(L_{12}) = w(L_{02}) . \]
\end{corollary}

\begin{proof}
The additivity follows by applying Theorem \ref{thm:derived main} to a
cyclic correspondence containing $L_{02}$ and its transpose.  Consider
the cyclic generalized correspondences
\[ \ul{L} = (L_{01}, L_{12}, L_{02}^T), \quad \ul{L}' = (L_{02},
L_{02}^T) \]
where $L_{02}^T \subset M_2^- \times M_0$ is the transpose.  Since all
components are simply connected and monotone with the same constant,
both are monotone, and $\ul{L}'$ is obtained from $\ul{L}$ by an
embedded geometric composition.  Choose the Hamiltonian perturbation
for $\ul{L}'$ to be given by a Morse function on $L_{02}$ in a
Weinstein neighborhood.  The generalized intersection points of
$\ul{L}'$ correspond to the critical points of this function, so that
$CF(\ul{L}') \neq 0$, and by the identification of generators also
$CF(\ul{L}) \neq 0$.  By Theorem \ref{thm:derived main},
$w(\ul{L}) = w(\ul{L}')$.  By Remark \ref{rem:cancel} a correspondence
and its transpose contribute opposite disk invariants, so
\[ w(\ul{L}) = w(L_{01}) + w(L_{12}) - w(L_{02}), \quad w(\ul{L}')
= w(L_{02}) - w(L_{02}) = 0 ,\]
which proves the claim.
\end{proof}

The diagonal, and so the graph of any symplectomorphism, has vanishing
disk invariant.

\begin{lemma} \label{lem:diagdisk} {\rm (Disk invariant of the
    diagonal)} Let $M$ be a compact monotone symplectic manifold and
  $\Delta_M \subset M^- \times M$ the diagonal.  Then $w(\Delta_M) =
  0$.  Consequently $w(\Gamma(\psi)) = 0$ for the graph $\Gamma(\psi)
  = (1 \times \psi)(\Delta_M)$ of any symplectomorphism $\psi$ of $M$.
\end{lemma}

\begin{proof}
The proof identifies disks with boundary on the diagonal with spheres
in $M$.  Let $\dim_\R M = 2m$ and let $J$ be an almost complex
structure on $M$ that is generic for simple spheres.  For the split
almost complex structure $(-J) \oplus J$ on $M^- \times M$, a
holomorphic disk $u = (u_1,u_2)$ with boundary on $\Delta_M$ doubles to
a $J$-holomorphic sphere $v$, obtained by gluing $u_2$ to $u_1$
precomposed with complex conjugation along the boundary circle, and
the Maslov index of $u$ equals $2 c_1(v)$.  Thus Maslov index two disks
through a point $(x,x) \in \Delta_M$ correspond to $J$-holomorphic
spheres of Chern number one through $x$, together with the choice of a
circle on the sphere through the marked point.  Spheres of Chern
number one are simple, since $M$ is monotone, so their moduli space
with one marked point has dimension
\[ 2m + 2 - 6 + 2 = 2m - 2 , \]
and the image of its evaluation map has codimension two in $M$.  For
$x$ outside this image, there are no Maslov index two disks through
$(x,x)$.  The empty moduli space is regular, so $w(\Delta_M) = 0$ by
independence of the disk invariant from the choice of almost complex
structure.  The second statement follows from the invariance $w(L) =
w(\psi(L))$ of Remark \ref{rem:cancel} under the symplectomorphism $1
\times \psi$ of $M^- \times M$.
\end{proof}

\subsection{Quilted Fukaya categories}

\begin{definition} 
{\rm (Fukaya category)} Let $(M,\omega)$ be a monotone symplectic
background.  For any $w \in \Z$ let $\Fuk(M,w)$ be the {\em Fukaya
  category} as in Sheridan \cite{sh:hmsfano} whose
\begin{enumerate} 
\item {\rm (Objects)} objects are the set of admissible Lagrangian
  branes as in \cite[Definition 4.3]{Ainfty} with disk invariant $w(L) = w$;
\item {\rm (Morphisms)} for any pair of objects $(L_0,L_1)$, the
  ``space'' of morphisms $\Hom(L_0,L_1) := CF(L_0,L_1)$.
\item {\rm (Composition maps)} the higher composition maps
\[ \Hom(L_0,L_1) \times \ldots \times \Hom(L_{n-1},L_n) \to
  \Hom(L_0,L_n)[2-n], \quad n \ge 1 \]
are defined by counting holomorphic polygons with boundary on
$L_0,\ldots, L_n$.
\end{enumerate}
In \cite[Section 4.4]{Ainfty} the quilted Fukaya category
$\Fuk^{\sharp}(M,w)$ is defined similarly, by allowing generalized
branes with total disk invariant $w$; the \ainfty relations are
\cite[Theorem 4.15]{Ainfty}.
\end{definition} 

The following results are proved in \cite{Ainfty}.  Correspondences with non-zero disk invariant shift the
disk invariant of generalized Lagrangians: a generalized Lagrangian
correspondence $\ul{L}$ from $M_-$ to $M_+$ defines a functor
\[ \Phi(\ul{L}) : \Fuk^{\sharp}(M_-,w) \to \Fuk^{\sharp}(M_+, w +
w(\ul{L})) \]
by concatenation, since the total disk invariant of a concatenation is
the sum of the disk invariants of its components.  The functor is
defined on objects in \cite[Definition 5.10]{Ainfty}, the \ainfty
functor axioms are \cite[Theorem 5.13]{Ainfty}, the extension to
generalized correspondences is \cite[Remark 5.14]{Ainfty}, and
independence of choices up to quasiisomorphism is \cite[Theorem
6.14]{Ainfty}.  The admissibility conditions in \cite[Definition
4.3(a)]{Ainfty} require minimal Maslov number at least three or
vanishing disk invariant for each component.  However, this condition
is used only to ensure that the Floer operators square to zero, as in
\cite[Theorem 4.13]{Ainfty}, and the arguments apply without change to
generalized Lagrangians whose total disk invariant is fixed, by
Theorem \ref{thm:d2}.

\begin{theorem} \label{thm:composethm}  {\rm (Functor
for a geometric composition of Lagrangian correspondences)} Let
  $M_0,M_1,M_2$ be symplectic backgrounds with the same monotonicity
  constant and
\[L_{01} \subset M_0^- \times M_1, \quad L_{12} \subset M_1^- \times
  M_2\] 
compact, oriented, simply-connected Lagrangian correspondences
equipped with admissible brane structures. If $L_{01} \circ L_{12}$
is transverse and embedded into $M_0^- \times M_2$, and is simply
connected and equipped with the admissible brane structure induced
from those of $L_{01}$ and $L_{12}$, then 
\[ \Phi(L_{01} \circ L_{12}), \ \Phi(L_{01}) \circ \Phi(L_{12}) :
\Fuk^{\sharp}(M_0,w) \to \Fuk^{\sharp}(M_2,w + w(L_{01} \circ L_{12})) \]
are homotopic \ainfty functors.  Here the targets agree since $w(L_{01}
\circ L_{12}) = w(L_{01}) + w(L_{12})$ by Corollary
\ref{cor:additive}.
  \end{theorem} 

\noindent This is \cite[Theorem 1.2]{Ainfty}, with the admissibility
condition of \cite[Definition 4.3(a)]{Ainfty} on the disk invariants
replaced by the hypotheses above.  The proof in \cite[Section
7.5]{Ainfty} combines the algebraic composition theorem \cite[Theorem
7.20]{Ainfty} with the comparison of algebraic and geometric
composition \cite[Proposition 7.21]{Ainfty}; the latter is where strip
shrinking enters, and it applies here by the argument in the proof of
Theorem \ref{thm:derived main}.   The change to minimal Maslov number two affects none of the arguments, because by monotonicity any bubbling in the domain would create a configuration in a moduli space of negative dimension, except when the moduli space is one of strips.
In the strip case, the contributions of Maslov-index-two disk bubbles cancel because the source and target objects have the prescribed disk invariants. For example, for a quilted strip with boundary and seam labels 
$(L_0,L_{01},L_{12}, L_2)$ we have 
$ w(L_0) +  w(L_{01}) + w(L_{12}) - w(L_2) = 
 w + w(L_{01}) + w(L_{12}) -  w - w(L_{01} \circ L_{12})
 = 0 $ by Corollary \ref{cor:additive}.

  \begin{corollary} \label{cor:categorify} {\rm (Categorification
      functor)} For any $\tau > 0$, the assignments
\[ M \mapsto \Fuk^{\sharp}(M) := \bigsqcup_{w \in \Z}
\Fuk^{\sharp}(M,w), \quad \ul{L} \mapsto \Phi(\ul{L}) \]
for symplectic backgrounds $M$ with monotonicity constant $\tau$ and
generalized Lagrangian correspondences $\ul{L}$ with admissible brane
structures, where $\Phi(\ul{L})$ maps $\Fuk^{\sharp}(M_-,w)$ to
$\Fuk^{\sharp}(M_+,w + w(\ul{L}))$, define a {\em categorification
  functor} from $\Symp^\#_\tau$ to the category of (small curved
\ainfty categories, homotopy classes of \ainfty functors).  The disk
invariant $w(\ul{L}) = \sum_j w(L_{(j-1)j})$ is well-defined on
morphisms of $\Symp^\#_\tau$ by Corollary \ref{cor:additive}.
\end{corollary} 

\begin{proof}
The functor $\Phi(\ul{L})$ for a generalized correspondence is
defined by \cite[Theorem 5.13, Remark 5.14]{Ainfty} and is independent
of choices up to homotopy by \cite[Theorem 6.14]{Ainfty}.  The empty
correspondence induces the identity functor by \cite[Proposition
5.15]{Ainfty}, and concatenation of generalized correspondences
induces composition of functors by construction.  It remains to check
that $\Phi$ respects the equivalence relation defining the morphisms
of $\Symp^\#_\tau$, which is Theorem \ref{thm:composethm}.
\end{proof}

\section{Floer field theory for tangles}
\label{sec:compute}

\subsection{Definition of the theory} 
\label{sec:floer for links}

In this section we combine the results of the previous two sections to
obtain a field theory.  We first record the disk invariants of the
generalized correspondences associated to tangles.

\begin{proposition} \label{prop:delta} {\rm (Disk invariants of
    tangles)} Let $\Phi: \Tan(X,r,d) \to \Symp^\#_{1/2r}$ be the field
  theory of Theorem \ref{thm:extendt}.  There is a function $\delta$
  from the objects of $\Tan(X,r,d)$ to $\Z$, unique up to adding a
  constant on each set of objects related by adding and removing pairs
  of markings, such that for every morphism $[K]$ from $\ul{x}_-$ to
  $\ul{x}_+$
\[ w(\Phi([K])) = \delta(\ul{x}_+) - \delta(\ul{x}_-) . \]
\end{proposition}

\begin{proof}
The disk invariant defines a functor from $\Tan(X,r,d)$ to the
additive group $\Z$, which we integrate to a function on objects.  By
Corollary \ref{cor:additive}, the total disk invariant $w(\ul{L})$ is
invariant under the equivalence relation defining the morphisms of
$\Symp^\#_{1/2r}$, and is additive under concatenation.  Hence $[K] \mapsto w(\Phi([K]))$ is additive under composition of tangles.  By
Lemma \ref{lem:diagdisk} and Lemma \ref{lem:Lex} \eqref{cbord}, braids
have vanishing disk invariant.  An elementary cylindrical tangle with a
single critical point is the composition of a braid, a cup or a cap,
and a braid.  The disk invariant of a cup depends only on the labels
and signs of the markings and on the label and orientation of the new
strand: two cups with the same such data are related by composing
with braids, since the configuration space of labelled points on $X$
is connected.  Write 
\[ c(\ul{m},j) \in \Z  \] 
for the disk invariant of a cup
adding a pair of markings with label ${\bf j} := \omega_j/2$, $1 \le j
\le r-1$, and opposite signs to markings with multiset $\ul{m}$ of
labels and signs.  The cap removing the pair has
disk invariant $-c(\ul{m},j)$, since its correspondence is the
transpose.  Two cups on disjoint pairs of markings commute up to
isotopy by critical point reversal, so
\[ c(\ul{m},j) + c(\ul{m} + j, k) = c(\ul{m},k) + c(\ul{m} + k, j) \]
where $\ul{m} + j$ denotes the multiset $\ul{m}$ with the pair of
markings labelled ${\bf j}$ added.  Each
multiset $\ul{m}$ is obtained uniquely from a multiset $\ul{m}_0$
containing no pair of markings with equal labels and opposite signs by
adding pairs, and admissibility of $\ul{m}$ depends only on
$\ul{m}_0$.  Define $\delta(\ul{m}_0) = 0$ and $\delta(\ul{m})$ as the
sum of the values $c$ along any sequence of pair additions from
$\ul{m}_0$ to $\ul{m}$.  By the displayed relation this is independent
of the order of the additions.  Then each cup and cap changes the
disk invariant by the difference of the values of $\delta$ at its
ends, and the claim follows by additivity.
\end{proof}

\begin{remark}  \label{rem:cupvalue}  The value of the cup disk invariants $c(\ul{m},j)$ will not be needed for the paper, but in Woodward
\cite{wo:tlag} their value is computed for $r = 2$ to be
$c(\ul{m},j)  = -4 $ in  the case of a cup connecting three strands to five strands.  The corresponding Lagrangian two-sphere lies in a del Pezzo of degree four as in Proposition \ref{prop:dp}, and the homology class is one of the Manin Lagrangian homology classes associated to the $D_5$ root system.  Any of the pants decompositions of the five-marked sphere $(X,\ul{x})$  give rise to a Goldman  map $\Phi: M(X,\ul{x}) \to \R^2$ for which 
the image is the convex hull of $(1/4,0), (1/4, 1/2), (0,1/4), (0,1/2)$, pictured in Figure 8 of \cite{cw:flips}; the Lagrangian is the set of bundles whose holonomy around one of these circles vanishes, and so forms one of the two-sphere fibers at a vertex of the moment polytope.  After replacing 
the local $T^*S^2$ model near this Lagrangian with its standard 
integrable system given by the energy and angular momentum by 
one of the almost toric structures described in Example 2.3 of \cite{vw:at}, one obtains the almost toric structure considered
in Example 1.8 of \cite{wo:tlag}.  The computation in Example 1.8 of \cite{wo:tlag} works equally well for all of these Lagrangian spheres, and shows the disk count to be $c(\ul{m},j)  = -4$.
\end{remark}

Combining Theorem \ref{thm:extendt}, Proposition \ref{prop:delta}
and the functor of Corollary \ref{cor:categorify} we obtain the
following more precise version of Theorem \ref{thm:main}:

\begin{theorem} \label{thm:extendtc} {\rm (Category-valued field theory
    for tangles)} For any coprime positive integers $r,d$ and integer
  $w$, the partially defined functor $\Phi$ from $\Tan(X,r,d)$ to (small
  \ainfty categories, homotopy classes of \ainfty functors) that assigns
\begin{itemize} 
\item to any object $(\ul{x},\ul{\mu})$ the Fukaya \ainfty category
\[ \GFuk(M(X,\ul{x}),w + \delta(\ul{x})) ,\]
  with $\delta$ as in Proposition \ref{prop:delta}, and
\item to any morphism $[(Y,K,\phi)]$ the homotopy class of the \ainfty
  functor $\Phi(L(Y,K,\phi))$
\end{itemize} 
extend to a field theory for tangles from $\Tan(X,r,d)$ to the
category of (small \ainfty categories, homotopy classes of \ainfty
functors).
\end{theorem} 

In order to define group-valued invariants of manifolds containing
links, we apply the following device suggested to us by Seidel.  The
problem is that an empty set of markings is not admissible, since the
moduli space of flat bundles in this case contains reducibles.  We add
markings and components to the tangles in order to obtain well-defined
link invariants.  Let
\[\vert^{r+1} \subset [-1,1] \times S^2\] 
be the tangle with $r+1$ trivial strands labelled by $\omega_1/2 $.
Now suppose $K$ is a link in $S^3$ with components labelled by the
element $\omega_1/2$.  Let
\[ (Y,\ti{K}) = (S^3,K) {\sharp} ([-1,1] \times S^2,\vert^{r+1})\]
denote the connected sum of the ambient manifolds in $(S^3,K)$ and $([-1,1] \times
S^2,\vert^{r+1})$, as in Figure \ref{fig:addthree} for the case $r = 2$.
Equipped with the identity identifications $\phi$ on the boundary, we
obtain a bordism $(Y,\ti{K},\phi)$ from $r+1$ points to itself.
Denote by $\ul{L}(K) := \ul{L}(Y,\ti{K},\phi)$ the corresponding
generalized Lagrangian correspondence from Theorem \ref{thm:extendt}.

The correspondence $\ul{L}(K)$ depends, as a morphism of
$\Symp^\#_{1/2r}$, only on the isotopy class of the link $K$.  Indeed,
an isotopy of $K$ in $S^3$ may be assumed to avoid the ball removed in
forming the connected sum, since in $[0,1] \times S^3$ the trace of
the isotopy is two-dimensional and so generically misses an arc $[0,1]
\times \{ \pt \}$.  Isotopic links therefore give tangles $\ti{K}$ that
are isotopic rel boundary in $[-1,1] \times S^2$, and the claim
follows from Theorem \ref{thm:extendt}.  The same argument shows that
$\ul{L}(K)$ does not depend on the choice of the ball.

\begin{figure}[h]
\includegraphics[height=1in]{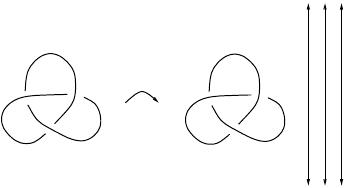}
\caption{Adding three trivial strands}
\label{fig:addthree}
\end{figure}

Knot invariants can now be defined by a device analogous to that used
in Seidel-Smith \cite{ss:li}, in which one obtains a knot invariant by
taking the Floer cohomology of a Lagrangian pair associated to a braid
presentation.  The tangle of the previous paragraph gives rise to a
Lagrangian correspondence $\ul{L}(K)$ from $M(X,\ul{x})$ to itself,
where $X = S^2$ and $\ul{x}$ is a set of order $r+1$.  We claim that
the quilted Floer cohomology of $\ul{L}(K)$ is well-defined.  To see
this, let $\Cap \subset [-1,1] \times X$ denote the tangle from $2n+r+1$
markings to $r+1$ markings that matches the $(n+1-i)$-th marking with
the $(n+i)$-th marking, for $i = 1,\ldots, n$, and consists of trivial
strands on the remaining $r+1$ markings; and let $\Cup$ denote the
tangle obtained by reversing it.  Let $L(\Cap)$ denote the
corresponding Lagrangian and $L(\Cup)^T$ its transpose, identified
with a fibered coisotropic
\[ L(\Cup) 
\cong \{ [a_1,\ldots,a_{2n+r+1}] \in M(S^2,\{
  x_1,\ldots,x_{2n+r+1} \}), \ a_j
  a_{2n+1-j} = 1,\ j = 1,\ldots, n \} . \]
Since $L(\Cap)$ and $L(\Cup)$ related by transpose, the disk
invariants $w(L(\Cap)), w(L(\Cup))$ cancel and the Lagrangian Floer
cohomology for these Lagrangians and their images under
symplectomorphisms is well-defined.

\begin{proposition} \label{prop:braidprop}    
Suppose $K$ is a link in $S^3$ given as the braid closure
of a braid $\beta$ in the spherical braid group $B_{n}$, obtained by
composing the braid element $\beta \times 1_n \in B_{2n}$ with the cup
and cap above.  Then the total disk invariant of $L(K)$ vanishes and
there is an isomorphism of Floer cohomology groups
\[HF(K) := HF(L(K)) \cong HF(L(\Cup),(\beta \times 1_n) L(\Cap)) .\]
\end{proposition}

\begin{proof} 
The disk invariant vanishes by Proposition \ref{prop:delta}: The
tangle $(S^3,K) \sharp ([-1,1] \times S^2,\vert^{r+1})$ is a morphism
from the marking $\ul{x}$ to itself, so its generalized correspondence
has disk invariant $\delta(\ul{x}) - \delta(\ul{x}) = 0$.  To compute
the Floer cohomology, choose a cylindrical Cerf decomposition of $(S^3,K) {\sharp} ([-1,1] \times S^2,\vert^{r+1})$
given by a Morse function with the property that all index $0$
critical points have smaller values than the index $1$ critical
points:
\[ ( I(y_0) < I(y_1) )  \implies ( f(y_0) < f(y_1) ), \quad \forall y_0, y_1 \in \Crit(f |_K) .\]
The composition of the corresponding Lagrangian correspondences is
smooth and embedded by Lemma \ref{lem:Lex} \eqref{qcups}.  So Theorem
\ref{thm:derived main} identifies the Floer cohomology of $L(K)$ with
that of the pair $(L(\Cup),(\beta \times 1_n) L(\Cap))$, which is
well-defined since the total disk invariant vanishes.
\end{proof}

\begin{example}
\label{ex:unknot} {\rm (Sphere-summed Floer homology of the unknot)}  
Take the cylindrical Cerf decomposition of the unknot $\bigcirc$
consisting of a cup $\cup$ and cap $\cap$, so that
\[ HF(\bigcirc) = HF( L(\cup), L(\cap)) \]
where 
\[ L(\cup) = \{ [g_1,\ldots,g_{r+3}] \in M(S^2, \{ x_1,\ldots,x_{r+3}
\}), \ g_1g_2 = 1 \}.\]
and $L(\cap)= L(\cup)^T$.  The moduli space $M(S^2, \{
x_3,\ldots,x_{r+3} \})$ of $r+1$ markings labelled $\omega_1/2$ is a
point, since its dimension is $(r+1) \cdot 2(r-1) - 2(r^2 - 1) = 0$.  It follows that 
$L(\cup)$ is a Lagrangian submanifold of $M(S^2, \{
x_1,\ldots,x_{r+3} \})$.  The map
\[ L(\cup) \to M(S^2, \{ x_3,\ldots,x_{r+3} \}) \]
forgetting $g_{1},g_{2}$ is a fiber bundle whose fiber is the
conjugacy class $\cC$ labelling the $1$-st and $2$-nd strands, so
$L(\cup) \cong \cC$.  (The forgetful map is not defined on all of
$M(S^2,\{ x_1,\ldots, x_{r+3} \})$, since $g_1 g_2$ need not be the
identity there.)  The conjugacy class $\cC$ is diffeomorphic to a
partial flag variety, as in \eqref{eq:quotient}, and admits a perfect
Morse function with only even indices: for example, $\cC$ admits the
structure of a Hamiltonian $SU(r)$-manifold with only finitely many
torus fixed points, and a generic component of the moment map is such
a function.  Since $L(\cup)$ is orientable, the differentials in the
spectral sequence of Oh, see \cite{fooo}, from $H(L(\cup);\Z)$ to
$HF(L(\cup),L(\cup))$ have odd degree, and so vanish.  Hence
\[ HF(\bigcirc) = HF(L(\cup),L(\cup)) \cong H(\cC;\Z) \]
as $\Z_{N}$-graded groups, where $N$ is the minimal Maslov number of
$L(\cup)$.  For example, if the label is $\omega_1/2$ then $\cC \cong
\C P^{r-1}$ and $HF(\bigcirc) \cong H(\C P^{r-1},\Z)$.
\end{example}

\noindent Kronheimer-Mrowka \cite{km:kh} investigate the similarity
with Khovanov-Rozansky homology \cite{kr:ma} in greater detail, in the
setting of instanton knot homology.

\subsection{Application to symplectic mapping classes} 
\label{sec:nontriv}
\label{sec:app}

In this section we give an application of the functors described above
to the symplectic topology of representation varieties.  Recall from
Remark \ref{rem:sbraid} that orientation-preserving diffeomorphisms of a
compact, oriented surface induce symplectomorphisms of the moduli
spaces of flat bundles.  In this section we study the case of marked
spheres and show that certain of these symplectomorphisms are
non-trivial in the symplectic mapping class group.

We introduce the following notation.  For $\mu \in \Alc$ let
$M_n(\mu)$ be the moduli space of flat bundles on the sphere $X$ with
a set of markings $\ul{x}$ of order $n$, $G = SU(2)$, and all labels
equal to the label $\mu$.  Recall that any smooth projective
complex-algebraic Fano surface is isomorphic either to $\P^1 \times
\P^1$ or to one of the {\em del Pezzo surfaces} $dP_d$, obtained by
blowing up $\P^2$ at $9-d \ge 0$ generic points.  Since $b_2(\P^1
\times \P^1) = 2$, a Fano surface with $b_2 \ge 3$ is determined up to
deformation by its second Betti number.

\begin{proposition} \label{prop:dp} {\rm (Identification of the first non-trivial moduli space as a del Pezzo)} For $\mu \in \Alc \cong [0,1/2]$ the moduli
  space $M_{5}(\mu) $ is diffeomorphic to the smooth manifold
  underlying
\begin{enumerate} 
\item {\rm (First Chamber)} the del Pezzo $dP_5$ for $\mu \in (0,
  \ff)$;
\item {\rm (Second Chamber)} the del Pezzo $dP_4$, for $\mu \in
  (\ff,\frac{2}{5})$; and
\item {\rm (Third Chamber)} the empty manifold, for $\mu \in (
  \frac{2}{5},1/2]$.
\end{enumerate} 
For $\mu = 0,1/5,2/5$ the moduli space contains reducibles.
\end{proposition}  

\begin{proof} 
By Boden-Hu \cite[Lemma 2.7]{boden:var}, the set of holonomies has a chamber
structure, so that within each chamber the diffeomorphism type of the
moduli space is constant.  To determine the chamber structure, it
suffices to find the moduli spaces $M_5(\mu)$ containing reducibles.
These are $M_5(\mu)$ with $\mu = 0,1/5,2/5$ corresponding to elements
\[g_1 = \ldots = g_5 = \diag(\exp (\pm 2\pi i \mu)), \quad g_1 \ldots
g_5 = 1 .\]

We first show that the moduli spaces are all Fano surface, that is,
have positive first Chern classes.  The moduli space $M_5(\mu)$ is
Fano in the first chamber $\mu < 1/5$, since $M_5(\mu)$ is a quotient
of $(S^2)^5$ by the diagonal action.  Indeed as explained in
\cite{ag:ei} bundles that are Mehta-Seshadri semistable \cite{ms:pb}
for weights in this range must have underlying bundle semistable, and
so trivial.  It follows that $M_5(\mu)$ is simply a git quotient of a
product of projective lines.  The moduli space $M_5(\mu)$ is also Fano
in the second chamber $\mu\in (1/5,2/5)$ by Theorem \ref{thm:monthm} since
for $\mu = 1/4$ the moduli space is monotone.

To identify which del Pezzo surfaces appear as moduli spaces of flat
bundles it suffices to determine the second Betti number.  The Betti
numbers of $M_5(\mu)$ in the first chamber $\mu < 1/5$ can be computed
by the method of Kirwan \cite{ki:coh}.  In this case the moduli space
is the geometric invariant theory quotient of $(\P^1)^5$ by the
diagonal action of $SL(2,\C)$:
\[ M_5(\mu)  = (\P^1)^5 \qu SL(2,\C) , \quad \mu < 1/5  .\]
Indeed in this chamber we adopt the Mehta-Seshadri description
\cite{ms:pb} and note that the underlying bundle is automatically
semistable of degree zero, hence trivial since the curve is genus
zero.  For $n \ge 1$
let
\[P_n(\mu,t) = \sum_{j = 0}^\infty \rank(H^j(M_n(\mu))) t^j \]
denote the Poincar\'e polynomial of $M_n(\mu), \mu < 1/n$.  By
\cite[p.193]{ki:coh}
\[ P_n(\mu,t) = (1 + t^2)^n ( 1 - t^4)^{-1} - 
\sum_{\frac{n}{2} < r \leq n} \left( \begin{array}{c} n \\ r
\end{array} \right)
t^{2(r-1)} (1-t^2)^{-1} .\]
In particular, $ P_5(\mu,t) = 1 + 5t^2 + t^4 \ \text{if} \ \mu < 1/5 $
which identifies the moduli space in the first chamber as the blow-up
of the plane at four points.

The Poincar\'e polynomial for the second chamber can be computed by
two techniques: the original approach of Atiyah-Bott \cite{at:mo},
extended to the parabolic case by Nitsure \cite{ni:co}, and the
recursive approach of Thaddeus \cite{th:pe}.  In the special case $\mu
= {1/4}$, the Atiyah-Bott approach gives
\[ P_n(\mu,t) = (1 + t^2)^n(1 - t^2)^{-1} (1- t^4)^{-1}  
- 2^{n-1} t^{n-1} (1-t^2)^{-2} \]
where the first term is the contribution from the equivariant
cohomology of the affine space of connections and the second is the
contribution from the unstable strata corresponding to abelian
orbifold connections, c.f. Street \cite[Theorem 3.8]{st:bu}.  Hence
\[ P_5(\mu,t) = 1 + 6t^2 + t^4 \ \text{if} \ 1/5 < \mu < 2/5 \]
which identifies the moduli space in the second chamber with the blow-up of the plane 
at five points. 

To see that the moduli space is empty in the third chamber, we use the
identification of the moduli space as the moduli space of piecewise
geodesics on the three-sphere.  Any solution $g_1 g_2 g_3 g_4 g_5 = 1$
with each $g_i$ having eigenvalues $\exp(\pm 2\pi i \mu)$ gives rise
to a geodesic pentagon in $SU(2) \cong S^3$ with edge lengths $2 \pi
\mu$.  Replacing each $g_i$ with $-g_i$ gives rise to a non-closed
geodesic $5$-gon with edge lengths $\pi - 2\pi \mu$ connecting
antipodes in $S^3$.  For $\mu > 2/5$, these edge lengths are less than
$\pi/5$ and so cannot connect antipodal points.  The latter have
distance $\pi$, which is greater than the sum of the edge lengths.
This contradicts the triangle inequality for the spherical metric, 
no solutions exist. 
\end{proof}  

\noindent Note that the moduli spaces in the first chamber are all
monotone: The symmetric group $S_5$ acts on $M_5(\mu)$ by permuting the
markings, preserving the symplectic class and the first Chern class,
and the $S_5$-invariant part of $H^2(dP_5;\R)$ has rank one.  In
general one expects monotonicity only for discrete values of the
holonomy parameters by Theorem \ref{thm:monthm}.  In the
identification of Fano surfaces we use that the symplectic form on
$M_5(\mu)$ is K\"ahler for the complex structure given by the
Mehta-Seshadri description \cite{ms:pb}.

We now study the squares of Dehn twists in moduli spaces of flat
bundles with fixed holonomies, following \cite[Section 3]{wo:ex}.  Let
$\sigma_{ij}^{(n)} \in \Diff_+(X)$ be the half-twist exchanging
markings $x_i$ and $x_j$ in Remark \ref{rem:sbraid}.  As long as the
labels $\mu_i$ and $\mu_j$ are equal, the diffeomorphism
$\sigma_{ij}^{(n)}$ induces a symplectomorphism of $M(X,\ul{\mu})$ by
pull-back of representations of the fundamental group under
$(\sigma_{ij}^{(n)})^{-1}$.

For a monotone symplectic manifold $M$ we write 
\[\Fuk^{\sharp}(M,M) := \Fuk^{\sharp}(M^- \times M, w(\Delta_M))\] 
for the Fukaya category of the product with disk invariant equal to
that $w(\Delta_M)$ of the diagonal, which vanishes by Lemma
\ref{lem:diagdisk}.  By Remark \ref{rem:cancel}, the graph
$\Gamma(\psi) = (1 \times \psi)(\Delta_M)$ of any symplectomorphism
$\psi$ of $M$ has the same disk invariant as the diagonal, and so
defines an object of $\Fuk^{\sharp}(M,M)$.  Recall
that $H(\Hom(\Delta_M,\Delta_M)) \cong QH(M)$, and that
$H(\Hom(\Delta_M,\Gamma(\psi)))$ is the fixed point Floer cohomology
$HF(\psi)$ of $\psi$, as a module over $QH(M)$ under composition; see
for example \cite{quiltfloer} for the identification of the
Lagrangian Floer cohomology of the pair $(\Delta_M,\Gamma(\psi))$ with
the Floer cohomology of $\psi$ by folding strips in $M^- \times M$ to
strips in $M$ with boundary values related by $\psi$.

The case $n = 5$ of the half-twist is a classical Dehn twist.  For $\mu
= \qq$ the Goldman function $h_{12}$ of Remark \ref{rem:smoothrem} takes
values in $[0,1/2]$, and by \cite[Section 3]{wo:ex} the half-twist
$\sigma_{12}^{(n)}$ acts on $M_n(\qq)$ as a fibered Dehn twist along
the coisotropic submanifold $h_{12}^{-1}(0)$, which is fibered over
$M_{n-2}(\qq)$ with fiber $\cC_{\qq} \cong S^2$.  For $n = 5$ the base
$M_3(\qq)$ is a point, since three traceless elements of $SU(2)$ with
product one are unique up to conjugacy.  Hence
\[ L := h_{12}^{-1}(0) = \{ [g_1, \ldots, g_5] \, | \, g_1 g_2 = 1 \}
\subset M_5(\qq) \]
is a Lagrangian sphere, and $\sigma_{12}^{(5)}$ acts on $M_5(\qq)$ by
a Dehn twist $\tau_L$ along $L$.  This agrees with Seidel's
description \cite{se:le} of $M_5(\qq)$, the monotone blow-up of the
projective plane at five points, as the moduli space of parabolic
bundles with weights $1/2$ on the five-pointed sphere, on which the
generators of the braid group act by Dehn twists along Lagrangian
spheres.

\begin{theorem} \label{thm:sq5}
 {\rm (Graph of the square of a Dehn twist is not the diagonal up to
   isomorphism)} Let $\sigma_{12}^{(5)}$ be the half twist around the
 first two markings in the spherical braid group $B_5$, and
 $(\sigma_{12}^{(5)} )^2$ its square.  Let
\[\Gamma((\sigma_{12}^{(5)})^2) \subset M_5(\qq)^- \times
M_5(\qq), \quad \Delta_5 \subset M_5(\qq)^- \times M_5(\qq)\]
be the graph of the action of $(\sigma_{12}^{(5)})^2$ on the moduli
space of bundles $M_5(\qq)$ resp.  the diagonal.  Then, working over a
coefficient field of characteristic zero,
\[\Delta_5 \ncong \Gamma((\sigma_{12}^{(5)})^2) \in \Fuk^{\sharp}(M_5(\qq),M_5(\qq)) \]
where $\cong$ denotes quasiisomorphism. 
\end{theorem}

\begin{proof}
This is essentially a result of Seidel \cite[Corollary 2.8, Example
  2.10]{se:le}, whose statement we first recall.  Let $M$ be the
blow-up of the projective plane at five points with a monotone
symplectic form, $L \subset M$ a Lagrangian sphere with class $l =
[L]$, and $\tau_L$ the Dehn twist along $L$.  Let $I_l \subset QH(M)$
be the ideal generated by $l$.  Seidel shows that $HF(\tau_L) \cong
QH(M)/I_l$ as $QH(M)$-modules, using the exact triangle for Dehn
twists \cite[Lemma 2.7]{se:le}.  There is a non-degenerate pairing
$HF(\tau_L) \otimes HF(\tau_L^{-1}) \to \Lambda$, compatible with the
$QH(M)$-module structures \cite[Lemma 2.2]{se:le}.  If there were an
isomorphism of $QH(M)$-modules $HF(\tau_L^{-1}) \cong HF(\tau_L)$,
then the pairing would make $QH(M)/I_l$ a Frobenius algebra.  Seidel
shows that this is impossible:  The image in $QH(M)/I_l$ of the
orthogonal complement of the canonical class is a subspace of
dimension $4$ on which the quantum product vanishes, while
$QH(M)/I_l$ has dimension $6$.  The conclusion follows from the fact that an isotropic subspace for a Frobenius
pairing has dimension at most half the total dimension.  Seidel
formulates the conclusion under the hypothesis that $\tau_L^2$ is
symplectically isotopic to the identity, but this hypothesis is used
only to produce an isomorphism $HF(\tau_L^{-1}) \cong HF(\tau_L)$ of
$QH(M)$-modules.

We translate this into a result about Lagrangian correspondences
as follows. Let $M = M_5(\qq)$ and $\tau_L$ the action of $\sigma_{12}^{(5)}$,
as explained before the statement of the Theorem.  Suppose that
$\Delta_5 \cong \Gamma(\tau_L^2)$ in $\Fuk^{\sharp}(M,M)$.  The
symplectomorphism $1 \times \tau_L^{-1}$ of $M^- \times M$ preserves
the disk invariant and the brane structures, and so induces an
automorphism of $\Fuk^{\sharp}(M,M)$.  It maps $\Delta_5$ to
$\Gamma(\tau_L^{-1})$ and $\Gamma(\tau_L^2)$ to $\Gamma(\tau_L)$.
Hence there is a quasiisomorphism $f: \Gamma(\tau_L^{-1}) \to
\Gamma(\tau_L)$.  Composition with $f$ induces an isomorphism
\[ HF(\tau_L^{-1}) \cong H(\Hom(\Delta_5, \Gamma(\tau_L^{-1}))) \to
H(\Hom(\Delta_5, \Gamma(\tau_L))) \cong HF(\tau_L) .\]
This isomorphism commutes with composition on the other side by
elements of $H(\Hom(\Delta_5,\Delta_5)) \cong QH(M)$, by associativity
of the composition in cohomology.  So it is an isomorphism of
$QH(M)$-modules, which contradicts Seidel's result.
\end{proof}  

 \begin{theorem} \label{thm:nontrivthm} {\rm (Non-triviality of squares of Dehn twists)} For $n \ge 1$ let
\[(\sigma_{12}^{(2n+3)})^2 \in
\on{Symp}(M_{2n+3}(\qq))\] 
denote the symplectomorphism associated to the square of the
half-twist of strands $1,2$, 
\[\Gamma((\sigma_{12}^{(2n+3)})^2) \subset
M_{2n+3}(\qq)^- \times M_{2n+3}(\qq)\] 
the corresponding Lagrangian, and $\Delta_{2n+3} \subset
M_{2n+3}(\qq)^- \times M_{2n+3}(\qq)$ the diagonal.  Then
\[\Gamma((\sigma_{12}^{(2n+3)})^2) \ncong \Delta_{2n+3} \in
\Fuk^{\sharp}(M_{2n+3}(\qq), M_{2n+3}(\qq)) \]
where $\cong$ denotes quasiisomorphism.
\end{theorem}

\begin{proof} The argument is an induction on the positive integer $n$.  
The case $n = 1$ is Seidel's Theorem \ref{thm:sq5}.  Suppose that the
statement in the Theorem holds for integers less than $n$ and suppose
that there exists a quasiisomorphism
\[ f \in \Hom(\Delta_{2n+3},\Gamma((\sigma_{12}^{(2n+3)})^2)) .\]
Let $K_{\vert \cup} $ resp. $K_{\cap \vert}$ denote a cup resp. cap at
the $3,4$ strands resp. $4,5$ strands. Then, thinking of a braid as a
special case of equivalence class of tangles, we have a Cerf
decomposition expressing the square $(\sigma_{12}^{(2n+1)})^2$ in
terms of $(\sigma_{12}^{(2n+3)})^2$
\[ (\sigma_{12}^{(2n+1)})^2 = K_{\vert \cup} \circ
(\sigma_{12}^{(2n+3)})^2 \circ K_{\cap \vert} \]
as in Figure \ref{fig:cupcap}.
\begin{figure}[h]
\includegraphics[height=2in]{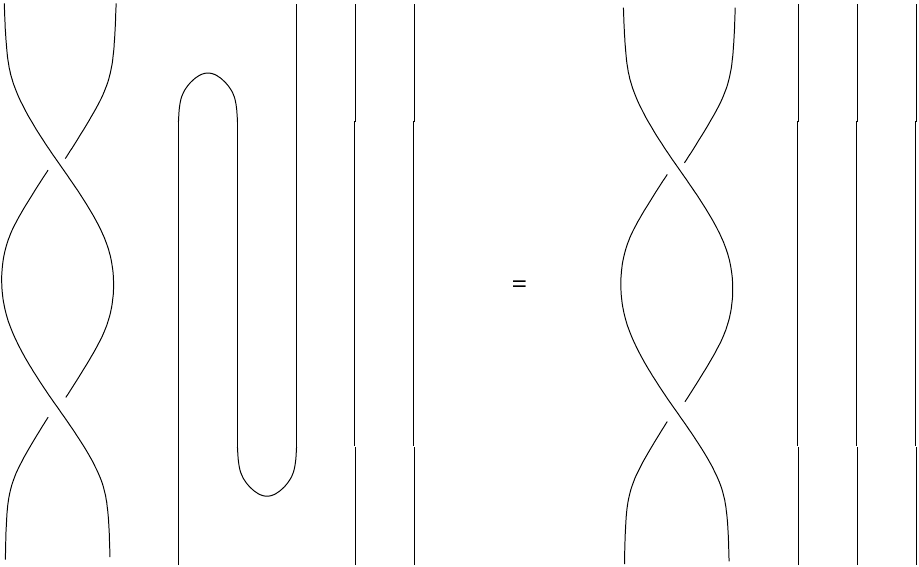}
\caption{Equivalence of full twists}
\label{fig:cupcap}
\end{figure}
Composition with the cup and the cap is implemented by a single
Lagrangian correspondence between the product moduli spaces.  Write
$M_k := M_{2n+k}(\qq)$ for $k = 1,3$, so that $L(K_{\vert \cup})
\subset M_1^- \times M_3$ and $L(K_{\cap \vert}) \subset M_3^- \times
M_1$.  Let
\[ C := \bigl\{ \bigl( (y,y'),(x,x') \bigr) \ \big| \ (x,y) \in
L(K_{\vert \cup}), \ (y',x') \in L(K_{\cap \vert}) \bigr\} \subset (M_3^-
\times M_3)^- \times (M_1^- \times M_1) .\]
Then $C \cong L(K_{\vert \cup}) \times L(K_{\cap \vert})$ is a
compact, simply connected Lagrangian correspondence by Lemma
\ref{lem:Kbrane}, with the product brane structure.  For a Lagrangian
$\Lambda \subset M_3^- \times M_3$, regarded as a correspondence from a
point, the geometric composition $\Lambda \circ C$ is the triple
composition $L(K_{\vert \cup}) \circ \Lambda \circ L(K_{\cap \vert})$,
and it is transverse and embedded exactly when the triple composition
is.

The two triple compositions that occur are transverse and embedded.
For $\Lambda = \Delta_{2n+3}$, the composition $L(K_{\vert \cup})
\circ L(K_{\cap \vert})$ is transverse and embedded and equal to the
diagonal $\Delta_{2n+1}$ by Lemma \ref{lem:transcrit} \eqref{tc2},
since the two critical points lie on the strand through the fourth
marking.  For $\Lambda = \Gamma((\sigma_{12}^{(2n+3)})^2)$, composition
with the graph of a symplectomorphism is always transverse and
embedded, and since the half-twist $\sigma_{12}$ involves only the
first two markings, $\Gamma((\sigma_{12}^{(2n+3)})^2) \circ L(K_{\cap
  \vert})$ is the Lagrangian of the elementary tangle obtained by
composing $K_{\cap \vert}$ with the braid, which still has a maximum
on the strand joining the fourth and fifth markings.  Lemma
\ref{lem:transcrit} \eqref{tc2} then shows that
\[ L(K_{\vert \cup}) \circ \Gamma((\sigma_{12}^{(2n+3)})^2) \circ
L(K_{\cap \vert}) = \Gamma((\sigma_{12}^{(2n+1)})^2) , \]
transversally and embedded, using the Cerf decomposition of Figure
\ref{fig:cupcap}.  The diagonal and the graph are diffeomorphic to
$M_1$, and so simply connected by \cite{ms:pb,ni:co}, and by the
compatibility statement of Lemma \ref{lem:Kbrane} their induced brane
structures agree with the standard ones.  Their disk invariants vanish
by Lemma \ref{lem:diagdisk}.  In particular the disk invariant of $C$
is
\[ w(C) = w(L(K_{\vert \cup})) + w(L(K_{\cap \vert})) =
w(\Delta_{2n+1}) = 0 \]
by Corollary \ref{cor:additive}, so that $\Phi(C)$ maps
$\Fuk^{\sharp}(M_3,M_3)$ to $\Fuk^{\sharp}(M_1,M_1)$.
Theorem \ref{thm:composethm} applied to $\Lambda$ and $C$ now
identifies the image of $\Lambda$ under the \ainfty functor $\Phi(C)$
with $\Lambda \circ C$ up to quasiisomorphism.  Applying $\Phi(C)$ to
the quasiisomorphism $f$ would therefore give a quasiisomorphism
\[ \Delta_{2n+1} \to \Gamma((\sigma_{12}^{(2n+1)})^2) \]
in $\Fuk^{\sharp}(M_1,M_1)$.  This is impossible by the inductive
hypothesis.  
\end{proof}

We now show that the symplectomorphisms induced by square of Dehn
twists are smoothly isotopic to the identity.  We first recall a
description of the Dehn twist as a Hamiltonian flow.

\begin{remark}  \label{rem:smoothrem}
By \cite[Section 3.4]{wo:ex}, the half-twist $\sigma_{jk}$ of markings
$x_j$ and $x_k$ acts on $M_n(\mu)$ by the Hamiltonian flow given as
follows.  Let 
\[ h_{jk}: M_n(\mu) \to [0,1/2] \] 
denote the {\em Goldman
  function} defined \cite{go:in} by the holonomy around an embedded
path $\gamma_{jk}$ containing only the markings $x_j,x_k$:
\[ [\exp( \diag( \pm 2 \pi i h_{jk} ([\varphi] )))] =
[\varphi(\gamma_{jk})] \in G/\Ad(G) .\]
Then $\sigma_{jk}$ acts by the Hamiltonian flow of 
\[h_{jk}(h_{jk} + 1/4): M_n(\mu) \to \R_{\ge 0} \]
on a dense open subset.  Note that the Goldman function $h_{jk}$ is
not smooth, but the quadratic function $h_{jk}^2$ is smooth near
$h_{jk}^{-1}(0)$ and the quadratic function $(h_{jk} - 1/2)^2$ is
smooth near $h_{jk}^{-1}(1/2)$ by considerations involving symplectic
cross-sections.  The square $\sigma_{jk}^2$ acts by the flow of
$2h_{jk}(h_{jk} + 1/4)$.
\end{remark} 

\begin{proposition} \label{prop:sisot}
Suppose that $M_n(\mu)$ is smooth.  Then the square $\sigma_{jk}^2$ of any half-twist
$\sigma_{jk}$ acts on $M_n(\mu)$ by a diffeomorphism that is smoothly
isotopic to the identity.
\end{proposition} 

\begin{proof}   First we introduce an isotopy corresponding 
to the deformation of the holonomy parameter.  The moduli spaces
$M_n(\mu)$ and the diffeomorphism $\sigma_{jk}^2$ fit into smooth
families as the label $\mu$ is varied within the chamber of values for
which the moduli spaces $M_n(\mu)$ are smooth.  That is, for $\eps>0$
small the union
\[ \ti{M}_n(\mu- \eps, \mu + \eps) := \cup_{|t| < \eps} M_n(\mu -
t) \]
is a smooth manifold fibering over $(\mu - \eps, \mu + \eps)$, and
$\sigma_{jk}^2$ acts on $\ti{M}_n(\mu - \eps,\mu + \eps)$ preserving
each fiber $M_n(\mu - t)$.  Trivializing the fibration gives a smooth
family of diffeomorphisms $\varphi_t: M_n(\mu) \to M_n(\mu - t)$ with
$\varphi_0 = \on{Id}$, and
\[ t \mapsto \varphi_t^{-1} \circ \sigma_{jk}^2 \circ \varphi_t: M_n(\mu) \to M_n(\mu) \]
is a smooth isotopy from $\sigma_{jk}^2$ to the conjugate of
$\sigma_{jk}^2 | M_n(\mu - t)$.  Thus it suffices to show that
$\sigma_{jk}^2 | M_n(\mu')$ is smoothly isotopic to the identity for
some $\mu'$ in the same chamber as $\mu$.  Since the chambers are
open, we may take $\mu' \neq \qq$.  If $\mu \neq \qq$ we take $\mu' =
\mu$, and the first step is not needed.

This has the following consequence for the image of the 
Goldman function used to generate the Dehn twist:  For $\mu' \neq \qq$ the image of the Goldman function
$h_{jk}$ on $M_n(\mu')$ is contained in $[0,\min(2\mu', 1 - 2\mu')]$,
and in particular in $[0,1/2)$.  Indeed, the holonomy around
$\gamma_{jk}$ is the product $g_j g_k$ of the holonomies around $x_j$
and $x_k$, both conjugate to $\exp(\diag(\pm 2\pi i \mu'))$.  In the
bi-invariant metric on $SU(2) \cong S^3$ the element $\exp(\diag(\pm 2
\pi i \lambda))$, $\lambda \in [0,1/2]$, has distance $2\pi\lambda$
from the identity.  By the triangle inequality,
\[ d(1, g_j g_k) \le d(1, g_j) + d(g_j, g_j g_k) = 4 \pi \mu' .\]
Applying the same inequality to $g_j g_k = (-g_j)(-g_k)$, where $-g_j$
and $-g_k$ are conjugate to $\exp(\diag(\pm 2\pi i (1/2 - \mu')))$,
gives $d(1, g_jg_k) \le 4\pi (1/2 - \mu')$.  This proves the claim;
see \cite{ag:ei} for the general description of products of conjugacy
classes.  For $\mu' = \qq$ the value $h_{jk} = 1/2$ is attained, which
is why the first step is needed in that case.

It follows that the squared Goldman function $h_{jk}^2$ is smooth on all of the perturbed moduli space $M_n(\mu')$.  Indeed,
$h_{jk}$ is smooth on $h_{jk}^{-1}((0,1/2))$, since the function
assigning to an element of $SU(2) \ssm \{ \pm 1 \}$ its eigenvalue
parameter is smooth there.  By Remark \ref{rem:smoothrem}, $h_{jk}^2$ is
smooth near $h_{jk}^{-1}(0)$, and by the claim $h_{jk}^{-1}(1/2)$ is
empty.

By Remark \ref{rem:smoothrem}, the squared Dehn twist $\sigma_{jk}^2$ acts on the dense open
subset $U = h_{jk}^{-1}((0,1/2))$ by the time-one flow of
\[ 2h_{jk}(h_{jk} + 1/4) = 2h_{jk}^2 + h_{jk}/2 . \]
The functions
$h_{jk}^2$ and $h_{jk}$ Poisson commute.  The time-one flow of
$h_{jk}/2$ is the action of $-\on{Id}$ by conjugation on the
holonomies $g_j,g_k$, and so trivial, see \cite[Section 3.4]{wo:ex}.
Hence $\sigma_{jk}^2$ agrees on $U$ with the time-one flow $\psi_1$
of the smooth function $2h_{jk}^2$.  Since both are continuous on
$M_n(\mu')$ and $U$ is dense, $\sigma_{jk}^2 = \psi_1$ on all of
$M_n(\mu')$.  The time-one flows $\psi_s$ of $2 s h_{jk}^2$, $s \in
[0,1]$, then give a smooth isotopy from the identity to
$\sigma_{jk}^2 | M_n(\mu')$.  Combining this isotopy with the isotopy
in the first paragraph gives a smooth isotopy from $\sigma_{jk}^2$ to
the identity in $\Diff(M_n(\mu))$.\end{proof}

\begin{remark}
The isotopy $\psi_s$ used in the proof of the Proposition is a Hamiltonian isotopy of $M_n(\mu')$.  This
does not contradict Theorem \ref{thm:nontrivthm}, since the conjugation by
$\varphi_t$ in the first step is not symplectic.  At $\mu = \qq$
itself, $\sigma_{jk}^2$ is not the time-one flow of any smooth function
of $h_{jk}$: since the flow of $h_{jk}$ has period $1/2$ on $U$, such a
function $F$ would satisfy $F' \equiv 4s$ mod $\frac{1}{2}\Z$ on
$(0,1/2)$, hence $F'(s) = 4s + c$ for a constant $c$, while
smoothness near $h_{jk}^{-1}(0)$ and $h_{jk}^{-1}(1/2)$ requires
$F'(0) = F'(1/2) = 0$.
\end{remark}

\begin{proof}[Proof of Theorem \ref{thm:braid}]
The inequality $[\varphi] \neq [\on{Id}] \in \Map(M(X,\ul{x}),\omega)$
follows immediately from Theorem \ref{thm:nontrivthm} since Hamiltonian
isotopy implies quasiisomorphism in the Fukaya category.  The second
equality $[\varphi] = [\on{Id}] \in \Map(M(X,\ul{x}))$
follows from Proposition \ref{prop:sisot}.
\end{proof} 

Theorem \ref{thm:braid} shows that the homomorphism from the braid group
to the symplectic mapping class group of the moduli space of bundles
does not factor through the symmetric group.  It would be interesting
to know for which labels this homomorphism factors through the
symmetric group and to identify the kernel and image.  In the case
without labels, M. Callahan (unpublished) announced a similar result
for a separating Dehn twist of a genus two surface, in the moduli
space of fixed-determinant bundles of rank two and degree one.
Callahan's result together with the results of this paper would imply
that a separating Dehn twist is not symplectically isotopic to the
identity in any genus.  Analogous results for surfaces without
markings are proved in I. Smith \cite{sm:qu}.  See Keating
\cite{keating:dehn} for related results.

\section{Floer field theory for graphs}

In this section we describe an extension to functors for graphs in
trivial bordisms.  Graphs naturally arise in the {\em surgery exact
  triangle} for higher-rank tangle functors, see \cite{wo:ex}.
  Related invariants in instanton knot homology appear in 
  Kronheimer-Mrowka \cite{km:su3}. 

\subsection{Cerf decompositions for graphs}

First we explain what we mean by a bordism with graph. Let $Y$ be a
bordism.  By a {\em graph} in $Y$ we mean a union $\Gamma$ of closed
{\em edges} $\Edge(\Gamma)$ meeting only at endpoints, the {\em
  vertices} $\Ver(\Gamma)$ of the graph, so that
\begin{itemize} 
\item the valence one vertices of $\Gamma$ are contained in the boundary
  of $Y$:
\[ ( v \in \Ver(\Gamma), |v| = 1) \implies v \in \partial Y .\]
\item the valence greater-than-one vertices of $\Gamma$ to the
  interior of $Y$:
\[ ( v \in \Ver(\Gamma), |v| > 1) \implies v \in Y \ssm \partial Y
  .\] 
\item the interior of each edge is contained in the interior of $Y$:
\[ e \in \Edge(\Gamma) \implies (e \ssm \partial e \subset Y \ssm
  \partial Y) .\]
\end{itemize} 
By a {\em graph bordism} from $(X_-,\ul{x}_-)$ to $(X_+,\ul{x}_+)$ we
mean a bordism $(Y,\phi)$ equipped with a graph $\Gamma$ so that
$\phi$ restricts to an identification
\[\phi |_{ \Gamma \cap \partial Y}: \Gamma \cap \partial Y \to
\ul{x}_- \cup \ul{x}_+ \] 
with orientation on the first factor reversed.  An {\em equivalence}
of graph bordisms is an orientation-preserving diffeomorphism
$(Y_1,\Gamma_1) \to (Y_2,\Gamma_2)$ inducing the identity on the
incoming and outgoing boundary components $(X_\pm,\ul{x}_\pm)$.

Cylindrical graphs are graphs in a fixed product of a surface with an
interval, as for tangles in Definition \ref{def:tancat}.  Let $X$ be a
closed, connected, oriented surface.  An {\em $X$-cylindrical graph}
from a marking $\ul{x}_-$ to a marking $\ul{x}_+$ of $X$ is an oriented
graph $\Gamma \subset [-1,1] \times X$ as above which is {\em vertical
  near the boundary}, in the sense that for some $\eps > 0$
\[ \Gamma \cap ([-1,-1+\eps) \times X) = [-1,-1+\eps) \times
\ul{x}_-, \quad \Gamma \cap ((1 - \eps,1] \times X) = (1-\eps,1]
\times \ul{x}_+ .\]
Two $X$-cylindrical graphs $\Gamma_0,\Gamma_1$ are {\em isotopic} if
$\Gamma_1 = \Psi(\Gamma_0)$ for a diffeomorphism $\Psi$ of $[-1,1]
\times X$ which is isotopic to the identity through diffeomorphisms
equal to the identity near the boundary, preserving orientations of
edges and labels.  Isotopic cylindrical graphs are equivalent graph
bordisms.

Next we define Cerf decompositions for graphs.  Let $(Y,\Gamma,\phi)$
consist of a trivial bordism $Y \cong [b_-,b_+] \times X$ of a closed,
connected, oriented surface $X$ and an oriented graph $\Gamma$ in $Y$.
A {\em cylindrical Morse datum} for $(Y,\Gamma,\phi)$ consists of a
pair $(f,\ul{b})$ consisting of a smooth function $f: Y \to \R$ and a
collection
\[\ul{b} = ( b_- = b_0 < \ldots < b_m = b_+) \] 
of real numbers such that the following hold:
\begin{enumerate} 
\item each $f^{-1}(b_j)$ contains no critical points of $f |_\Gamma$
  or interior vertices:
\[ f^{-1}(b_j) \cap \Crit(f) = f^{-1}(b_j) \cap \Ver(\Gamma) =
  \emptyset ;\]
\item each $f^{-1}(b_{k-1},b_k)$ contains at most one critical point
  of $f |_\Gamma$ or vertex of $\Gamma$:
\[ f^{-1}(b_{k-1},b_k) \cap ( \Crit(f |_\Gamma) \cup \Ver(\Gamma))
  \leq 1 ;\]
\item the sets $\{ b_+ \} \times X$ resp. $\{ b_- \} \times X$ are the
  set of maxima resp. minima of $f$;
\item the Morse function $f$ is cylindrical in the sense that
  $\partial_t f(t,x)>0$ for all $(t,x)\in Y$;
\item the function $f$ restricts to a Morse function on each edge $e$
  of $\Gamma$:
\[ (\d_y(f| e) = 0) \implies (\d^2_y(f|e) \neq 0), \quad \forall y \in
  \on{int}(e) ;\]
\item the restriction of $f$ to any edge has critical points only on
  the interior of the edge:
\[ \Crit(f |e) \cap \partial e = \emptyset, \quad \forall e \subset
  \Gamma ;\]
and
\item the restriction $f |_\Gamma$ is injective on the union of the
  critical set of $f|_\Gamma$ and the set of valence-greater-than-one
  vertices of $\Gamma$:
\[ f|_\Gamma : \Crit(f_\Gamma) \cup \Ver(\Gamma)  \hookrightarrow \R .\]
\end{enumerate} 
Any cylindrical Morse datum $(f,\ul{b})$ of $(Y,\Gamma,\phi)$ gives
rise to a {\em cylindrical Cerf decomposition} of $(Y,\Gamma,\phi)$
into {\em elementary bordisms-with-graphs}
\[(Y_j := f^{-1}([b_{j-1},b_j]), \Gamma_j:= Y_j \cap\Gamma, \phi_j), \quad j =
  1,\ldots, m .\]
That is, each $Y_j$ is cylindrical and $\Gamma_j$ has at most one
critical point or vertex.

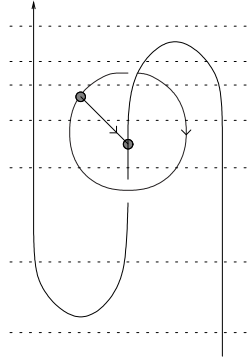
\begin{figure}[h]
\setlength{\unitlength}{0.00027489in}
\begingroup\makeatletter\ifx\SetFigFont\undefined%
\gdef\SetFigFont#1#2#3#4#5{%
  \reset@font\fontsize{#1}{#2pt}%
  \fontfamily{#3}\fontseries{#4}\fontshape{#5}%
  \selectfont}%
\fi\endgroup%
{\renewcommand{\dashlinestretch}{30}
\begin{picture}(4524,6789)(0,-10)
\path(2262,5412)(2261,5412)(2256,5412)
	(2244,5412)(2226,5412)(2204,5412)
	(2180,5412)(2158,5412)(2138,5412)
	(2120,5412)(2104,5412)(2089,5412)
	(2075,5412)(2062,5412)(2050,5412)
	(2036,5412)(2022,5411)(2007,5410)
	(1991,5408)(1974,5406)(1957,5402)
	(1939,5398)(1920,5394)(1901,5388)
	(1881,5381)(1861,5374)(1840,5365)
	(1825,5358)(1809,5351)(1792,5343)
	(1774,5334)(1755,5324)(1735,5313)
	(1715,5301)(1694,5288)(1672,5275)
	(1650,5261)(1628,5246)(1606,5231)
	(1585,5215)(1563,5199)(1543,5183)
	(1523,5166)(1503,5150)(1485,5133)
	(1467,5117)(1449,5100)(1434,5084)
	(1419,5068)(1404,5051)(1389,5033)
	(1374,5014)(1359,4995)(1344,4975)
	(1330,4953)(1316,4932)(1302,4909)
	(1288,4886)(1275,4862)(1263,4838)
	(1251,4813)(1241,4789)(1230,4764)
	(1221,4739)(1212,4715)(1204,4690)
	(1197,4665)(1190,4640)(1184,4615)
	(1179,4592)(1174,4568)(1170,4543)
	(1166,4517)(1162,4491)(1159,4463)
	(1156,4435)(1154,4407)(1152,4377)
	(1151,4348)(1150,4317)(1150,4287)
	(1150,4257)(1151,4226)(1152,4197)
	(1154,4167)(1156,4139)(1159,4111)
	(1162,4083)(1166,4057)(1170,4031)
	(1174,4006)(1179,3982)(1184,3959)
	(1190,3934)(1197,3909)(1204,3884)
	(1212,3859)(1221,3835)(1230,3810)
	(1241,3785)(1251,3761)(1263,3736)
	(1275,3712)(1288,3688)(1302,3665)
	(1316,3642)(1330,3621)(1344,3599)
	(1359,3579)(1374,3560)(1389,3541)
	(1404,3523)(1419,3506)(1434,3490)
	(1449,3474)(1465,3459)(1481,3444)
	(1498,3429)(1516,3414)(1535,3399)
	(1554,3384)(1574,3369)(1596,3355)
	(1617,3341)(1640,3327)(1663,3313)
	(1687,3300)(1711,3288)(1736,3276)
	(1760,3266)(1785,3255)(1810,3246)
	(1834,3237)(1859,3229)(1884,3222)
	(1909,3215)(1934,3209)(1957,3204)
	(1981,3199)(2006,3195)(2032,3191)
	(2058,3187)(2086,3184)(2114,3181)
	(2142,3179)(2172,3177)(2201,3176)
	(2232,3175)(2262,3175)(2292,3175)
	(2323,3176)(2352,3177)(2382,3179)
	(2410,3181)(2438,3184)(2466,3187)
	(2492,3191)(2518,3195)(2543,3199)
	(2567,3204)(2590,3209)(2615,3215)
	(2640,3222)(2665,3229)(2690,3237)
	(2714,3246)(2739,3255)(2764,3266)
	(2788,3276)(2813,3288)(2837,3300)
	(2861,3313)(2884,3327)(2907,3341)
	(2928,3355)(2950,3369)(2970,3384)
	(2989,3399)(3008,3414)(3026,3429)
	(3043,3444)(3059,3459)(3075,3474)
	(3090,3490)(3105,3506)(3120,3523)
	(3135,3541)(3150,3560)(3165,3579)
	(3180,3599)(3194,3621)(3208,3642)
	(3222,3665)(3236,3688)(3249,3712)
	(3261,3736)(3273,3761)(3283,3785)
	(3294,3810)(3303,3835)(3312,3859)
	(3320,3884)(3327,3909)(3334,3934)
	(3340,3959)(3345,3982)(3350,4006)
	(3354,4031)(3358,4057)(3362,4083)
	(3365,4111)(3368,4139)(3370,4167)
	(3372,4197)(3373,4226)(3374,4257)
	(3374,4287)(3374,4317)(3373,4348)
	(3372,4377)(3370,4407)(3368,4435)
	(3365,4463)(3362,4491)(3358,4517)
	(3354,4543)(3350,4568)(3345,4592)
	(3340,4615)(3334,4640)(3327,4665)
	(3320,4690)(3312,4715)(3303,4739)
	(3294,4764)(3283,4789)(3273,4813)
	(3261,4838)(3249,4862)(3236,4886)
	(3222,4909)(3208,4932)(3194,4953)
	(3180,4975)(3165,4995)(3150,5014)
	(3135,5033)(3120,5051)(3105,5068)
	(3090,5084)(3075,5100)(3057,5117)
	(3039,5133)(3021,5150)(3001,5166)
	(2981,5183)(2961,5199)(2939,5215)
	(2918,5231)(2896,5246)(2874,5261)
	(2852,5275)(2830,5288)(2809,5301)
	(2789,5313)(2769,5324)(2750,5334)
	(2732,5343)(2715,5351)(2699,5358)
	(2684,5365)(2663,5374)(2643,5381)
	(2624,5387)(2604,5392)(2584,5397)
	(2562,5401)(2540,5404)(2517,5407)
	(2495,5409)(2477,5410)(2463,5411)
	(2454,5412)(2451,5412)(2450,5412)
\put(1362,4962){\shade\ellipse{180}{180}}
\put(1362,4962){\ellipse{180}{180}}
\put(2262,4062){\shade\ellipse{180}{180}}
\put(2262,4062){\ellipse{180}{180}}
\path(1362,4962)(2262,4062)
\dashline{60.000}(12,5187)(4512,5187)
\dashline{60.000}(12,5637)(4512,5637)
\dashline{60.000}(12,4512)(4512,4512)
\dashline{60.000}(12,1812)(4512,1812)
\dashline{60.000}(12,3612)(4512,3612)
\dashline{60.000}(12,6312)(4512,6312)
\dashline{60.000}(12,462)(4512,462)
\path(3462,4328)(3372,4238)(3282,4328)
\path(2042,4402)(2042,4275)(1915,4275)
\path(492.000,6642.000)(462.000,6762.000)(432.000,6642.000)
\path(462,6762)(462,6744)(462,6732)
	(462,6715)(462,6694)(462,6667)
	(462,6634)(462,6596)(462,6552)
	(462,6501)(462,6445)(462,6382)
	(462,6314)(462,6241)(462,6162)
	(462,6079)(462,5992)(462,5901)
	(462,5807)(462,5710)(462,5611)
	(462,5511)(462,5410)(462,5308)
	(462,5206)(462,5105)(462,5004)
	(462,4904)(462,4806)(462,4710)
	(462,4615)(462,4522)(462,4432)
	(462,4344)(462,4259)(462,4176)
	(462,4096)(462,4018)(462,3942)
	(462,3870)(462,3799)(462,3731)
	(462,3666)(462,3602)(462,3541)
	(462,3482)(462,3425)(462,3370)
	(462,3316)(462,3265)(462,3215)
	(462,3166)(462,3119)(462,3074)
	(462,3029)(462,2986)(462,2944)
	(462,2902)(462,2862)(462,2798)
	(462,2736)(462,2676)(462,2617)
	(462,2560)(463,2504)(463,2450)
	(463,2397)(464,2345)(465,2294)
	(466,2245)(467,2197)(469,2150)
	(470,2104)(472,2060)(474,2017)
	(477,1975)(480,1934)(483,1895)
	(486,1857)(490,1821)(494,1785)
	(499,1751)(503,1718)(509,1687)
	(514,1657)(520,1627)(527,1599)
	(533,1572)(540,1546)(548,1521)
	(556,1497)(564,1473)(573,1450)
	(582,1428)(591,1406)(601,1384)
	(612,1362)(625,1336)(639,1310)
	(654,1284)(670,1258)(687,1232)
	(705,1206)(724,1180)(744,1153)
	(764,1127)(786,1101)(809,1075)
	(832,1050)(856,1025)(881,1000)
	(907,976)(933,954)(959,932)
	(986,911)(1014,891)(1041,872)
	(1068,855)(1096,839)(1123,824)
	(1151,811)(1178,800)(1205,790)
	(1232,781)(1258,774)(1284,769)
	(1310,765)(1336,763)(1362,762)
	(1386,763)(1411,765)(1435,768)
	(1460,773)(1485,780)(1510,787)
	(1535,797)(1560,807)(1586,820)
	(1612,833)(1638,849)(1664,865)
	(1690,883)(1715,903)(1741,924)
	(1766,946)(1791,969)(1816,994)
	(1840,1019)(1863,1046)(1886,1073)
	(1909,1101)(1930,1130)(1951,1159)
	(1971,1189)(1990,1219)(2008,1250)
	(2025,1281)(2042,1313)(2057,1344)
	(2072,1376)(2086,1409)(2100,1441)
	(2112,1474)(2123,1505)(2133,1535)
	(2142,1566)(2151,1598)(2159,1631)
	(2167,1666)(2175,1701)(2182,1738)
	(2188,1776)(2194,1816)(2200,1859)
	(2206,1903)(2211,1949)(2216,1997)
	(2220,2048)(2225,2101)(2229,2155)
	(2233,2212)(2236,2270)(2240,2329)
	(2243,2389)(2246,2449)(2248,2509)
	(2251,2567)(2253,2623)(2255,2676)
	(2256,2725)(2258,2770)(2259,2809)
	(2260,2843)(2261,2872)(2261,2894)
	(2261,2911)(2262,2923)(2262,2931)
	(2262,2935)(2262,2937)
\path(2262,3387)(2262,3389)(2262,3393)
	(2262,3402)(2262,3414)(2262,3432)
	(2262,3454)(2262,3481)(2262,3512)
	(2262,3546)(2262,3583)(2262,3622)
	(2262,3661)(2262,3701)(2262,3740)
	(2262,3778)(2262,3815)(2262,3851)
	(2262,3885)(2262,3919)(2262,3952)
	(2262,3984)(2262,4015)(2262,4047)
	(2262,4078)(2262,4110)(2262,4142)
	(2262,4174)(2262,4202)(2262,4230)
	(2262,4259)(2262,4289)(2262,4319)
	(2263,4350)(2263,4383)(2264,4415)
	(2265,4449)(2266,4484)(2268,4519)
	(2269,4555)(2271,4592)(2274,4629)
	(2276,4666)(2279,4704)(2283,4742)
	(2287,4780)(2291,4819)(2296,4857)
	(2301,4894)(2307,4932)(2313,4969)
	(2319,5006)(2326,5042)(2334,5077)
	(2342,5112)(2350,5146)(2359,5179)
	(2368,5212)(2378,5244)(2389,5276)
	(2400,5306)(2412,5337)(2425,5369)
	(2439,5401)(2454,5432)(2470,5464)
	(2487,5495)(2505,5526)(2524,5556)
	(2544,5587)(2564,5617)(2586,5647)
	(2609,5676)(2632,5704)(2656,5732)
	(2681,5759)(2707,5785)(2733,5810)
	(2759,5834)(2786,5856)(2814,5878)
	(2841,5897)(2868,5916)(2896,5933)
	(2923,5948)(2951,5961)(2978,5973)
	(3005,5984)(3032,5992)(3058,5999)
	(3084,6005)(3110,6009)(3136,6011)
	(3162,6012)(3188,6011)(3214,6009)
	(3240,6005)(3266,6000)(3292,5993)
	(3319,5984)(3346,5974)(3373,5963)
	(3401,5950)(3428,5935)(3456,5919)
	(3483,5902)(3510,5883)(3538,5863)
	(3565,5842)(3591,5820)(3617,5798)
	(3643,5774)(3668,5749)(3692,5724)
	(3715,5699)(3738,5673)(3760,5647)
	(3780,5621)(3800,5594)(3819,5568)
	(3837,5542)(3854,5516)(3870,5490)
	(3885,5464)(3899,5438)(3912,5412)
	(3923,5390)(3933,5368)(3942,5346)
	(3951,5324)(3960,5301)(3968,5277)
	(3976,5253)(3984,5228)(3991,5202)
	(3997,5175)(4004,5147)(4010,5117)
	(4015,5087)(4021,5056)(4025,5023)
	(4030,4989)(4034,4953)(4038,4917)
	(4041,4879)(4044,4840)(4047,4799)
	(4050,4757)(4052,4714)(4054,4670)
	(4055,4624)(4057,4577)(4058,4529)
	(4059,4480)(4060,4429)(4061,4377)
	(4061,4324)(4061,4270)(4062,4214)
	(4062,4157)(4062,4098)(4062,4038)
	(4062,3976)(4062,3912)(4062,3872)
	(4062,3830)(4062,3788)(4062,3745)
	(4062,3700)(4062,3655)(4062,3608)
	(4062,3559)(4062,3509)(4062,3458)
	(4062,3404)(4062,3349)(4062,3292)
	(4062,3233)(4062,3172)(4062,3108)
	(4062,3043)(4062,2975)(4062,2904)
	(4062,2832)(4062,2756)(4062,2678)
	(4062,2598)(4062,2515)(4062,2430)
	(4062,2342)(4062,2252)(4062,2159)
	(4062,2064)(4062,1968)(4062,1870)
	(4062,1770)(4062,1669)(4062,1568)
	(4062,1466)(4062,1364)(4062,1263)
	(4062,1163)(4062,1064)(4062,967)
	(4062,873)(4062,782)(4062,695)
	(4062,612)(4062,533)(4062,460)
	(4062,392)(4062,329)(4062,273)
	(4062,222)(4062,178)(4062,140)
	(4062,107)(4062,80)(4062,59)
	(4062,42)(4062,30)(4062,21)
	(4062,16)(4062,13)(4062,12)
\end{picture}
}
\caption{Cerf decomposition of a graph}
\label{fig:graphhandle} 
\end{figure}

\begin{theorem} \label{thm:cerfgraph} {\rm (Cerf theory for graphs)}
  Let $\Gamma_0$ and $\Gamma_1$ be isotopic $X$-cylindrical graphs.
  Then any cylindrical Cerf decomposition of $\Gamma_0$ and any
  cylindrical Cerf decomposition of $\Gamma_1$ are related, up to
  equivalence of decompositions as in Definition
  \ref{def:handletangle}, by a finite sequence of the following moves
  and their inverses.  Here $[\Gamma_i]$, $[\Gamma_{i+1}]$ denote
  consecutive elementary graphs in a decomposition, meeting along the
  intermediate marking $\ul{x}_i = \Gamma \cap f^{-1}(b_i)$.
\begin{enumerate}
\item \label{gcancel} {\rm (Critical point cancellation)} Two
  elementary graphs $[\Gamma_i]$ and $[\Gamma_{i+1}]$ carrying a local
  minimum $y_i \in \Gamma_i$ resp.\ local maximum $y_{i+1} \in
  \Gamma_{i+1}$ in the interior of the same edge of $\Gamma_i \cup
  \Gamma_{i+1}$,
\[ \d^2_{y_i} f|_e > 0, \quad \d^2_{y_{i+1}} f|_e < 0 , \]
  or a local maximum and local minimum in that order, are replaced by
  the elementary graph $[\Gamma_i] \circ [\Gamma_{i+1}]$, which has no
  critical points or vertices.  The inverse move is {\em critical
    point creation}.
\item \label{greversal} {\rm (Reversals)} Two elementary graphs
  $[\Gamma_i], [\Gamma_{i+1}]$ carrying points $y_i \in \Gamma_i$,
  $y_{i+1} \in \Gamma_{i+1}$, each of which is either a critical point
  of $f|_\Gamma$ or a vertex of $\Gamma$, such that the edges through
  $y_i$ and the edges through $y_{i+1}$ meet $\ul{x}_i$ in disjoint
  sets, are replaced by two elementary graphs $[\Gamma_i'],
  [\Gamma_{i+1}']$ carrying $y_{i+1}$ and $y_i$ in this order, such
  that
\[ [\Gamma_i] \circ [\Gamma_{i+1}] = [\Gamma_i'] \circ [\Gamma_{i+1}'] . \]
  There are three cases: a {\em critical point/critical point
    reversal}, as in Theorem \ref{thm:cerftangle}
  \eqref{reversal}; a {\em critical point/vertex reversal}, as in
  Figure \ref{fig:critvert}; and a {\em vertex/vertex reversal}, as
  in Figure \ref{fig:vertvert}.
\item \label{gvertcrit} {\rm (Vertex/critical point cancellation)} Two
  elementary graphs $[\Gamma_i], [\Gamma_{i+1}]$, one of which carries
  a vertex $v$ and the other a critical point $y$ of $f|_e$ on an edge
  $e$ incident to $v$, such that the segment of $e$ between $v$ and
  $y$ contains no other critical point, are replaced by the elementary
  graph $[\Gamma_i] \circ [\Gamma_{i+1}]$.  This elementary graph
  carries the single vertex $v$, at which the edge $e$ now leaves in
  the opposite direction: an edge at $v$ going down and turning up at
  a minimum becomes an edge going up from $v$, and similarly for a
  maximum.  See Figures \ref{fig:vertcritcan} and \ref{fig:another}.
  The inverse move is {\em vertex/critical point creation}.
\item \label{ggraphglue} {\rm (Gluing)} Two elementary graphs
  $[\Gamma_i], [\Gamma_{i+1}]$, one of which has no critical points or
  vertices, and so is the class of a braid, are replaced by the
  elementary graph $[\Gamma_i] \circ [\Gamma_{i+1}]$.  The inverse
  move splits off a braid.
\end{enumerate}
\end{theorem}

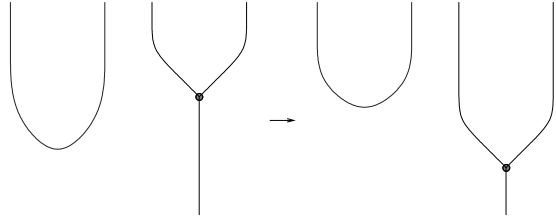
\begin{figure}[ht]
\setlength{\unitlength}{0.00027489in}
\begingroup\makeatletter\ifx\SetFigFont\undefined%
\gdef\SetFigFont#1#2#3#4#5{%
  \reset@font\fontsize{#1}{#2pt}%
  \fontfamily{#3}\fontseries{#4}\fontshape{#5}%
  \selectfont}%
\fi\endgroup%
{\renewcommand{\dashlinestretch}{30}
\begin{picture}(10374,4089)(0,-10)
\path(12,4062)(12,4060)(12,4055)
	(12,4047)(12,4034)(12,4016)
	(12,3992)(12,3962)(12,3928)
	(12,3889)(12,3846)(12,3800)
	(12,3752)(12,3703)(12,3653)
	(12,3604)(12,3556)(12,3510)
	(12,3465)(12,3421)(12,3380)
	(12,3340)(12,3302)(12,3266)
	(12,3231)(12,3198)(12,3165)
	(12,3134)(12,3103)(12,3072)
	(12,3042)(12,3012)(12,2982)
	(12,2952)(12,2921)(12,2890)
	(13,2859)(13,2827)(14,2795)
	(15,2763)(16,2729)(17,2696)
	(19,2662)(21,2627)(23,2593)
	(26,2558)(29,2522)(33,2487)
	(37,2452)(42,2416)(47,2381)
	(52,2347)(58,2312)(65,2278)
	(72,2245)(80,2211)(88,2179)
	(97,2147)(106,2115)(116,2084)
	(127,2053)(138,2022)(149,1992)
	(162,1962)(174,1934)(188,1905)
	(202,1877)(217,1848)(232,1820)
	(249,1791)(266,1762)(284,1733)
	(303,1704)(323,1675)(344,1647)
	(365,1619)(388,1591)(411,1563)
	(434,1537)(458,1511)(483,1486)
	(508,1461)(533,1438)(559,1416)
	(584,1396)(610,1376)(636,1358)
	(662,1342)(688,1327)(714,1313)
	(739,1301)(764,1291)(789,1282)
	(814,1275)(839,1269)(863,1265)
	(888,1263)(912,1262)(936,1263)
	(961,1265)(985,1269)(1010,1275)
	(1035,1282)(1060,1291)(1085,1301)
	(1110,1313)(1136,1327)(1162,1342)
	(1188,1358)(1214,1376)(1240,1396)
	(1265,1416)(1291,1438)(1316,1461)
	(1341,1486)(1366,1511)(1390,1537)
	(1413,1563)(1436,1591)(1459,1619)
	(1480,1647)(1501,1675)(1521,1704)
	(1540,1733)(1558,1762)(1575,1791)
	(1592,1820)(1607,1848)(1622,1877)
	(1636,1905)(1650,1934)(1662,1962)
	(1675,1992)(1686,2022)(1697,2053)
	(1708,2084)(1718,2115)(1727,2147)
	(1736,2179)(1744,2211)(1752,2245)
	(1759,2278)(1766,2312)(1772,2347)
	(1777,2381)(1782,2416)(1787,2452)
	(1791,2487)(1795,2522)(1798,2558)
	(1801,2593)(1803,2627)(1805,2662)
	(1807,2696)(1808,2729)(1809,2763)
	(1810,2795)(1811,2827)(1811,2859)
	(1812,2890)(1812,2921)(1812,2952)
	(1812,2982)(1812,3012)(1812,3042)
	(1812,3072)(1812,3103)(1812,3134)
	(1812,3165)(1812,3198)(1812,3231)
	(1812,3266)(1812,3302)(1812,3340)
	(1812,3380)(1812,3421)(1812,3465)
	(1812,3510)(1812,3556)(1812,3604)
	(1812,3653)(1812,3703)(1812,3752)
	(1812,3800)(1812,3846)(1812,3889)
	(1812,3928)(1812,3962)(1812,3992)
	(1812,4016)(1812,4034)(1812,4047)
	(1812,4055)(1812,4060)(1812,4062)
\path(5862,4062)(5862,4059)(5862,4053)
	(5862,4043)(5862,4026)(5862,4004)
	(5862,3976)(5862,3943)(5862,3905)
	(5862,3864)(5862,3822)(5862,3778)
	(5862,3735)(5862,3693)(5862,3652)
	(5862,3614)(5862,3577)(5862,3542)
	(5862,3509)(5862,3478)(5862,3448)
	(5862,3420)(5862,3392)(5862,3365)
	(5862,3338)(5862,3312)(5862,3286)
	(5862,3259)(5862,3232)(5863,3205)
	(5863,3177)(5864,3148)(5865,3119)
	(5867,3090)(5869,3060)(5871,3030)
	(5875,2999)(5878,2968)(5883,2937)
	(5888,2906)(5894,2875)(5900,2844)
	(5908,2814)(5916,2784)(5925,2755)
	(5935,2726)(5945,2697)(5957,2669)
	(5969,2642)(5982,2615)(5997,2588)
	(6012,2562)(6025,2541)(6039,2519)
	(6054,2497)(6070,2476)(6087,2454)
	(6105,2432)(6124,2410)(6144,2388)
	(6164,2366)(6186,2344)(6209,2323)
	(6232,2302)(6256,2281)(6281,2261)
	(6307,2241)(6333,2222)(6359,2203)
	(6386,2186)(6414,2169)(6441,2154)
	(6468,2139)(6496,2126)(6523,2114)
	(6551,2103)(6578,2094)(6605,2085)
	(6632,2078)(6658,2072)(6684,2068)
	(6710,2065)(6736,2063)(6762,2062)
	(6788,2063)(6814,2065)(6840,2068)
	(6866,2072)(6892,2078)(6919,2085)
	(6946,2094)(6973,2103)(7001,2114)
	(7028,2126)(7056,2139)(7083,2154)
	(7110,2169)(7138,2186)(7165,2203)
	(7191,2222)(7217,2241)(7243,2261)
	(7268,2281)(7292,2302)(7315,2323)
	(7338,2344)(7360,2366)(7380,2388)
	(7400,2410)(7419,2432)(7437,2454)
	(7454,2476)(7470,2497)(7485,2519)
	(7499,2541)(7512,2562)(7527,2588)
	(7542,2615)(7555,2642)(7567,2669)
	(7579,2697)(7589,2726)(7599,2755)
	(7608,2784)(7616,2814)(7624,2844)
	(7630,2875)(7636,2906)(7641,2937)
	(7646,2968)(7649,2999)(7653,3030)
	(7655,3060)(7657,3090)(7659,3119)
	(7660,3148)(7661,3177)(7661,3205)
	(7662,3232)(7662,3259)(7662,3286)
	(7662,3312)(7662,3338)(7662,3365)
	(7662,3392)(7662,3420)(7662,3448)
	(7662,3478)(7662,3509)(7662,3542)
	(7662,3577)(7662,3614)(7662,3652)
	(7662,3693)(7662,3735)(7662,3778)
	(7662,3822)(7662,3864)(7662,3905)
	(7662,3943)(7662,3976)(7662,4004)
	(7662,4026)(7662,4043)(7662,4053)
	(7662,4059)(7662,4062)
\put(3612,2262){\shade\ellipse{128}{128}}
\put(3612,2262){\ellipse{128}{128}}
\put(9462,912){\shade\ellipse{128}{128}}
\put(9462,912){\ellipse{128}{128}}
\path(3612,2262)(3612,12)
\path(9462,912)(9462,12)
\path(4962,1812)(5412,1812)
\path(5292.000,1782.000)(5412.000,1812.000)(5292.000,1842.000)
\path(3612,2262)(3610,2264)(3606,2268)
	(3598,2276)(3586,2288)(3569,2305)
	(3549,2325)(3526,2348)(3500,2374)
	(3472,2402)(3444,2430)(3416,2458)
	(3389,2485)(3362,2512)(3337,2537)
	(3314,2560)(3292,2582)(3271,2603)
	(3251,2623)(3232,2642)(3214,2660)
	(3197,2677)(3179,2695)(3162,2712)
	(3145,2729)(3127,2747)(3110,2764)
	(3092,2782)(3074,2800)(3056,2819)
	(3038,2838)(3020,2857)(3001,2877)
	(2983,2897)(2965,2917)(2947,2937)
	(2930,2957)(2913,2977)(2897,2997)
	(2881,3017)(2867,3036)(2853,3055)
	(2840,3074)(2828,3092)(2816,3110)
	(2806,3127)(2796,3145)(2787,3162)
	(2777,3183)(2768,3204)(2760,3225)
	(2753,3247)(2746,3269)(2740,3292)
	(2735,3315)(2730,3339)(2726,3363)
	(2722,3387)(2720,3411)(2717,3435)
	(2716,3459)(2714,3482)(2713,3505)
	(2713,3527)(2712,3549)(2712,3570)
	(2712,3591)(2712,3612)(2712,3633)
	(2712,3654)(2712,3675)(2712,3697)
	(2712,3721)(2712,3746)(2712,3773)
	(2712,3802)(2712,3833)(2712,3866)
	(2712,3900)(2712,3933)(2712,3965)
	(2712,3995)(2712,4019)(2712,4038)
	(2712,4051)(2712,4059)(2712,4062)
\path(3612,2262)(3614,2264)(3618,2268)
	(3626,2276)(3638,2288)(3655,2305)
	(3675,2325)(3698,2348)(3724,2374)
	(3752,2402)(3780,2430)(3808,2458)
	(3835,2485)(3862,2512)(3887,2537)
	(3910,2560)(3932,2582)(3953,2603)
	(3973,2623)(3992,2642)(4010,2660)
	(4027,2677)(4045,2695)(4062,2712)
	(4079,2729)(4097,2747)(4114,2764)
	(4132,2782)(4150,2800)(4168,2819)
	(4186,2838)(4204,2857)(4223,2877)
	(4241,2897)(4259,2917)(4277,2937)
	(4294,2957)(4311,2977)(4327,2997)
	(4343,3017)(4357,3036)(4371,3055)
	(4384,3074)(4396,3092)(4408,3110)
	(4418,3127)(4428,3145)(4437,3162)
	(4447,3183)(4456,3204)(4464,3225)
	(4471,3247)(4478,3269)(4484,3292)
	(4489,3315)(4494,3339)(4498,3363)
	(4502,3387)(4504,3411)(4507,3435)
	(4508,3459)(4510,3482)(4511,3505)
	(4511,3527)(4512,3549)(4512,3570)
	(4512,3591)(4512,3612)(4512,3633)
	(4512,3654)(4512,3675)(4512,3697)
	(4512,3721)(4512,3746)(4512,3773)
	(4512,3802)(4512,3833)(4512,3866)
	(4512,3900)(4512,3933)(4512,3965)
	(4512,3995)(4512,4019)(4512,4038)
	(4512,4051)(4512,4059)(4512,4062)
\path(9462,912)(9460,914)(9456,918)
	(9448,926)(9436,938)(9419,955)
	(9399,975)(9376,998)(9350,1024)
	(9322,1052)(9294,1080)(9266,1108)
	(9239,1135)(9212,1162)(9187,1187)
	(9164,1210)(9142,1232)(9121,1253)
	(9101,1273)(9082,1292)(9064,1310)
	(9047,1327)(9029,1345)(9012,1362)
	(8995,1379)(8977,1397)(8960,1414)
	(8942,1432)(8924,1450)(8906,1469)
	(8888,1488)(8870,1507)(8851,1527)
	(8833,1547)(8815,1567)(8797,1587)
	(8780,1607)(8763,1627)(8747,1647)
	(8731,1667)(8717,1686)(8703,1705)
	(8690,1724)(8678,1742)(8666,1760)
	(8656,1777)(8646,1795)(8637,1812)
	(8629,1828)(8622,1844)(8616,1860)
	(8609,1877)(8604,1895)(8598,1913)
	(8594,1932)(8589,1952)(8585,1973)
	(8581,1996)(8578,2019)(8575,2043)
	(8572,2068)(8570,2095)(8568,2122)
	(8567,2150)(8565,2180)(8564,2210)
	(8564,2241)(8563,2273)(8563,2306)
	(8562,2340)(8562,2375)(8562,2411)
	(8562,2448)(8562,2487)(8562,2515)
	(8562,2543)(8562,2573)(8562,2604)
	(8562,2636)(8562,2670)(8562,2706)
	(8562,2743)(8562,2783)(8562,2825)
	(8562,2869)(8562,2916)(8562,2966)
	(8562,3018)(8562,3073)(8562,3130)
	(8562,3190)(8562,3252)(8562,3316)
	(8562,3382)(8562,3448)(8562,3515)
	(8562,3581)(8562,3646)(8562,3709)
	(8562,3768)(8562,3823)(8562,3873)
	(8562,3918)(8562,3956)(8562,3988)
	(8562,4014)(8562,4033)(8562,4047)
	(8562,4055)(8562,4060)(8562,4062)
\path(9462,912)(9464,914)(9468,918)
	(9476,926)(9488,938)(9505,955)
	(9525,975)(9548,998)(9574,1024)
	(9602,1052)(9630,1080)(9658,1108)
	(9685,1135)(9712,1162)(9737,1187)
	(9760,1210)(9782,1232)(9803,1253)
	(9823,1273)(9842,1292)(9860,1310)
	(9877,1327)(9895,1345)(9912,1362)
	(9929,1379)(9947,1397)(9964,1414)
	(9982,1432)(10000,1450)(10018,1469)
	(10036,1488)(10054,1507)(10073,1527)
	(10091,1547)(10109,1567)(10127,1587)
	(10144,1607)(10161,1627)(10177,1647)
	(10193,1667)(10207,1686)(10221,1705)
	(10234,1724)(10246,1742)(10258,1760)
	(10268,1777)(10278,1795)(10287,1812)
	(10295,1828)(10302,1844)(10308,1860)
	(10315,1877)(10320,1895)(10326,1913)
	(10330,1932)(10335,1952)(10339,1973)
	(10343,1996)(10346,2019)(10349,2043)
	(10352,2068)(10354,2095)(10356,2122)
	(10357,2150)(10359,2180)(10360,2210)
	(10360,2241)(10361,2273)(10361,2306)
	(10362,2340)(10362,2375)(10362,2411)
	(10362,2448)(10362,2487)(10362,2515)
	(10362,2543)(10362,2573)(10362,2604)
	(10362,2636)(10362,2670)(10362,2706)
	(10362,2743)(10362,2783)(10362,2825)
	(10362,2869)(10362,2916)(10362,2966)
	(10362,3018)(10362,3073)(10362,3130)
	(10362,3190)(10362,3252)(10362,3316)
	(10362,3382)(10362,3448)(10362,3515)
	(10362,3581)(10362,3646)(10362,3709)
	(10362,3768)(10362,3823)(10362,3873)
	(10362,3918)(10362,3956)(10362,3988)
	(10362,4014)(10362,4033)(10362,4047)
	(10362,4055)(10362,4060)(10362,4062)
\end{picture}
}
\caption{Critical point/vertex switch}
\label{fig:critvert}
\end{figure}

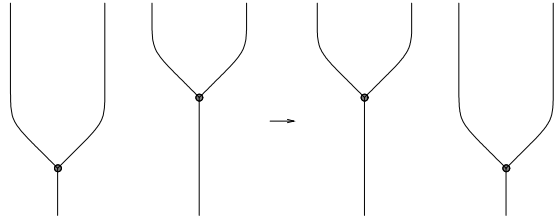
\begin{figure}[ht]
\setlength{\unitlength}{0.00027489in}
\begingroup\makeatletter\ifx\SetFigFont\undefined%
\gdef\SetFigFont#1#2#3#4#5{%
  \reset@font\fontsize{#1}{#2pt}%
  \fontfamily{#3}\fontseries{#4}\fontshape{#5}%
  \selectfont}%
\fi\endgroup%
{\renewcommand{\dashlinestretch}{30}
\begin{picture}(10374,4089)(0,-10)
\put(3612,2262){\shade\ellipse{128}{128}}
\put(3612,2262){\ellipse{128}{128}}
\path(3612,2262)(3612,12)
\path(3612,2262)(3610,2264)(3606,2268)
	(3598,2276)(3586,2288)(3569,2305)
	(3549,2325)(3526,2348)(3500,2374)
	(3472,2402)(3444,2430)(3416,2458)
	(3389,2485)(3362,2512)(3337,2537)
	(3314,2560)(3292,2582)(3271,2603)
	(3251,2623)(3232,2642)(3214,2660)
	(3197,2677)(3179,2695)(3162,2712)
	(3145,2729)(3127,2747)(3110,2764)
	(3092,2782)(3074,2800)(3056,2819)
	(3038,2838)(3020,2857)(3001,2877)
	(2983,2897)(2965,2917)(2947,2937)
	(2930,2957)(2913,2977)(2897,2997)
	(2881,3017)(2867,3036)(2853,3055)
	(2840,3074)(2828,3092)(2816,3110)
	(2806,3127)(2796,3145)(2787,3162)
	(2777,3183)(2768,3204)(2760,3225)
	(2753,3247)(2746,3269)(2740,3292)
	(2735,3315)(2730,3339)(2726,3363)
	(2722,3387)(2720,3411)(2717,3435)
	(2716,3459)(2714,3482)(2713,3505)
	(2713,3527)(2712,3549)(2712,3570)
	(2712,3591)(2712,3612)(2712,3633)
	(2712,3654)(2712,3675)(2712,3697)
	(2712,3721)(2712,3746)(2712,3773)
	(2712,3802)(2712,3833)(2712,3866)
	(2712,3900)(2712,3933)(2712,3965)
	(2712,3995)(2712,4019)(2712,4038)
	(2712,4051)(2712,4059)(2712,4062)
\path(3612,2262)(3614,2264)(3618,2268)
	(3626,2276)(3638,2288)(3655,2305)
	(3675,2325)(3698,2348)(3724,2374)
	(3752,2402)(3780,2430)(3808,2458)
	(3835,2485)(3862,2512)(3887,2537)
	(3910,2560)(3932,2582)(3953,2603)
	(3973,2623)(3992,2642)(4010,2660)
	(4027,2677)(4045,2695)(4062,2712)
	(4079,2729)(4097,2747)(4114,2764)
	(4132,2782)(4150,2800)(4168,2819)
	(4186,2838)(4204,2857)(4223,2877)
	(4241,2897)(4259,2917)(4277,2937)
	(4294,2957)(4311,2977)(4327,2997)
	(4343,3017)(4357,3036)(4371,3055)
	(4384,3074)(4396,3092)(4408,3110)
	(4418,3127)(4428,3145)(4437,3162)
	(4447,3183)(4456,3204)(4464,3225)
	(4471,3247)(4478,3269)(4484,3292)
	(4489,3315)(4494,3339)(4498,3363)
	(4502,3387)(4504,3411)(4507,3435)
	(4508,3459)(4510,3482)(4511,3505)
	(4511,3527)(4512,3549)(4512,3570)
	(4512,3591)(4512,3612)(4512,3633)
	(4512,3654)(4512,3675)(4512,3697)
	(4512,3721)(4512,3746)(4512,3773)
	(4512,3802)(4512,3833)(4512,3866)
	(4512,3900)(4512,3933)(4512,3965)
	(4512,3995)(4512,4019)(4512,4038)
	(4512,4051)(4512,4059)(4512,4062)
\put(9462,912){\shade\ellipse{128}{128}}
\put(9462,912){\ellipse{128}{128}}
\path(9462,912)(9462,12)
\path(9462,912)(9460,914)(9456,918)
	(9448,926)(9436,938)(9419,955)
	(9399,975)(9376,998)(9350,1024)
	(9322,1052)(9294,1080)(9266,1108)
	(9239,1135)(9212,1162)(9187,1187)
	(9164,1210)(9142,1232)(9121,1253)
	(9101,1273)(9082,1292)(9064,1310)
	(9047,1327)(9029,1345)(9012,1362)
	(8995,1379)(8977,1397)(8960,1414)
	(8942,1432)(8924,1450)(8906,1469)
	(8888,1488)(8870,1507)(8851,1527)
	(8833,1547)(8815,1567)(8797,1587)
	(8780,1607)(8763,1627)(8747,1647)
	(8731,1667)(8717,1686)(8703,1705)
	(8690,1724)(8678,1742)(8666,1760)
	(8656,1777)(8646,1795)(8637,1812)
	(8629,1828)(8622,1844)(8616,1860)
	(8609,1877)(8604,1895)(8598,1913)
	(8594,1932)(8589,1952)(8585,1973)
	(8581,1996)(8578,2019)(8575,2043)
	(8572,2068)(8570,2095)(8568,2122)
	(8567,2150)(8565,2180)(8564,2210)
	(8564,2241)(8563,2273)(8563,2306)
	(8562,2340)(8562,2375)(8562,2411)
	(8562,2448)(8562,2487)(8562,2515)
	(8562,2543)(8562,2573)(8562,2604)
	(8562,2636)(8562,2670)(8562,2706)
	(8562,2743)(8562,2783)(8562,2825)
	(8562,2869)(8562,2916)(8562,2966)
	(8562,3018)(8562,3073)(8562,3130)
	(8562,3190)(8562,3252)(8562,3316)
	(8562,3382)(8562,3448)(8562,3515)
	(8562,3581)(8562,3646)(8562,3709)
	(8562,3768)(8562,3823)(8562,3873)
	(8562,3918)(8562,3956)(8562,3988)
	(8562,4014)(8562,4033)(8562,4047)
	(8562,4055)(8562,4060)(8562,4062)
\path(9462,912)(9464,914)(9468,918)
	(9476,926)(9488,938)(9505,955)
	(9525,975)(9548,998)(9574,1024)
	(9602,1052)(9630,1080)(9658,1108)
	(9685,1135)(9712,1162)(9737,1187)
	(9760,1210)(9782,1232)(9803,1253)
	(9823,1273)(9842,1292)(9860,1310)
	(9877,1327)(9895,1345)(9912,1362)
	(9929,1379)(9947,1397)(9964,1414)
	(9982,1432)(10000,1450)(10018,1469)
	(10036,1488)(10054,1507)(10073,1527)
	(10091,1547)(10109,1567)(10127,1587)
	(10144,1607)(10161,1627)(10177,1647)
	(10193,1667)(10207,1686)(10221,1705)
	(10234,1724)(10246,1742)(10258,1760)
	(10268,1777)(10278,1795)(10287,1812)
	(10295,1828)(10302,1844)(10308,1860)
	(10315,1877)(10320,1895)(10326,1913)
	(10330,1932)(10335,1952)(10339,1973)
	(10343,1996)(10346,2019)(10349,2043)
	(10352,2068)(10354,2095)(10356,2122)
	(10357,2150)(10359,2180)(10360,2210)
	(10360,2241)(10361,2273)(10361,2306)
	(10362,2340)(10362,2375)(10362,2411)
	(10362,2448)(10362,2487)(10362,2515)
	(10362,2543)(10362,2573)(10362,2604)
	(10362,2636)(10362,2670)(10362,2706)
	(10362,2743)(10362,2783)(10362,2825)
	(10362,2869)(10362,2916)(10362,2966)
	(10362,3018)(10362,3073)(10362,3130)
	(10362,3190)(10362,3252)(10362,3316)
	(10362,3382)(10362,3448)(10362,3515)
	(10362,3581)(10362,3646)(10362,3709)
	(10362,3768)(10362,3823)(10362,3873)
	(10362,3918)(10362,3956)(10362,3988)
	(10362,4014)(10362,4033)(10362,4047)
	(10362,4055)(10362,4060)(10362,4062)
\put(912,912){\shade\ellipse{128}{128}}
\put(912,912){\ellipse{128}{128}}
\path(912,912)(912,12)
\path(912,912)(910,914)(906,918)
	(898,926)(886,938)(869,955)
	(849,975)(826,998)(800,1024)
	(772,1052)(744,1080)(716,1108)
	(689,1135)(662,1162)(637,1187)
	(614,1210)(592,1232)(571,1253)
	(551,1273)(532,1292)(514,1310)
	(497,1327)(479,1345)(462,1362)
	(445,1379)(427,1397)(410,1414)
	(392,1432)(374,1450)(356,1469)
	(338,1488)(320,1507)(301,1527)
	(283,1547)(265,1567)(247,1587)
	(230,1607)(213,1627)(197,1647)
	(181,1667)(167,1686)(153,1705)
	(140,1724)(128,1742)(116,1760)
	(106,1777)(96,1795)(87,1812)
	(79,1828)(72,1844)(66,1860)
	(59,1877)(54,1895)(48,1913)
	(44,1932)(39,1952)(35,1973)
	(31,1996)(28,2019)(25,2043)
	(22,2068)(20,2095)(18,2122)
	(17,2150)(15,2180)(14,2210)
	(14,2241)(13,2273)(13,2306)
	(12,2340)(12,2375)(12,2411)
	(12,2448)(12,2487)(12,2515)
	(12,2543)(12,2573)(12,2604)
	(12,2636)(12,2670)(12,2706)
	(12,2743)(12,2783)(12,2825)
	(12,2869)(12,2916)(12,2966)
	(12,3018)(12,3073)(12,3130)
	(12,3190)(12,3252)(12,3316)
	(12,3382)(12,3448)(12,3515)
	(12,3581)(12,3646)(12,3709)
	(12,3768)(12,3823)(12,3873)
	(12,3918)(12,3956)(12,3988)
	(12,4014)(12,4033)(12,4047)
	(12,4055)(12,4060)(12,4062)
\path(912,912)(914,914)(918,918)
	(926,926)(938,938)(955,955)
	(975,975)(998,998)(1024,1024)
	(1052,1052)(1080,1080)(1108,1108)
	(1135,1135)(1162,1162)(1187,1187)
	(1210,1210)(1232,1232)(1253,1253)
	(1273,1273)(1292,1292)(1310,1310)
	(1327,1327)(1345,1345)(1362,1362)
	(1379,1379)(1397,1397)(1414,1414)
	(1432,1432)(1450,1450)(1468,1469)
	(1486,1488)(1504,1507)(1523,1527)
	(1541,1547)(1559,1567)(1577,1587)
	(1594,1607)(1611,1627)(1627,1647)
	(1643,1667)(1657,1686)(1671,1705)
	(1684,1724)(1696,1742)(1708,1760)
	(1718,1777)(1728,1795)(1737,1812)
	(1745,1828)(1752,1844)(1758,1860)
	(1765,1877)(1770,1895)(1776,1913)
	(1780,1932)(1785,1952)(1789,1973)
	(1793,1996)(1796,2019)(1799,2043)
	(1802,2068)(1804,2095)(1806,2122)
	(1807,2150)(1809,2180)(1810,2210)
	(1810,2241)(1811,2273)(1811,2306)
	(1812,2340)(1812,2375)(1812,2411)
	(1812,2448)(1812,2487)(1812,2515)
	(1812,2543)(1812,2573)(1812,2604)
	(1812,2636)(1812,2670)(1812,2706)
	(1812,2743)(1812,2783)(1812,2825)
	(1812,2869)(1812,2916)(1812,2966)
	(1812,3018)(1812,3073)(1812,3130)
	(1812,3190)(1812,3252)(1812,3316)
	(1812,3382)(1812,3448)(1812,3515)
	(1812,3581)(1812,3646)(1812,3709)
	(1812,3768)(1812,3823)(1812,3873)
	(1812,3918)(1812,3956)(1812,3988)
	(1812,4014)(1812,4033)(1812,4047)
	(1812,4055)(1812,4060)(1812,4062)
\put(6762,2262){\shade\ellipse{128}{128}}
\put(6762,2262){\ellipse{128}{128}}
\path(6762,2262)(6762,12)
\path(6762,2262)(6760,2264)(6756,2268)
	(6748,2276)(6736,2288)(6719,2305)
	(6699,2325)(6676,2348)(6650,2374)
	(6622,2402)(6594,2430)(6566,2458)
	(6539,2485)(6512,2512)(6487,2537)
	(6464,2560)(6442,2582)(6421,2603)
	(6401,2623)(6382,2642)(6364,2660)
	(6347,2677)(6329,2695)(6312,2712)
	(6295,2729)(6277,2747)(6260,2764)
	(6242,2782)(6224,2800)(6206,2819)
	(6188,2838)(6170,2857)(6151,2877)
	(6133,2897)(6115,2917)(6097,2937)
	(6080,2957)(6063,2977)(6047,2997)
	(6031,3017)(6017,3036)(6003,3055)
	(5990,3074)(5978,3092)(5966,3110)
	(5956,3127)(5946,3145)(5937,3162)
	(5927,3183)(5918,3204)(5910,3225)
	(5903,3247)(5896,3269)(5890,3292)
	(5885,3315)(5880,3339)(5876,3363)
	(5872,3387)(5870,3411)(5867,3435)
	(5866,3459)(5864,3482)(5863,3505)
	(5863,3527)(5862,3549)(5862,3570)
	(5862,3591)(5862,3612)(5862,3633)
	(5862,3654)(5862,3675)(5862,3697)
	(5862,3721)(5862,3746)(5862,3773)
	(5862,3802)(5862,3833)(5862,3866)
	(5862,3900)(5862,3933)(5862,3965)
	(5862,3995)(5862,4019)(5862,4038)
	(5862,4051)(5862,4059)(5862,4062)
\path(6762,2262)(6764,2264)(6768,2268)
	(6776,2276)(6788,2288)(6805,2305)
	(6825,2325)(6848,2348)(6874,2374)
	(6902,2402)(6930,2430)(6958,2458)
	(6985,2485)(7012,2512)(7037,2537)
	(7060,2560)(7082,2582)(7103,2603)
	(7123,2623)(7142,2642)(7160,2660)
	(7177,2677)(7195,2695)(7212,2712)
	(7229,2729)(7247,2747)(7264,2764)
	(7282,2782)(7300,2800)(7318,2819)
	(7336,2838)(7354,2857)(7373,2877)
	(7391,2897)(7409,2917)(7427,2937)
	(7444,2957)(7461,2977)(7477,2997)
	(7493,3017)(7507,3036)(7521,3055)
	(7534,3074)(7546,3092)(7558,3110)
	(7568,3127)(7578,3145)(7587,3162)
	(7597,3183)(7606,3204)(7614,3225)
	(7621,3247)(7628,3269)(7634,3292)
	(7639,3315)(7644,3339)(7648,3363)
	(7652,3387)(7654,3411)(7657,3435)
	(7658,3459)(7660,3482)(7661,3505)
	(7661,3527)(7662,3549)(7662,3570)
	(7662,3591)(7662,3612)(7662,3633)
	(7662,3654)(7662,3675)(7662,3697)
	(7662,3721)(7662,3746)(7662,3773)
	(7662,3802)(7662,3833)(7662,3866)
	(7662,3900)(7662,3933)(7662,3965)
	(7662,3995)(7662,4019)(7662,4038)
	(7662,4051)(7662,4059)(7662,4062)
\path(4962,1812)(5412,1812)
\path(5292.000,1782.000)(5412.000,1812.000)(5292.000,1842.000)
\end{picture}
}
\caption{Vertex/vertex switch}
\label{fig:vertvert}
\end{figure}

\begin{figure}[ht]
\setlength{\unitlength}{0.00027489in}
\begingroup\makeatletter\ifx\SetFigFont\undefined%
\gdef\SetFigFont#1#2#3#4#5{%
  \reset@font\fontsize{#1}{#2pt}%
  \fontfamily{#3}\fontseries{#4}\fontshape{#5}%
  \selectfont}%
\fi\endgroup%
{\renewcommand{\dashlinestretch}{30}
\begin{picture}(6774,4089)(0,-10)
\put(5862,3162){\shade\ellipse{128}{128}}
\put(5862,3162){\ellipse{128}{128}}
\path(5862,3162)(5862,4062)
\path(5862,3162)(5860,3160)(5856,3156)
	(5848,3148)(5836,3136)(5819,3119)
	(5799,3099)(5776,3076)(5750,3050)
	(5722,3022)(5694,2994)(5666,2966)
	(5639,2939)(5612,2912)(5587,2887)
	(5564,2864)(5542,2842)(5521,2821)
	(5501,2801)(5482,2782)(5464,2764)
	(5447,2747)(5429,2729)(5412,2712)
	(5395,2695)(5377,2677)(5360,2660)
	(5342,2642)(5324,2624)(5306,2605)
	(5288,2586)(5270,2567)(5251,2547)
	(5233,2527)(5215,2507)(5197,2487)
	(5180,2467)(5163,2447)(5147,2427)
	(5131,2407)(5117,2388)(5103,2369)
	(5090,2350)(5078,2332)(5066,2314)
	(5056,2297)(5046,2279)(5037,2262)
	(5029,2246)(5022,2230)(5016,2214)
	(5009,2197)(5004,2179)(4998,2161)
	(4994,2142)(4989,2122)(4985,2101)
	(4981,2078)(4978,2055)(4975,2031)
	(4972,2006)(4970,1979)(4968,1952)
	(4967,1924)(4965,1894)(4964,1864)
	(4964,1833)(4963,1801)(4963,1768)
	(4962,1734)(4962,1699)(4962,1663)
	(4962,1626)(4962,1587)(4962,1559)
	(4962,1531)(4962,1501)(4962,1470)
	(4962,1438)(4962,1404)(4962,1368)
	(4962,1331)(4962,1291)(4962,1249)
	(4962,1205)(4962,1158)(4962,1108)
	(4962,1056)(4962,1001)(4962,944)
	(4962,884)(4962,822)(4962,758)
	(4962,692)(4962,626)(4962,559)
	(4962,493)(4962,428)(4962,365)
	(4962,306)(4962,251)(4962,201)
	(4962,156)(4962,118)(4962,86)
	(4962,60)(4962,41)(4962,27)
	(4962,19)(4962,14)(4962,12)
\path(5862,3162)(5864,3160)(5868,3156)
	(5876,3148)(5888,3136)(5905,3119)
	(5925,3099)(5948,3076)(5974,3050)
	(6002,3022)(6030,2994)(6058,2966)
	(6085,2939)(6112,2912)(6137,2887)
	(6160,2864)(6182,2842)(6203,2821)
	(6223,2801)(6242,2782)(6260,2764)
	(6277,2747)(6295,2729)(6312,2712)
	(6329,2695)(6347,2677)(6364,2660)
	(6382,2642)(6400,2624)(6418,2605)
	(6436,2586)(6454,2567)(6473,2547)
	(6491,2527)(6509,2507)(6527,2487)
	(6544,2467)(6561,2447)(6577,2427)
	(6593,2407)(6607,2388)(6621,2369)
	(6634,2350)(6646,2332)(6658,2314)
	(6668,2297)(6678,2279)(6687,2262)
	(6695,2246)(6702,2230)(6708,2214)
	(6715,2197)(6720,2179)(6726,2161)
	(6730,2142)(6735,2122)(6739,2101)
	(6743,2078)(6746,2055)(6749,2031)
	(6752,2006)(6754,1979)(6756,1952)
	(6757,1924)(6759,1894)(6760,1864)
	(6760,1833)(6761,1801)(6761,1768)
	(6762,1734)(6762,1699)(6762,1663)
	(6762,1626)(6762,1587)(6762,1559)
	(6762,1531)(6762,1501)(6762,1470)
	(6762,1438)(6762,1404)(6762,1368)
	(6762,1331)(6762,1291)(6762,1249)
	(6762,1205)(6762,1158)(6762,1108)
	(6762,1056)(6762,1001)(6762,944)
	(6762,884)(6762,822)(6762,758)
	(6762,692)(6762,626)(6762,559)
	(6762,493)(6762,428)(6762,365)
	(6762,306)(6762,251)(6762,201)
	(6762,156)(6762,118)(6762,86)
	(6762,60)(6762,41)(6762,27)
	(6762,19)(6762,14)(6762,12)
\put(2262,2262){\shade\ellipse{128}{128}}
\put(2262,2262){\ellipse{128}{128}}
\path(2262,2262)(2262,12)
\path(3612,1812)(4062,1812)
\path(3942.000,1782.000)(4062.000,1812.000)(3942.000,1842.000)
\path(2262,2262)(2264,2264)(2268,2268)
	(2276,2276)(2288,2288)(2305,2305)
	(2325,2325)(2348,2348)(2374,2374)
	(2402,2402)(2430,2430)(2458,2458)
	(2485,2485)(2512,2512)(2537,2537)
	(2560,2560)(2582,2582)(2603,2603)
	(2623,2623)(2642,2642)(2660,2660)
	(2677,2677)(2695,2695)(2712,2712)
	(2729,2729)(2747,2747)(2764,2764)
	(2782,2782)(2800,2800)(2818,2819)
	(2836,2838)(2854,2857)(2873,2877)
	(2891,2897)(2909,2917)(2927,2937)
	(2944,2957)(2961,2977)(2977,2997)
	(2993,3017)(3007,3036)(3021,3055)
	(3034,3074)(3046,3092)(3058,3110)
	(3068,3127)(3078,3145)(3087,3162)
	(3097,3183)(3106,3204)(3114,3225)
	(3121,3247)(3128,3269)(3134,3292)
	(3139,3315)(3144,3339)(3148,3363)
	(3152,3387)(3154,3411)(3157,3435)
	(3158,3459)(3160,3482)(3161,3505)
	(3161,3527)(3162,3549)(3162,3570)
	(3162,3591)(3162,3612)(3162,3633)
	(3162,3654)(3162,3675)(3162,3697)
	(3162,3721)(3162,3746)(3162,3773)
	(3162,3802)(3162,3833)(3162,3866)
	(3162,3900)(3162,3933)(3162,3965)
	(3162,3995)(3162,4019)(3162,4038)
	(3162,4051)(3162,4059)(3162,4062)
\path(2262,2262)(2260,2264)(2256,2268)
	(2248,2276)(2236,2288)(2219,2305)
	(2199,2325)(2176,2348)(2150,2374)
	(2122,2402)(2094,2430)(2066,2458)
	(2039,2485)(2012,2512)(1987,2537)
	(1964,2560)(1942,2582)(1921,2603)
	(1901,2623)(1882,2642)(1864,2660)
	(1847,2677)(1829,2695)(1812,2712)
	(1795,2729)(1777,2747)(1760,2764)
	(1742,2782)(1724,2800)(1705,2818)
	(1686,2836)(1667,2854)(1647,2873)
	(1627,2891)(1607,2909)(1587,2927)
	(1567,2944)(1547,2961)(1527,2977)
	(1507,2993)(1488,3007)(1469,3021)
	(1450,3034)(1432,3046)(1414,3058)
	(1397,3068)(1379,3078)(1362,3087)
	(1345,3095)(1327,3103)(1310,3110)
	(1292,3116)(1273,3121)(1254,3126)
	(1234,3129)(1214,3132)(1193,3133)
	(1171,3134)(1149,3133)(1127,3131)
	(1104,3128)(1080,3123)(1057,3117)
	(1033,3110)(1009,3102)(985,3093)
	(961,3082)(937,3071)(912,3058)
	(887,3044)(862,3029)(837,3012)
	(817,2999)(797,2985)(777,2970)
	(756,2954)(734,2937)(712,2918)
	(690,2899)(667,2879)(643,2858)
	(619,2836)(595,2812)(571,2788)
	(546,2762)(522,2736)(497,2709)
	(472,2681)(448,2652)(424,2623)
	(401,2593)(378,2563)(355,2532)
	(334,2502)(313,2470)(292,2439)
	(273,2408)(254,2377)(237,2345)
	(220,2314)(204,2282)(189,2251)
	(175,2219)(162,2187)(150,2156)
	(139,2126)(128,2094)(118,2062)
	(109,2029)(100,1996)(92,1961)
	(84,1926)(76,1891)(69,1855)
	(63,1818)(57,1780)(51,1742)
	(46,1704)(41,1665)(37,1626)
	(33,1587)(29,1548)(26,1509)
	(24,1470)(21,1432)(19,1394)
	(18,1356)(16,1319)(15,1283)
	(14,1248)(13,1213)(13,1178)
	(12,1145)(12,1112)(12,1080)
	(12,1048)(12,1018)(12,987)
	(12,953)(12,919)(12,885)
	(12,852)(12,818)(12,783)
	(12,748)(12,711)(12,674)
	(12,635)(12,594)(12,552)
	(12,508)(12,463)(12,417)
	(12,370)(12,323)(12,277)
	(12,232)(12,189)(12,150)
	(12,115)(12,85)(12,60)
	(12,41)(12,28)(12,19)
	(12,14)(12,12)
\end{picture}
}
\caption{Vertex/critical point cancellation}
\label{fig:vertcritcan} 
\end{figure}
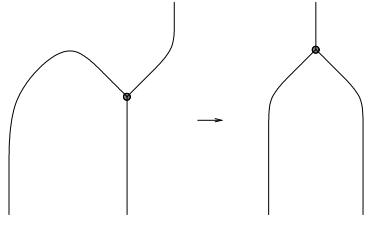

\begin{figure}[ht]
\includegraphics{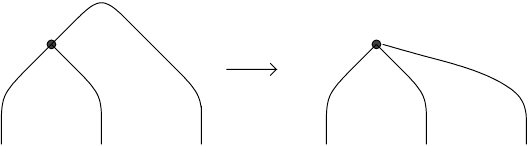}
\caption{Another vertex/critical point cancellation}
\label{fig:another}
\end{figure}

\begin{proof}
The proof is similar to that of Theorem \ref{thm:cerftangle}.  Let
$(f_j,\ul{b}_j)$ be cylindrical Morse data for $\Gamma_j$, $j = 0,1$.
As in the proof of Theorem \ref{thm:cerftangle}, composing with the
diffeomorphisms $(t,x) \mapsto (f_j(t,x),x)$ reduces to the case $f_0 =
f_1 = t$, and by the relative isotopy extension theorem there is a
family $\Psi_s$, $s \in [0,1]$ of diffeomorphisms of $Y = [-1,1] \times
X$ equal to the identity near the boundary with $\Psi_0$ the identity
and $\Psi_1(\Gamma_0) = \Gamma_1$.  It therefore suffices to consider
the family of functions $f_s := t \circ \Psi_s$ on the fixed graph
$\Gamma = \Gamma_0$.  We say that such a family $(f_s)$ is {\em good}
if except for
a finite number of values of $s$, each $f_s$ is a Morse function
injective on its critical set and the critical set is disjoint from
the vertices,
\[ f_s |_{ \Crit(f_s)} : \Crit(f_s) \hookrightarrow \R,
\quad \Crit(f_s) \cap \Ver(\Gamma) = \emptyset \]
and at the remaining finite number of times $s_1,\ldots s_m \in [0,1]$
at most one of the following occur:
\begin{enumerate} 
\item \label{gev1} critical point cancellation occurs in the interior of an edge;
\[ \exists y \in \on{int}(e), \quad \d_y f_s|_e = \d^2_y f_s|_e = 0,
\quad \d^3_y f_s|_e \neq 0 ;\]
\item \label{gev2} a critical point occurs at an endpoint of an edge:
\[ \exists y \in \partial e, \quad \d_y f_s|_e = 0, \quad \d^2_y
f_s|_e \neq 0 ; \ \text{or} \]
\item \label{gev3} two critical points or endpoints have the same value:
\[ \exists y_1,y_2 \in \Crit(f_s|_\Gamma) \cup \Ver(\Gamma), \quad y_1
  \neq y_2, \quad f_s(y_1) = f_s(y_2) .\]

\end{enumerate} 
Indeed, for no critical point cancellation to occur at the endpoints
it suffices that for each endpoint $p$ and time $s$, either $\d f_s
|_e (p)$ or $\d^2 f_s |_e (p)$ is non-zero.  This is always the case
for generic families, since these conditions are codimension two.
After a small generic perturbation, which as in the proof of Theorem
\ref{thm:cerftangle} is realized by moving the graph in the
$t$-direction, the family $(f_s)$ is good.  Between the parameters
$s_1,\ldots, s_m$ the decompositions change only by equivalence, by
the argument with the braids traced out by the intermediate markings
in the proof of Theorem \ref{thm:cerftangle}, which gives the gluing
moves \eqref{ggraphglue}.  At each critical time $s_k$, the
decompositions for $s$ slightly less than and slightly greater than
$s_k$ differ by one of the remaining moves: event \eqref{gev1} gives a
critical point cancellation or creation \eqref{gcancel}; event
\eqref{gev2}, in which a critical point on an edge passes through the
vertex at its end, gives a vertex/critical point cancellation or
creation \eqref{gvertcrit}; and event \eqref{gev3} gives one of the
three reversals \eqref{greversal}.  The disjointness condition holds
by the argument at the end of the proof of Theorem
\ref{thm:cerftangle}: an edge through both points would contain an arc
between them on which $f_s|_\Gamma$ is strictly monotone for $s$ near
$s_k$ and constant at $s = s_k$, contradicting the finiteness of the
critical set of $f_{s_k}|_\Gamma$.  Here the edges through a vertex
are the edges incident to it, and the edges through a critical point
in the interior of an edge are the two branches of that edge at the
point.
Breaking up the interval $[0,1]$ into subintervals each containing at
most one critical time proves the Theorem.
\end{proof}

\subsection{Lagrangian correspondences for elementary graphs}

The next step in the construction is to define a notion of admissible
labellings of graphs.  As in the case of tangles, this notion is
designed to avoid reducibles in the moduli spaces of flat bundles.

\begin{definition} \label{def:labels}
\begin{enumerate} 
\item {\rm (Labels meeting a vertex)} Let $(Y,\Gamma,\phi)$ be an
  elementary graph with a single vertex $v$.  Denote the boundary of
  $Y$ by $\partial Y = X_- \cup X_+$.  Let $B(v) \subset Y$ be a small
  open ball containing $v$, and
\[S(v)= \partial B(v) ,  \quad 
\ul{x}(v) := S(v) \cap \Gamma \]
denote the sphere around the vertex and the intersections with the
graph. The complement $Y \ssm B(v)$ of $B(v)$ can be viewed as a
three-dimensional bordism from $X_- \cup S(v)$ to $X_+$.  It contains
a tangle 
\[\Gamma \backslash (B(v) \cap \Gamma) \subset Y \ssm B(v) .\] 
Let $\ul{\mu}_\pm$ denote the labels for $\ul{x}_\pm = \Gamma \cap
X_\pm$, and $\ul{\mu}(v)$ denote the set of labels for $\ul{x}(v)$,
given by the labels of the edges incoming to a vertex and the images
of the labels under the involution $*$ for the outgoing edges.
\item {\rm (Admissible labellings)} 
 \label{veradm}
A set of labels $\ul{\mu}(v)$ at a vertex $v$ is {\em
  vertex-admissible} if the moduli space of flat bundles on the
punctured sphere is either empty or a point:
\[ \# M(S(v),\ul{x}(v),\ul{\mu}(v)) \leq 1 .\]
A {\em vertex-admissible labelling} of $\Gamma$ is a labelling of the
edges of $\Gamma$ by admissible labels, such that at each vertex the
collection of labels is vertex-admissible.  A  {\em vertex-admissible
  graph} is a graph equipped with a vertex-admissible labelling.
\item \label{standlab} {\rm (Standard labellings)} Suppose that $G =
  SU(r)$ with $r \ge 3$.  Let $\omega_j$ denote the $j$-th fundamental
  coweight of $SU(r)$, and write ${\bf j} = \omega_j/2$ as in the proof
  of Proposition \ref{prop:delta}.  A {\em
    standard labelling} of $\Gamma$ is a labelling of each edge by
  ${\bf 1}$ or ${\bf 2}$, so that each vertex is trivalent with labels
  ${\bf 1}, {\bf 1}, *{\bf 2}$, if all edges are oriented towards the
  vertex.
\end{enumerate} 
\end{definition} 

\noindent For a standard labelling, admissibility of the labels on
the incoming surface implies admissibility on the outgoing surface.
Indeed, in the formulation of Remark \ref{rem:parityrem}, the signed sum
$\sum_{i=1}^n \eps_i j_i$ mod $r$ over the $n$ markings is unchanged across a critical point of
$f|_\Gamma$, and also across a vertex, since $*{\bf 2} = \omega_{r-2}/2$
and $1 + 1 + (r-2) \equiv 0$ mod $r$.

\noindent The triple ${\bf 1}, {\bf 1}, {\bf 2}$ is analogous to
Khovanov-Rozansky's ${\bf 1},{\bf 2}$ (or thin, thick) labels
\cite{kr:ma}.

\begin{lemma} \label{lem:point} Suppose that $r \ge 3$ and let $\ul{x}
\subset S^2$ be a triple of distinct points.  The moduli space $M(S^2,
\ul{x}, {\bf 1 }, {\bf 1}, *{\bf 2})$ is a point.  More precisely, let
$\zeta = \exp(-\pi i/r)$ and for a subspace $V \subset \C^r$ let $P_V$
denote orthogonal projection onto $V$.  Then
\[ \cC_{\bf 1} = \{ \zeta (I - 2 P_\ell) \, | \, \ell \in \P(\C^r) \},
\quad \cC_{\bf 2} = \{ \zeta^2 (I - 2 P_W) \, | \, W \in
\on{Gr}(2,\C^r) \} ,\]
and the set
\begin{equation} \label{eq:vertorbit}
O_v := \{ (g_1,g_2) \in \cC_{\bf 1}^2 \, | \, g_1 g_2 \in \cC_{\bf 2}
\} = \{ (\zeta(I - 2P_{\ell_1}), \zeta(I - 2 P_{\ell_2})) \, | \,
\ell_1 \perp \ell_2 \} \end{equation}
is a single $G$-orbit, with stabilizers conjugate to 
\begin{equation} \label{eq:Hdef} H := S(U(1)
\times U(1) \times U(r-2)) .\end{equation} 
Hence any standard labelling of a
bordism-with-graph is vertex-admissible.
\end{lemma} 

\begin{proof}  By definition, $\cC_{\bf 1}$ is the conjugacy class of
\[ \exp(2\pi i \diag(\omega_1/2)) = \zeta \diag(-1,1,\ldots,1) \]
and
$\cC_{\bf 2}$ is the conjugacy class of 
\[ \exp(2\pi i
\diag(\omega_2/2)) = \zeta^2 \diag(-1,-1,1,\ldots,1) . \]
Let $\ell_1,\ell_2$
be lines in $\C^n$.  If $\ell_1 = \ell_2$ then 
\[ (I - 2P_{\ell_1})(I -
2P_{\ell_2}) = I . \]
Otherwise the product is the identity on
$(\ell_1 + \ell_2)^\perp$ and preserves the plane $\ell_1 + \ell_2$,
on which it has determinant one and trace $4c^2 - 2$, where $c$ is
the absolute value of the inner product of unit vectors in $\ell_1$
and $\ell_2$.  Hence the product equals $I - 2P_W$ for a plane $W$ if
and only if $c = 0$, that is, $\ell_1 \perp \ell_2$, and then $W =
\ell_1 \oplus \ell_2$.  This proves \eqref{eq:vertorbit}.  The group
$U(r)$ acts transitively on pairs of orthogonal lines, and since $r \ge
3$ so does $SU(r)$, by adjusting the determinant on $(\ell_1 \oplus
\ell_2)^\perp$.  The stabilizer of $(\ell_1,\ell_2)$ is $S(U(\ell_1)
\times U(\ell_2) \times U((\ell_1 \oplus \ell_2)^\perp)) \cong H$.  The
moduli space $M(S^2, \ul{x}, {\bf 1 }, {\bf 1}, *{\bf 2})$ is $O_v/G$,
which is therefore a point.
\end{proof}  

\begin{remark} \label{rem:pointrem}
The point of Lemma \ref{lem:point} is not cut out transversally: The
stabilizer $H$ has positive dimension, so the corresponding
representation is reducible, and the expected dimension of the moduli
space is negative.  The arguments below use only that $O_v$ is a single
$G$-orbit $O_v \cong G/H$, a smooth closed submanifold of
$\cC_{\bf 1}^2$ of dimension $4r - 6$.
\end{remark}  

\begin{lemma}  \label{lem:Gspin} {\rm (Correspondences for elementary graphs)} 
Let $(Y,\Gamma,\phi)$ be an elementary cylindrical graph from
$(X,\ul{x}_-)$ to $(X,\ul{x}_+)$ with a standard labelling, containing
a single trivalent vertex $v$ and no critical points of $f|_\Gamma$.
Let $k_\mp$ denote the number of edges at $v$ meeting $X_\mp$, so
that $k_- + k_+ = 3$, and suppose that $k_- > k_+$.  Then 
$L(Y,\Gamma,\phi)$ is a smooth submanifold of
  $M(X,\ul{x}_-)^- \times M(X,\ul{x}_+)$, the projection
  $L(Y,\Gamma,\phi) \to M(X,\ul{x}_-)$ is an embedding, and the
  projection $L(Y,\Gamma,\phi) \to M(X,\ul{x}_+)$ is a fiber bundle
  with fiber
\[ F_v \cong \begin{cases} S^2 & k_+ = 1, \ \text{the edge meeting
  $X_+$ has label ${\bf 2}$} \\
\P^{r-2} & k_+ = 1, \ \text{the edge meeting $X_+$ has label ${\bf
    1}$} \\
SU(r)/H & k_+ = 0 . \end{cases} \]
Furthermore, $L(Y,\Gamma,\phi)$ is simply connected.
\end{lemma}

\begin{proof}   We examine the Lagrangian correspondence in terms of framed bundles. Let $E_\mp(v)$ denote the set of markings in $\ul{x}_\mp$
on edges at $v$.  Choose standard presentations of the fundamental
groups in which the generators around $E_-(v)$ are consecutive, and
those for $\ul{x}_+$ are obtained by replacing these generators by
those around $E_+(v)$, as in the proof of Lemma \ref{lem:vertlag}.  After
inverting the holonomies of edges oriented away from $v$, the
holonomies around the three edges at $v$ form an element of the orbit
$O_v' \cong O_v \cong G/H$ of Lemma \ref{lem:point}, the set of triples in
$\cC_{\bf 1} \times \cC_{\bf 1} \times \cC_{*{\bf 2}}$ with product equal to the identity.  Since $k_- \ge 2$, the holonomies around $E_+(v)$ are determined
by those around $E_-(v)$, as their product.  Let 
\[ \cC_v \subset
\prod_{e \in E_-(v)} \cC_e \]
denote the image of $O_v'$ under the
projection to the holonomies around $E_-(v)$, so that $\cC_v \cong G/H$,
and let
\[ \pi_v: \cC_v \to \prod_{e \in E_+(v)} \cC_e \]
be the map giving the holonomies around $E_+(v)$.   The target is a conjugacy class $\cC_c$ if $k_+ = 1$ and a point if $k_+ = 0$.  Then
$\pi_v$ is a $G$-equivariant map of homogeneous spaces, hence a fiber
bundle with fiber $G_c/H$, where $G_c$ is the stabilizer of a point
of $\cC_c$ (and $G_c = G$ if $k_+ = 0$).  In terms of the pairs of
orthogonal lines $(\ell_1,\ell_2)$ in \eqref{eq:vertorbit}: 
\begin{itemize} 
\item If the edge
meeting $X_+$ has label ${\bf 2}$, the fiber is the set of lines
$\ell_1$ in the fixed plane $W = \ell_1 \oplus \ell_2$, that is,
$\P(W) \cong S^2$. 
\item If the edge has label ${\bf 1}$, say corresponding to
$\ell_2$, the fiber is $\P(\ell_2^\perp) \cong \P^{r-2}$; and if $k_+
= 0$ the fiber is all of $O_v \cong SU(r)/H$. 
\end{itemize} 
Let 
\[M^{\on{fr}}(X,\ul{x}_-,v) = \{ b \in M^{\on{fr}}(X,\ul{x}_-) \, | \, b|_{E_-(v)}
\in \cC_v \}, \quad \pi: M^{\on{fr}}(X,\ul{x}_-,v) \to M^{\on{fr}}(X,\ul{x}_+)\]
where $\pi$ replaces the holonomies around $E_-(v)$ by their image
under $\pi_v$.  Then $\pi$ is a fiber bundle with fiber $G_c/H$, and
since the generators around $E_-(v)$ are consecutive,
$\Phi_{\ul{x}_-}|_{M^{\on{fr}}(X,\ul{x}_-,v)} = \Phi_{\ul{x}_+} \circ \pi$.  As in
the proof of Lemma \ref{lem:elemhol}, it follows that:
\begin{itemize} 
\item $M^{\on{fr}}(X,\ul{x}_-,v) \cap
Z_{\ul{x}_-} = \pi^{-1}(Z_{\ul{x}_+})$ is a $G$-invariant submanifold
on which $G/Z(G)$ acts freely, 
\item $L(Y,\Gamma,\phi)$ is the graph of
the induced bundle projection $(M^{\on{fr}}(X,\ul{x}_-,v) \cap Z_{\ul{x}_-})/G \to
M(X,\ul{x}_+)$, and 
\item the projection of $L(Y,\Gamma,\phi)$ to
$M(X,\ul{x}_-)$ is an embedding with image $(M^{\on{fr}}(X,\ul{x}_-,v) \cap
Z_{\ul{x}_-})/G$. 
\end{itemize} 
This proves the Lemma except for the claim that the 
correspondence is simply-connected. 

The base $M(X,\ul{x}_+)$ is simply connected by \cite{ms:pb,ni:co}, as
in the proof of Lemma \ref{lem:Kbrane}.  The fiber is connected and simply
connected: $S^2$ and $\P^{r-2}$ are, and $SU(r)/H$ is since $SU(r)$
is simply connected and $H$ is connected.  The long exact sequence of
homotopy groups proves simply-connectedness.
\end{proof}

\begin{remark} \label{rem:modbg} {\rm (Background classes for standard labellings)} 
Let $G = SU(r)$, $r \ge 3$, and let $(X,\ul{x})$ be a marked surface
with markings $\ul{x} = (x_1,\ldots,x_n)$ whose labels are ${\bf 1}$ or ${\bf 2}$ and which is admissible, so
that $\sum_{i=1}^n \eps_i j_i \equiv d$ mod $r$ in the notation of Remark
\ref{rem:parityrem}.  Write
\[ w_i := w_2(T\cC_{\mu_i} \qu G) \in H^2(M(X,\ul{x});\Z_2) \]
for the class associated to the $i$-th marking as in \eqref{eq:bpm}, so
that 
\[ b_\pm(X,\ul{x}) = \sum_{1 \le i \le n, \ \eps_i = \pm} w_i . \]
By Lemma
\ref{lem:point}, the holonomy around a marking $x_i$ labelled ${\bf 1}$,
with respect to the orientation of its strand, has the form $\zeta(I
- 2P_{\ell_i})$ for a line $\ell_i \subset \C^r$; similarly a marking
labelled ${\bf 2}$ determines a plane $W_i$.  This defines
$G$-equivariant bundles 
\[ \tau_i \to M^{\on{fr}}(X,\ul{x}), \quad W_i \to M^{\on{fr}}(X,\ul{x}) \] 
with fibers $\ell_i$ and $W_i$ respectively on
the framed space $M^{\on{fr}}(X,\ul{x})$, on which the center $Z(G)
\cong \Z_r$ acts by the scalars.  A line bundle on which $Z(G)$ acts
trivially descends to $M(X,\ul{x})$, since $G/Z(G)$ acts freely on the
level set.  Since 
\[ T\cC_{\bf 1} \cong \on{Hom}(\ell,\ell^\perp), \quad T\cC_{\bf 2} \cong \on{Hom}(W,W^\perp) \]
the classes $w_i$ are the
mod $2$ reductions of $c_1(\tau_i^{-r})$, resp. $c_1(\det W_i^{-r})$.

\begin{enumerate}
\item {\rm ($r$ even)} The number of markings labelled ${\bf 1}$ is
  odd, since $d$ is odd; in particular there is at least one.  For
  any two markings $q,q'$ labelled ${\bf 1}$ the line bundle 
  \[ \tau_q
  \otimes \tau_{q'}^\dual \to M^{\on{fr}}(X,\ul{x}) \]
  descends to $M(X,\ul{x})$, so 
  \[ w_q - w_{q'} = r
  c_1(\tau_q^\dual \otimes \tau_{q'}) = 0 \ \text{mod} 2. \] 
  For a marking $s$ labelled ${\bf 2}$, the line bundle $\det
  W_s^{r/2}$ descends and so $w_s = 0$.  Since the number of markings
  labelled ${\bf 1}$ is odd, it follows that
  \begin{equation} \label{eq:w2wq} w_2(TM(X,\ul{x})) = \sum_{i=1}^n
  w_i = w_q \end{equation}
  for any marking $q$ labelled ${\bf 1}$, and the classes
  $b_\pm(X,\ul{x})$ are equal to either $0$ or $w_2(TM(X,\ul{x}))$.  For
  $s$ labelled ${\bf 2}$ let
\begin{equation} \label{eq:csclass} c_s := c_1(\det W_s \otimes \tau_q^{-2}) \ \text{mod} \ 2. \end{equation}
This is well-defined since $Z(G)$ acts trivially on $\det W_s \otimes \tau_q^{-2}$.   Furthermore, $c_s$ is independent of the choice of $q$ labelled ${\bf
  1}$ since the difference for two choices is $2 c_1(\tau_q^\dual
\otimes \tau_{q'})$.  Finally let 
\[ N(\ul{x}) := \Bigl(\sum_{i=1}^n \eps_i j_i -
d\Bigr)/r \in \Z . \] 
\item {\rm ($r$ odd)} Let 
\begin{equation} \label{eq:cs2} 
c_s := w_s \end{equation} 
for $s$ labelled ${\bf 2}$, and $N(\ul{x}) := 0$.
\end{enumerate}
Define background classes
\[ \ti{b}_\pm(X,\ul{x}) := b_\pm(X,\ul{x}) + {\textstyle \sum_{1 \le s \le n, \ \mu_s =
    {\bf 2}}} c_s + N(\ul{x}) w_2(TM(X,\ul{x})) \in H^2(M(X,\ul{x});\Z_2) .\]
Since $TG^{2g} \qu G$ is obtained from the associated bundle of $\g \oplus \g$, 
which has vanishing second Stiefel-Whitney class, 
\begin{eqnarray*} 
w_2(TM(X,\ul{x})) &=& \sum_{i=1}^n
w_i \\ &=& b_+(X,\ul{x}) + b_-(X,\ul{x}) \\ &=& \ti{b}_+(X,\ul{x}) + \ti{b}_-(X,\ul{x}) . \end{eqnarray*}
It follows that $\ti{b}_+$ and $\ti{b}_-$ are equivalent modulo $w_2(TM(X,\ul{x}))$.
\begin{itemize}
\item For $r$ odd, $\ti{b}_\pm$ is the background class \eqref{eq:bpm} computed
after reversing the sign of every marking labelled ${\bf 2}$. 
\item For $r$
even, reversing signs of markings labelled ${\bf 2}$ does not change
$b_\pm$, since $w_s = 0$.
\end{itemize} 
For $r = 4$, where ${\bf 2} = *{\bf 2}$, the
orientation of an edge labelled ${\bf 2}$ is not part of the data, so
one must check that $\ti{b}_\pm$ does not depend on it.  Reversing the orientation of an edge replaces $W_s$ by $W_s^\perp$, and changes the class $c_s$ of \eqref{eq:csclass} by
\[ c_1(\det W_s^{-2}) = c_1(\tau_q^{-4}) = w_q = w_2(TM(X,\ul{x})) \
\text{mod} \ 2 ,\]
using $r = 4$ and \eqref{eq:w2wq}.
Indeed, the difference $2c_1(\det W_s^{-1} \otimes \tau_q^2)$ is even.   Reversing
the orientation changes
$N(\ul{x})$ by $\pm 1$. Hence $\ti{b}_\pm$ is independent of this choice.

\end{remark}

The modified background classes are compatible with the relative spin
structures of Lemma \ref{lem:Kbrane} on correspondences of elementary
tangles.

\begin{lemma} \label{lem:modtangle} {\rm (Modified background classes
    for elementary tangles)} Let $(Y,K,\phi)$ be an elementary
  cylindrical tangle from $(X,\ul{x}_-)$ to $(X,\ul{x}_+)$ whose
  strands are labelled ${\bf 1}$ or ${\bf 2}$.  Then $L(Y,K,\phi)$
  admits a unique relative spin structure with background class
  $(\ti{b}_\pm(X_-,\ul{x}_-),\ti{b}_\mp(X_+,\ul{x}_+))$ of Remark
  \ref{rem:modbg}, for either choice of sign.
\end{lemma}

\begin{proof}
By Lemma \ref{lem:Kbrane} and its proof, it suffices to show that the
restriction of $\ti{b}_\pm(X_-,\ul{x}_-) + \ti{b}_\mp(X_+,\ul{x}_+)$ to
the correspondence equals that of $b_\pm(X_-,\ul{x}_-) +
b_\mp(X_+,\ul{x}_+)$.  The two differ by the following terms, each of
which vanishes:
\begin{itemize}
\item A marking $s$ on a strand meeting both ends contributes its
  class $c_s$ once from each end.  These agree on the correspondence,
  since the plane $W_s$ is the same and $c_s$ does not depend on the
  choice of $q$: The bundle $\tau_q \otimes \tau_{q'}^\dual$ is defined
  on the correspondence for markings $q,q'$ on either end.
\item The two markings of a strand with both endpoints on one end
  determine the same plane, since their holonomies are inverse to each
  other on the correspondence, and so contribute $2c_s = 0$.
\item For $r$ even, the terms
\[ N(\ul{x}_-) w_2(TM(X,\ul{x}_-)), \quad N(\ul{x}_+)
w_2(TM(X,\ul{x}_+)) \]
agree on the correspondence. 
\begin{enumerate} 
\item First, $N(\ul{x}_-) = N(\ul{x}_+)$,
since the signed sums $\sum_{i=1}^{n_\pm} \eps_i j_i$ over the $n_\pm$
markings of $\ul{x}_\pm$ are equal: Two markings with the same label
and opposite signs are added or removed.  
\item Second, by \eqref{eq:w2wq}
and the independence of $q$, the classes $w_2(TM(X,\ul{x}_\pm)) = w_q$
restrict to the same class from either end.  
\end{enumerate} 
For $r$ odd these terms
vanish, since $N(\ul{x}_\pm) = 0$.
\end{itemize}
\end{proof}

\begin{proposition} \label{prop:Gspinprop} {\rm (Relative spin structures
  for elementary graphs)} Let $(Y,\Gamma,\phi)$ be an elementary
cylindrical graph from $(X,\ul{x}_-)$ to $(X,\ul{x}_+)$ with a standard
labelling, containing a single trivalent vertex $v$ and no critical
points of $f|_\Gamma$.  Then $L(Y,\Gamma,\phi)$ admits a relative spin
structure with background class
$(\ti{b}_\pm(X_-,\ul{x}_-),\ti{b}_\mp(X_+,\ul{x}_+))$ of Remark
\ref{rem:modbg}, for either choice of sign.
\end{proposition}

\begin{proof}
We may assume that $k_- > k_+$ as in Lemma \ref{lem:Gspin}; the other
case follows by exchanging the roles of $\ul{x}_-$ and $\ul{x}_+$.
Write $L = L(Y,\Gamma,\phi)$.  As in the proof of Lemma
\ref{lem:Kbrane}, by \cite[Remark 3.1.10, Proposition 3.3.1(b)]{orient}
it suffices to show that $w_2(TL)$ equals the restriction of the
background class to $L$.  We compute both sides, using the following
notation.
\begin{itemize} 
\item For an edge $e$ at $v$ let $w_e \in H^2(L;\Z_2)$ denote the
  class $w_i$ of Remark \ref{rem:modbg} for the marking $x_i$ of $e$,
  pulled back to $L$ from the end containing $x_i$, and similarly for
  $c_e$ if $e$ is labelled ${\bf 2}$.
\item For $r$ even, let $w_2(TM)|_L \in H^2(L;\Z_2)$ denote the
  pullback of $w_2(TM(X,\ul{x}_\pm))$ to $L$ from either end, which is
  independent of the end by the proof of Lemma \ref{lem:modtangle}.
\item Let $w_v \in H^2(L;\Z_2)$ denote the second Stiefel-Whitney
  class of the bundle on $L$ induced by the tangent bundle of the
  orbit $O_v \cong G/H$ of Lemma \ref{lem:point}.
\item Let
\begin{equation} \label{eq:Einout}
E_{\rm in}(v) := \{ e \in \Edge(\Gamma) \, | \, e \to v \}, \quad
E_{\rm out}(v) := \{ e \in \Edge(\Gamma) \, | \, v \to e \}
\end{equation}
be the sets of edges at $v$ oriented towards resp. away from $v$.
For $r$ even let $\delta_v := \# E_{\rm out}(v) - 1$, and for $r$ odd
let $\delta_v := 0$ and $\delta_v w_2(TM)|_L := 0$.
\end{itemize}

\vskip .1in \noindent {\em Step 1: The tangent bundle.}  
Let 
\[ \ul{x}_{\rm far} := \ul{x}_- \ssm E_-(v) = \ul{x}_+ \ssm E_+(v) \] 
be the set of markings not on edges at $v$; each lies on a strand
meeting both ends, which identifies the markings of $\ul{x}_-$ and
$\ul{x}_+$ not on edges at $v$.  By the construction in the proof of Lemma
\ref{lem:Gspin}, $L$ is the quotient of the level set in
$M^{\on{fr}}(X,\ul{x}_{\rm far}) \times \cC_v$, and $\cC_v \cong O_v$.  As in Remark
\ref{rem:modbg}, with 
\[ \ad := (\g| Z_{\ul{x}})/G  \subset (G^{2g} \times \cC_{\ul{x}_{\rm far}} \times O_v) \qu G \] 
the quotient construction gives 
\[ TL \oplus \ad \oplus \ad
\cong T(G^{2g} \times \cC_{\ul{x}_{\rm far}} \times O_v) \qu G \] 
Hence 
\[ w_2(TL) = {\textstyle \sum_{x_i \in \ul{x}_{\rm far}}} w_i + w_v .\]

\vskip .1in \noindent {\em Step 2: The background classes.}  An edge
oriented from $X_-$ towards $X_+$ meets the two ends in
markings of the same sign, which we take to be positive; the other
convention exchanges $b_+$ and $b_-$ and does not affect the argument.
Each marking in $\ul{x}_{\rm far}$ lies on an edge meeting both ends with the same
sign, and so contributes $w_i$ to exactly one of $b_+(X_-,\ul{x}_-)$
and $b_-(X_+,\ul{x}_+)$.  An edge $e$ at $v$ contributes $w_e$ to
$b_+(X_-,\ul{x}_-)$ if it lies below $v$ and is oriented upwards, and
to $b_-(X_+,\ul{x}_+)$ if it lies above $v$ and is oriented downwards;
in both cases it is oriented towards $v$.  Hence the restriction of
$(b_+(X_-,\ul{x}_-),b_-(X_+,\ul{x}_+))$ to $L$ is 
\[ \sum_{x_i \in \ul{x}_{\rm far}} w_i
+ \sum_{e \in E_{\rm in}(v)} w_e . \] 
For the other choice of sign the edges at
$v$ that contribute are those oriented away from $v$.  In the
remaining terms of $\ti{b}_\pm$, the classes $c_s$ for markings $s \in
\ul{x}_{\rm far}$ labelled ${\bf 2}$ occur once from each end and cancel, as in
the proof of Lemma \ref{lem:modtangle}, while each edge $e$ at $v$ labelled ${\bf 2}$
contributes $c_e$ once.  For $r$ odd the remaining terms vanish, and
for $r$ even the terms involving the second Stiefel-Whitney classes
of the ends contribute
\[ (N(\ul{x}_-) + N(\ul{x}_+)) w_2(TM)|_L = (N(\ul{x}_+) -
N(\ul{x}_-)) w_2(TM)|_L . \]
  An
edge $e$ at $v$ with label ${\bf j}_e$ meets exactly one of the ends,
in a single marking, and so contributes a single term to the
difference 
\[ \Delta_v := \sum_{i=1}^{n_+} \eps_i j_i - \sum_{i=1}^{n_-} \eps_i j_i \] 
over the markings of $\ul{x}_+$ and
$\ul{x}_-$.  If $e$ lies below $v$, its marking lies in $\ul{x}_-$ with
sign $+$ or $-$ according as $e$ is oriented towards or away from $v$,
contributing $-j_e$ resp. $+j_e$ to the difference $\Delta_v$.  If $e$ lies above
$v$, its marking lies in $\ul{x}_+$ with sign $-$ or $+$ according as
$e$ is oriented towards or away from $v$, again contributing $-j_e$
resp. $+j_e$ to the difference $\Delta_v$.  The markings in $\ul{x}_{\rm far}$ contribute equally
to both signed sums, so the total change in the signed sum is
\[ \Delta_v= \sum_{e \in E_{\rm out}(v)} j_e - \sum_{e \in E_{\rm in}(v)} j_e ,
\]
where ${\bf j}_e$ is the label of $e$.  For integers $1 \le j \le r-1$
write $*j := r - j$, so that $*{\bf j}$ is the label with index $*j$.
For an outgoing edge $e$ we have $*{\bf j}_e = {\bf r - j_e}$, so that
$j_e = r - {*j_e}$.  By Definition \ref{def:labels} \eqref{standlab},
the indices $j_e$ of the incoming edges together with the indices
$*j_e$ of the outgoing edges are $1, 1, r-2$, which sum to $r$.  Hence
\[ \Delta_v = r \cdot \# E_{\rm out}(v) - \Bigl( \sum_{e \in E_{\rm
    in}(v)} j_e + \sum_{e \in E_{\rm out}(v)} {*j_e} \Bigr) = r (\#
E_{\rm out}(v) - 1) .\]
It follows that, for $r$ even, 
\[ N(\ul{x}_+) - N(\ul{x}_-) = \Delta_v /
r = \# E_{\rm out}(v) - 1 = \delta_v  \] 
Combining Steps 1 and 2, $w_2(TL)$ minus the restriction of
$(\ti{b}_+(X_-,\ul{x}_-),\ti{b}_-(X_+,\ul{x}_+))$ to $L$ is
\begin{equation} \label{eq:vertobs}
D_{\rm in} := w_v + {\textstyle \sum_{e \in E_{\rm in}(v)}} w_e +
{\textstyle \sum_{e \ \text{at} \ v, \ \mu_e = {\bf 2}}} c_e + \delta_v
w_2(TM)|_L .
\end{equation}
For the background class
$(\ti{b}_-(X_-,\ul{x}_-),\ti{b}_+(X_+,\ul{x}_+))$ the difference is the
class $D_{\rm out}$ defined by the same formula with $E_{\rm in}(v)$
replaced by $E_{\rm out}(v)$.  It remains to show that $D_{\rm in}$ and
$D_{\rm out}$ vanish.

\vskip .1in \noindent {\em Step 3: Computation.}  We
compute the classes in \eqref{eq:vertobs} in terms of characters of $\ol{H} := H/Z(G)$, where $H$ is the subgroup of \eqref{eq:Hdef}.  By
the proof of Lemma \ref{lem:Gspin}, the Lagrangian $L$ is diffeomorphic to the quotient by the adjoint group $ \widetilde{Z}/(G/Z(G))$, where
\[ \widetilde{Z} = M^{\on{fr}}(X,\ul{x}_-,v) \cap Z_{\ul{x}_-} \] 
and the group $G/Z(G)$
acts freely.  Furthermore, the map 
\[ \widetilde{Z} \to \cC_v \cong O_v \cong
(G/Z(G))/\ol{H} \] 
is equivariant.  Consider the classifying map 
\[ L \to
E(G/Z(G))\times_{G/Z(G)} O_v \simeq B\ol{H} .\] 
Any
$G$-equivariant vector bundle on $O_v$ on which $Z(G)$ acts trivially
induces on $L$ the pullback of the bundle on $B\ol{H}$ associated to
the corresponding representation of $\ol{H}$.  This Kempf descent principle applies to $TO_v$,
to the pullbacks of $T\cC_e$ for edges $e$ at $v$, and to the line
bundles $\det W_e \otimes \tau_q^{-2}$ and $\tau_q^{-r}$ where $q$ is
an edge at $v$ labelled ${\bf 1}$; by the independence of $q$ in Remark
\ref{rem:modbg} (which holds on $L$ for markings $q$ on either end), these
compute $c_e$ and, for $r$ even, $w_2(TM)|_L$ by \eqref{eq:w2wq}.  Note here that all classes in $D_{\rm in}$ and $D_{\rm out}$ are pulled back from
$H^2(B\ol{H};\Z_2)$.  The derived subgroup $[\ol{H}, \ol{H}]$ of $\ol{H}$ is the image of the $SU(r-2)$ acting
on $(\ell_1 \oplus \ell_2)^\perp$, which meets $Z(G)$ trivially and is
simply connected. Hence $\pi_1(\ol{H})$ is free abelian and
$H^2(B\ol{H};\Z_2) \cong \on{Char}(\ol{H}) \otimes \Z_2$, where
\[ \on{Char}(\ol{H}) = \on{Hom}(\ol{H},U(1)) \] 
denotes the group of
characters of $\ol{H}$ and the characteristic class
of a representation is the reduction of the character of its
determinant.  The characters of $H$ are 
\[ H \to \C^\times, \quad h \mapsto z_1^\alpha
z_2^\beta , \] 
where $z_1,z_2$ are the eigenvalues of $h$ on
$\ell_1,\ell_2$, and $Z(G)$ acts with weight $\alpha + \beta$.
Hence
\[ \on{Char}(\ol{H}) = \{ (\alpha,\beta) \in \Z^2 \, | \, \alpha + \beta
\equiv 0 \ \text{mod} \ r \} \] 
and 
\[ H^2(B\ol{H};\Z_2) \cong \on{Char}(\ol{H}) \otimes \Z_2 \to
\Z_2^2, \ (\alpha,\beta) \mapsto (\alpha, (\alpha+\beta)/r) \
\text{mod} \ 2 . \] 
This last map is an isomorphism, as one sees using the basis
$(1,-1),(0,r)$ of $\on{Char}(\ol{H})$.  The
tautological bundles $\tau_1,\tau_2$ with fibers $\ell_1,\ell_2$ have
characters $(1,0),(0,1)$, and $\tau_3$ with fiber $(\ell_1 \oplus
\ell_2)^\perp$ has $\det \tau_3$ with character $(-1,-1)$.  Using
\[ \det \on{Hom}(A,B) = \det(A)^{-\rk B} \otimes \det(B)^{\rk A} \] 
we compute the characteristic classes of the various bundles 
in Table \ref{table:charclass};  the entries in the last column are taken mod $2$.

\begin{table}[ht]
\begin{center}
\begin{tabular}{lll}
bundle & character of $\det$ & class in $\Z_2^2$ \\ \hline
$T\cC_{\bf 1}$ at $\ell_1$: $\on{Hom}(\ell_1,\ell_1^\perp)$ & $(-r,0)$
& $(r,1)$ \\
$T\cC_{\bf 1}$ at $\ell_2$ & $(0,-r)$ & $(0,1)$ \\
$T\cC_{\bf 2}$ at $W = \ell_1 \oplus \ell_2$:
$\on{Hom}(W,W^\perp)$ & $(-r,-r)$ & $(r,0)$ \\
$TO_v$: $\on{Hom}(\ell_1,\ell_2) \oplus \on{Hom}(\ell_1,\tau_3)
\oplus \on{Hom}(\ell_2,\tau_3)$ & $(-r-1,-r+1)$ & $(r+1,0)$ \\
$\det W \otimes \tau_1^{-2}$ & $(-1,1)$ & $(1,0)$ \\
$\det W^\perp \otimes \tau_1^{-2}$, $r = 4$ & $(-3,-1)$ & $(1,1)$ \\
$\tau_1^{-r}$ & $(-r,0)$ & $(r,1)$
\end{tabular}
\end{center}
\caption{Characteristic classes of various bundles} 
\label{table:charclass}
\end{table}

After inverting the holonomies of edges oriented away from $v$, the
holonomies around the three edges at $v$ define a point of
\[ O_v' \subset \cC_{\bf 1} \times \cC_{\bf 1} \times \cC_{*{\bf 2}}
,\]
so that each edge corresponds to one of the three factors.  By Lemma
\ref{lem:point}, the holonomy in the first, second and third factor
determines the line $\ell_1$, the line $\ell_2$ and the plane $W =
\ell_1 \oplus \ell_2$ respectively; we refer to the factors by these
subspaces.  The class $w_e$ is the second Stiefel-Whitney class of the
pullback of $T\cC$ from the factor corresponding to $e$, using
$w_2(T\cC_{*\mu} \qu G) = w_2(T\cC_\mu \qu G)$ under the
identification $\cC_{*\mu} = \cC_\mu^{-1}$.  By
Definition \ref{def:labels} \eqref{standlab}, each edge at $v$ is labelled
${\bf 1}$ and oriented towards $v$, or labelled ${\bf j}$ and oriented
away from $v$ with $*{\bf j} \in \{ {\bf 1}, *{\bf 2} \}$, or labelled
${\bf 2} = *{\bf 2}$ and oriented towards $v$; the last case occurs
only for $r = 4$, the case $*{\bf j} = {\bf 1}$ with ${\bf j} = {\bf
  2}$ only for $r = 3$.

\vskip .1in \noindent {\em Case $r$ odd.}  In this case $\delta_v w_2(TM)|_L$ vanishes, and $c_e = w_e$. Furthermore, the classes of
the factors $\ell_1,\ell_2$ and $W$ are $(1,1),(0,1)$ and $(1,0)$, while $w_v = 0$.
In particular, the sum of the three classes $w_e$ vanishes.  Every edge
labelled ${\bf 1}$ is oriented towards $v$ and every edge labelled
${\bf 2}$ away from $v$.  So the edges labelled ${\bf 2}$ are exactly
those in $E_{\rm out}(v)$, and the expression 
\eqref{eq:vertobs} is given by 
\[ D_{\rm in} = w_v + \sum_{k=1}^3 w_{e_k} = 0 , \] 
where $e_1,e_2,e_3$ are the three edges at $v$.  Similarly 
\[ D_{\rm out} = w_v + 2
\sum_{1 \le k \le 3, \ \mu_{e_k} = {\bf 2}} w_{e_k} = 0 . \] 
This shows that $D_{\rm in}, D_{\rm out}$ vanish, as claimed, 
which completes the proof in the odd case.

\vskip .1in \noindent {\em Case $r$ even.}  We may assume $r \ge 4$, since there is no ${\bf 2}$ label otherwise.  The two edges labelled ${\bf 1}$
are oriented towards $v$ and correspond to the factors
$\ell_1,\ell_2$, and the edge $e_3$ labelled ${\bf 2}$ corresponds to
the factor $W$.  Table \ref{table:charclass}
gives
\[ w_{e_1} = w_{e_2} = w_2(TM)|_L = (0,1), \ w_{e_3} = 0, \ w_v = (1,0)
.\]
If
$e_3$ is oriented away from $v$, then its plane is $W_{e_3} = W$ (the
holonomy around $e_3$ along its orientation is the product
$\zeta^2(I - 2P_{\ell_1 \oplus \ell_2})$ of the other two), so $c_{e_3}
= (1,0)$, and $\delta_v = 0$; then
\[ D_{\rm in} = (1,0) + 2 (0,1) + (1,0) = 0 , \quad D_{\rm out} = (1,0) + w_{e_3} + (1,0) = 0 . \]  
If $r = 4$ and $e_3$ is oriented towards $v$, then the
holonomy along its orientation is 
\[ \zeta^{-2}(I - 2P_W) = \zeta^2(I -
2P_{W^\perp}) \] 
so 
\[ W_{e_3} = W^\perp, \quad c_{e_3} = (1,1), \quad \delta_v = -1. \] 
Then
\[ D_{\rm in} = (1,0) + 2(0,1) + (1,1) + (0,1) = 0 , \quad D_{\rm out} = (1,0) + (1,1) +
(0,1) = 0 . \]
This completes the proof.
\end{proof}  

\begin{definition} {\rm (Correspondence for vertex-admissible labellings)}  
Suppose $(Y,\Gamma)$ is a graph with labelling $\ul{\nu}$ that is
vertex-admissible for each vertex.  Let $M(Y,\Gamma)$ denote the
moduli space of flat bundles on the complement of $\Gamma$ in $Y$ with
holonomies around the edges of $\Gamma$ given by
$\ul{\nu}$. Restriction to the boundary and pullback under the
boundary identification define a map
\begin{equation} \label{eq:restr2} M(Y,\Gamma) \to
M(X_-,\ul{x}_-)^- \times M(X_+,\ul{x}_+).\end{equation}
Denote the image of \eqref{eq:restr2} by $L(Y,\Gamma,\phi)$.
\end{definition}  

\begin{lemma} \label{lem:vertlag} Let $(Y,\Gamma,\phi)$ be an elementary graph containing
a single vertex and $\nu$ a standard labelling of the edges of
$\Gamma$.  Then $L(Y,\Gamma,\phi)$ is a smooth Lagrangian
correspondence from $M(X_-,\ul{x}_-)$ to $M(X_+,\ul{x}_+).$
\end{lemma}

\begin{proof}  Smoothness is Lemma \ref{lem:Gspin}.  We write
$\ul{\mu}_\pm \ssm \ul{\mu}(v)$ resp.\ $\ul{\mu}_\pm \cap \ul{\mu}(v)$
  for the labels of those markings in $\ul{x}_\pm$ that are not
  resp.\ are connected to $v$ by an edge.  By \eqref{eq:fus2}, the
  symplectic forms on the two ends are those obtained by reduction
  from
\begin{equation} \label{eq:fus4}  
\omega_{g,\ul{\mu}_\pm - \ul{\mu}(v)} + 
 \omega_{0,\ul{\mu}_\pm \cap \ul{\mu}(v)}  + (1/2)
\langle \Phi_{g,\ul{\mu}_\pm - \ul{\mu}(v)}^* \theta \wedge 
\Phi_{0, \ul{\mu}_\pm \cap \ul{\mu}(v)}^* \ol{\theta} \rangle.
\end{equation}
Let $d$ be the value of $f$ at the vertex and $\eps$ a small number.
Choose a presentation for the fundamental group of $f^{-1}(d - \eps)$;
then a presentation for the fundamental group of $f^{-1}(d + \eps)$ is
obtained by replacing the generators for the strands incoming to the
vertex with those outgoing.  With respect to this set of generators,
the correspondence defined by the bordism is given by
\begin{equation} \label{eq:set}
 \prod_{\mu \in \ul{\mu}_- \cap \ul{\mu}(v)} c_\mu = \prod_{\mu \in
\ul{\mu}_+ \cap \ul{\mu}(v)} c_\mu \end{equation}  
and descending to the quotient.  The equation \eqref{eq:set} defines an
isotropic submanifold of $ \cC_{\ul{\mu}_- \cap \ul{\mu}(v)}^- \times
\cC_{\ul{\mu}_+ \cap \ul{\mu}(v)} $ since the moduli space for the
sphere around the vertex is a point by Lemma \ref{lem:point}.  It follows
from \eqref{eq:fus4} that the \eqref{eq:set} defines an isotropic
submanifold of the product $M(X_-,\ul{x}_-)^- \times
M(X_+,\ul{x}_+)$.

It remains to check that this isotropic submanifold has half the
dimension of the product.  Let 
\[ P \subset \cC_{\ul{\mu}_- \cap
  \ul{\mu}(v)} \times \cC_{\ul{\mu}_+ \cap \ul{\mu}(v)} \] 
  denote the set
of solutions of \eqref{eq:set}.  Up to inverting the holonomies of the
edges oriented away from $v$, $P$ is the set of triples in $\cC_{\bf
  1} \times \cC_{\bf 1} \times \cC_{\bf 2}$ with product the identity,
which is identified with $O_v$ of \eqref{eq:vertorbit} by Lemma
\ref{lem:point}.  Hence $P \cong G/H$ with $H$ as in \eqref{eq:Hdef},
and
\begin{equation} \label{eq:dimP}
\begin{aligned}
\dim P &= (r^2 - 1) - ((r-2)^2 + 1) = 4r - 6 \\
&= \frac{1}{2} \bigl(
2(r-1) + 2(r-1) + 4(r-2) \bigr) \\
&= \frac{1}{2} \sum_{e \ni v} \dim
\cC_{\mu_e} ,\end{aligned} \end{equation}
using 
\[ \dim \cC_{\bf 1} = \dim \P(\C^r) = 2(r-1), \quad \dim \cC_{\bf 2}
= \dim \on{Gr}(2,\C^r) = 4(r-2).\]
Write 
\[ \cC_{\rm rest} :=
\cC_{\ul{\mu}_\pm \ssm \ul{\mu}(v)} \] 
for the classes of the markings
not connected to $v$, and $\pi(p) \in G$ for the common value of the
two sides of \eqref{eq:set} at $p \in P$.  The correspondence is the
quotient by $G$ of
\[ L^{\rm fr} := \{ (\ul{a},\ul{c},p) \in G^{2g} \times \cC_{\rm rest}
\times P \ | \ \Phi_{g,\ul{\mu}_\pm \ssm \ul{\mu}(v)}(\ul{a},\ul{c})
\, \pi(p) = 1 \} . \]
As $p$ ranges over the orbit $P$, $\pi(p)$ ranges over a single
conjugacy class $\cC'$, which is either trivial or the class of a
label, and $\pi: P \to \cC'$ is a submersion.  So the identity is a
regular value of the map defining $L^{\rm fr}$, and $G/Z(G)$ acts
freely on $L^{\rm fr}$, because the same holds for the moduli space of
$X$ with the markings not connected to $v$ together with one marking
with class $\cC'$ (omitted if $\cC'$ is trivial).  That moduli space
is smooth by Proposition \ref{prop:smooth}, since its signed sum of
labels is that of $\ul{x}_\pm$.  Therefore
\[ \dim L(Y,\Gamma,\phi) = 2g \dim G + \dim \cC_{\rm rest} + \dim P - 2
\dim G , \]
while
\[ \dim M(X_\pm,\ul{x}_\pm) = 2g \dim G + \dim \cC_{\rm rest} +
{\textstyle \sum_{e \in E_\pm(v)}} \dim \cC_{\mu_e} - 2 \dim G ,\]
where $E_\pm(v)$ is the set of markings in $\ul{x}_\pm$ connected to
$v$.  By \eqref{eq:dimP}, 
\[ \dim L(Y,\Gamma,\phi) = \frac{1}{2}
(\dim M(X_-,\ul{x}_-) + \dim M(X_+,\ul{x}_+)) , \] 
so $L(Y,\Gamma,\phi)$ is Lagrangian.
\end{proof} 

\subsection{Morphisms for graphs}

The following associates a generalized Lagrangian 
correspondence to any graph with standard labelling:

\begin{definition} {\rm (Generalized Lagrangian correspondence for a decorated
graph)} Let $(f,\ul{b})$ be a cylindrical Cerf decomposition of
  $\Gamma$ equipped with vertex-admissible labels
  $\ul{\nu}$.  Let
\[L(Y_j,\Gamma_j,\phi_j) \subset M(X_{j-1},\ul{x}_{j-1})^- \times
  M(X_j,\ul{x}_j)\] 
denote the Lagrangian submanifold of representations that extend over
$(Y_j,\Gamma_j)$.  Define
\[ \ul{L}(Y,\Gamma,\phi) := (L(Y_1,\Gamma_1,\phi_1), \ldots ,
  L(Y_m,\Gamma_m,\phi_m)) .\] 
\end{definition}

We call the unique critical point of $f|_\Gamma$ or vertex of $\Gamma$
in an elementary graph its {\em event}.  For an elementary graph
from $(X,\ul{x}_-)$ to $(X,\ul{x}_+)$ let $E_\mp \subset \ul{x}_\mp$
be the markings on the edges through the event.  By Lemmas
\ref{lem:elemhol} and \ref{lem:Gspin}, $L(Y,\Gamma,\phi)$ embeds into the
moduli space of the end for which $|E_\mp|$ is larger, and is a fiber
bundle over the moduli space of the other end.

\begin{lemma} \label{lem:graphcomp} {\rm (Composing elementary graphs)} 
Let $(Y_1,\Gamma_1,\phi_1)$ from $(X,\ul{x}_0)$ to $(X,\ul{x}_1)$ and
$(Y_2,\Gamma_2,\phi_2)$ from $(X,\ul{x}_1)$ to $(X,\ul{x}_2)$ be
elementary cylindrical graphs with standard labellings, with sets of
markings $E^i_\mp$ as above.  Suppose that one of the following holds:
\begin{enumerate}
\item \label{gc1} $|E^1_-| > |E^1_+|$ or $|E^2_+| > |E^2_-|$;
\item \label{gc2} $E^1_+ \cap E^2_- = \emptyset$;
\item \label{gc3} one of the two events is a critical point, the other
  is a vertex, and $E^1_+ \cap E^2_-$ consists of a single marking.
\end{enumerate}
Then $L(Y_1,\Gamma_1,\phi_1) \circ L(Y_2,\Gamma_2,\phi_2)$ is
transverse and embedded, and equal to $L((Y_1,\Gamma_1,\phi_1) \circ
(Y_2,\Gamma_2,\phi_2))$.
\end{lemma}

\begin{proof}  Write $L_i = L(Y_i,\Gamma_i,\phi_i)$ and $M_i =
M(X,\ul{x}_i)$.  In all cases the composition equals
$L((Y_1,\Gamma_1,\phi_1) \circ (Y_2,\Gamma_2,\phi_2))$ as a set, by the
Seifert-van Kampen argument of Lemma \ref{lem:compcor}.

In case \eqref{gc1}, suppose first that $|E^1_-| > |E^1_+|$.  Then the
projection $L_1 \to M_1$ is a fiber bundle, hence a submersion, so the
fiber product $L_1 \times_{M_1} L_2$ is transversally cut out.
Moreover $L_1 \to M_0$ is an embedding.  As in the proof of Theorem
\ref{thm:extendt}, the map from the fiber product to $M_0 \times M_2$ is
then an injective immersion, hence an embedding by compactness.  The
case $|E^2_+| > |E^2_-|$ is the same with the roles of $L_1$ and $L_2$
exchanged.

Suppose that \eqref{gc1} fails, so that $L_1 \to M_1$ and $L_2 \to M_1$
are embeddings.  By Lemmas \ref{lem:elemhol} and \ref{lem:Gspin} their images
are $(A \cap Z)/G$ and $(B \cap Z)/G$, where $Z = Z_{\ul{x}_1}$ and $A,
B \subset M^{\on{fr}}(X,\ul{x}_1)$ are the $G$-invariant submanifolds
cut out by the conditions on the holonomies around $E^1_+$ resp.
$E^2_-$ described there; we choose a presentation in which the
markings in $E^1_+$ resp. $E^2_-$ are consecutive.  As in the proof of
Lemma \ref{lem:transcrit}, it suffices to show that $A$ and $B$ intersect
transversally, and that on $A \cap B$ the map $\Phi_{\ul{x}_1}$ factors
as $\Phi_{\ul{x}'} \circ \pi'$ with $\pi'$ a submersion and $\ul{x}'$
admissible.

In case \eqref{gc2} the conditions defining $A$ and $B$ involve
disjoint sets of holonomies, so $A$ and $B$ intersect transversally.
Let $\pi_1: A \to M^{\on{fr}}(X,(\ul{x}_1 \ssm E^1_+) \cup E^1_-)$ and
$\pi_2$ similarly be the maps of Lemmas \ref{lem:elemhol} and \ref{lem:Gspin}
replacing the holonomies around $E^1_+$ by those around $E^1_-$, resp.
around $E^2_-$ by those around $E^2_+$.  These are fiber bundles,
and on $A \cap B$ we have $\Phi_{\ul{x}_1} = \Phi_{\ul{x}'} \circ
\pi'$, where $\pi' = \pi_1 \circ \pi_2$ and
\[\ul{x}' = (\ul{x}_1 \ssm (E^1_+ \cup E^2_-)) \cup E^1_- \cup E^2_+
.\]
The marking $\ul{x}'$ is obtained from $\ul{x}_0$ by replacing $E^2_-$
with $E^2_+$, so it is admissible by the remark after Definition
\ref{def:labels}.  For embeddedness, suppose that two points $P, P'$ of
$A \cap B \cap Z$ have the same image in $M_0 \times M_2$, related by
$g \in G$ on the $\ul{x}_0$-holonomies and by $h \in G$ on the
$\ul{x}_2$-holonomies.  By equivariance, both $g$ and $h$ map
$\pi'(P) \in Z_{\ul{x}'}$ to $\pi'(P')$, so $g^{-1}h$ lies in the
stabilizer of $\pi'(P)$, which is $Z(G)$ since $\ul{x}'$ is
admissible.  Hence $hP = P'$.  The same argument for tangent vectors
shows that the map is an immersion, and so an embedding by
compactness.

In case \eqref{gc3}, since \eqref{gc1} fails, either the first event
is a minimum, joining markings $x_p, x_{p'} \in E^1_+$ with $x_{p'} \in
E^2_-$ and $x_p \notin E^2_-$, and the second is a vertex with
$|E^2_-| \in \{2,3\}$; or the first event is a vertex with $|E^1_+|
\in \{2,3\}$ and the second is a maximum.  The second possibility
is obtained from the first by reversing the direction of $Y$.  In the
first, $A = \{ b_p = b_{p'} \}$ is a graph over the factor $\cC_p$,
which is not involved in the definition of $B$, so $A$ and $B$
intersect transversally.  Let $E^2_- = \{ p' \} \cup Q$.  We may
order the markings so that $p, p'$ and the markings of $Q$ are
consecutive.  On $A \cap B$ the factor $b_p^{\eps_p}b_{p'}^{\eps_{p'}}$
in \eqref{eq:phi} is trivial.  If $|Q| = 1$, then $\Phi_{\ul{x}_1} =
\Phi_{\ul{x}_0} \circ \pi'$ on $A \cap B$, where $\pi'$ forgets
$b_p,b_{p'}$; $\pi'$ is a submersion on $A \cap B$ since the projection
of the orbit $A_{v}$ of Lemma \ref{lem:Gspin} to the holonomy around $Q$
is a fiber bundle over a conjugacy class.  If $|Q| = 2$, then the
holonomies around $\{ p' \} \cup Q$ have product one, so
$\Phi_{\ul{x}_1} = \Phi_{\ul{x}_2} \circ \pi'$ on $A \cap B$, where
$\pi'$ forgets $b_{p'}$ and the holonomies around $Q$; $\pi'$ is a
submersion since the projection of $\cC_v$ to the holonomy around $p'$
is a fiber bundle over $\cC_{p'}$.  In both cases the composite is
embedded:  Every holonomy around $\ul{x}_1$ is determined by those
around $\ul{x}_2$ (if $|Q| = 1$: $b_{p'} = b_p$ and the holonomy around
$Q$ is determined by $b_{p'}$ and the holonomy of the edge above the
vertex) resp. $\ul{x}_0$ (if $|Q| = 2$: $b_{p'}$ is determined by the
holonomies around $Q$, and $b_p = b_{p'}$). So the projection of the fiber product to $M_2$ resp. $M_0$ is an injective immersion.
\end{proof} 

\begin{proposition}
  \label{prop:indep3}
{\rm (Independence of the generalized Lagrangians from all choices up
  to equivalence)} Let $(Y,\Gamma,\phi)$ be a graph with a standard
labelling from $(X,\ul{x}_-)$ to $(X,\ul{x}_+)$, with admissible labels
on $X_-$.  Then the generalized Lagrangian correspondence
$\ul{L}(Y,\Gamma,\phi)$ is independent, up to the composition
equivalence of $\Symp^\#_{1/2r}$, of the choice of Cerf decomposition.
\end{proposition}

\begin{proof}  By Theorem \ref{thm:cerfgraph} it suffices to check
that the generalized Lagrangian correspondences are invariant under
the Cerf moves.  Gluing an elementary graph without events to an
adjacent elementary graph composes with the graph of a
symplectomorphism, and is handled as in the proof of Theorem
\ref{thm:extendt}.  Critical point cancellations and critical
point/critical point switches only involve the tangle part, and are
handled by Lemma \ref{lem:transcrit} and the proof of Theorem
\ref{thm:extendt}.

For a critical point/vertex switch or a vertex/vertex switch, the two
events lie on disjoint edges, so $E^1_+ \cap E^2_- = \emptyset$ for
both orders of the events.  By Lemma \ref{lem:graphcomp} \eqref{gc2},
both compositions are transverse and embedded with composite the
correspondence of the composed graph, which is the same for both
orders.

For a vertex/critical point cancellation as in Figures
\ref{fig:vertcritcan} and \ref{fig:another}, the two events share exactly one
edge, and the composite is an elementary graph with a single vertex
$(Y_{01},\Gamma_{01},\phi_{01})$.  By Lemma \ref{lem:graphcomp}
\eqref{gc1} or \eqref{gc3}, the composition $L(Y_0,\Gamma_0,\phi_0)
\circ L(Y_1,\Gamma_1,\phi_1)$ is transverse and embedded, and equal to
$L(Y_{01},\Gamma_{01},\phi_{01})$.  For example, in the case of Figure
\ref{fig:another} the first event is a vertex with two edges below, so
\eqref{gc1} applies: in terms of the holonomies of the incoming
surface,
\begin{eqnarray*} L(Y_0,\Gamma_0,\phi_0) \circ L(Y_1,\Gamma_1,\phi_1) &=&
\{ [a_1,\ldots,a_{2g},b_1,\ldots,b_n] \ | \ b_j b_{j+1} = b_{j+2}  \} \\
&=& L(Y_{01},\Gamma_{01},\phi_{01}) .\end{eqnarray*}
This completes the proof.
\end{proof}

\begin{definition} {\rm (Decorated Graphs)} For coprime integers $r,d >0$ and a compact oriented surface $X$ let $\Graph(X,r,d)$ denote the category of graphs whose
\begin{itemize}
\item objects are collections $\ul{x}$ of distinct oriented points of
  $X$ with admissible labels $\ul{\mu}$; that is, the same objects as
  for $\Tan(X,r,d)$;
\item morphisms are isotopy classes of labelled $X$-cylindrical
  graphs $\Gamma$; and
\item composition of morphisms is given by stacking as in Definition
  \ref{def:cobtangles} \eqref{stack}.
\end{itemize} 
As before, the identity is the equivalence class of the trivial graph.
\end{definition}  

The following extends Theorem \ref{thm:extendt} to graphs.

\begin{theorem} 
 \label{thm:extendg}
{\rm (Symplectic-valued field theory for graphs)} Let $r \ge 3$
and $d > 0$ be coprime integers, and consider graphs with standard
labellings, with the background classes $\ti{b}_\pm$ of Remark
\ref{rem:modbg}.  Then the
partially defined functor $\Phi: \Graph(X,r,d) \to \Symp^\#_{1/2r}$ for
elementary graphs extends to a field theory for graphs in $X$.
\end{theorem} 

\begin{proof} By Proposition \ref{prop:indep3}, it suffices to show 
that the correspondences arising, including the composites in the
Cerf moves, lie in the class of admissible correspondences.  The
elementary correspondences are simply connected by Lemma \ref{lem:Kbrane}
for critical points and Lemma \ref{lem:Gspin} for
vertices.  The composites arising in the moves are, by the proof of
Lemma \ref{lem:graphcomp}, fiber bundles over moduli spaces of admissible
marked surfaces whose fibers are products of the fibers in Lemmas
\ref{lem:elemhol} and \ref{lem:Gspin}, and so are also simply connected.  As in
the proof of Theorem \ref{thm:extendt}, simply connected Lagrangians in
the monotone moduli spaces are monotone with constant $\tau = 1/2r$,
since $\pi_2(M,L)$ is then the image of $\pi_2(M)$.  Relative spin
structures are provided by Lemma \ref{lem:modtangle}
for correspondences involving critical points, and by Proposition
\ref{prop:Gspinprop} for correspondences involving vertices.  By simple connectivity they are unique for the
given background classes, and so compatible under composition.
\end{proof} 

As for tangles, the field theory for graphs depends only on
equivalence classes of graphs.

\begin{corollary} \label{cor:descentgraph} {\rm (Invariance under
    equivalence for graphs)} Let $\Phi: \Graph(X,r,d) \to
  \Symp^\#_{1/2r}$ be the field theory of Theorem \ref{thm:extendg}.
  If $\Gamma_0$ and $\Gamma_1$ are labelled $X$-cylindrical graphs from
  $\ul{x}_-$ to $\ul{x}_+$ which are equivalent as graph bordisms, by an
  equivalence preserving the labels, then $\Phi([\Gamma_0]) =
  \Phi([\Gamma_1])$.
\end{corollary}

\begin{proof}
The proof of Proposition \ref{prop:descent} applies with tangles
replaced by graphs.  Step 1 of that proof concerns only diffeomorphisms
of $[-1,1] \times X$ equal to the identity on the boundary, and shows
that $\Gamma_1$ is isotopic to $g(\Gamma_0)$ for a level-preserving
diffeomorphism $g$.  In Step 2, the correspondences of elementary
graphs are natural under diffeomorphisms of marked surfaces, since
they are defined by restriction of flat bundles on the complements of
the graphs, and the symplectomorphisms $\gamma_i$ induced by the
diffeomorphisms $g_{b_i}$ preserve the background classes $\ti{b}_\pm$
of Remark \ref{rem:modbg}, since the classes $c_s$, the integers
$N(\ul{x})$ and $w_2(TM(X,\ul{x}))$ are defined naturally in terms of
the labelled, oriented markings.  The telescoping argument of Step 2
then applies without change.
\end{proof}

As for tangles, the disk invariants of the correspondences associated
to graphs are differences of a function on objects.

\begin{proposition} \label{prop:deltagraph} {\rm (Disk invariants of
    graphs)} Let $\Phi: \Graph(X,r,d) \to \Symp^\#_{1/2r}$ be the
  field theory of Theorem \ref{thm:extendg}.  There is a function
  $\delta$ from the objects of $\Graph(X,r,d)$ to $\Z$, unique up to
  adding a constant on each set of objects connected by morphisms,
  such that for every morphism $[\Gamma]$ from $\ul{x}_-$ to
  $\ul{x}_+$
\[ w(\Phi([\Gamma])) = \delta(\ul{x}_+) - \delta(\ul{x}_-) . \]
In particular, the restriction of $\delta$ to tangles satisfies the
conclusion of Proposition \ref{prop:delta}.
\end{proposition}

\begin{proof}
The proof is that of Proposition \ref{prop:delta}, with one additional
type of elementary piece.  By Theorem \ref{thm:extendg} and Corollary
\ref{cor:descentgraph}, $[\Gamma] \mapsto w(\Phi([\Gamma]))$ is well
defined on morphisms, and by Corollary \ref{cor:additive} it is
additive under composition and may be computed from any Cerf
decomposition as the sum of the disk invariants of the pieces.

\vskip .1in \noindent 
{\em Step 1: Normal form for vertices.} Every vertex may be moved by
an isotopy so that all three of its edges lie below it near the
vertex.  Indeed, an edge leaving a vertex upward may be bent so that
it leaves downward and turns up at a new local minimum; this is the
inverse of the vertex/critical point cancellation of Theorem
\ref{thm:cerfgraph} \eqref{gvertcrit}.  Hence every morphism has a Cerf
decomposition into braids, cups, caps, and elementary graphs
$\Gamma_v$ containing a single vertex $v$ with all edges below it.  By
the definition of a standard labelling, at such a vertex the two
edges labelled ${\bf 1}$ are oriented towards $v$ and the edge
labelled ${\bf 2}$ is oriented away from $v$.  So $\Gamma_v$ removes
from the incoming marking a {\em vertex triple}, consisting of two
markings labelled ${\bf 1}$ and one marking labelled ${\bf 2}$ with
these orientations, and leaves the remaining markings unchanged.

\vskip .1in \noindent 
{\em Step 2: Values of the pieces.} As in the proof of Proposition
\ref{prop:delta}, braids have vanishing disk invariant, a cup adding a
pair of markings with label ${\bf j}$ to markings with multiset
$\ul{m}$ of labels and orientations has disk invariant $c(\ul{m},j)$
depending only on this data, and the corresponding cap has disk
invariant $-c(\ul{m},j)$.  In the same way, the disk invariant of
$\Gamma_v$ depends only on the multiset $\ul{m}$ of labels and
orientations of the markings not in the vertex triple, since two such
elementary graphs with the same data are related by composing with
braids: The configuration space of labelled oriented points in $X$,
together with the position of the vertex, is connected.  Write
$t(\ul{m})$ for this disk invariant.

\vskip .1in \noindent 
{\em Step 3: Paths and commutation.} A Cerf decomposition as in Step 1
determines a path in the set of multisets of labels and orientations,
whose steps add a pair (a cup), remove a pair (a cap), or remove a
vertex triple (a vertex).  The steps preserve admissibility, as noted
after Definition \ref{def:labels}.  The disk invariant of the graph is
the sum of the values $c$, $-c$, $t$ of the steps.  Two steps acting
on disjoint sets of markings can be realized by elementary pieces
supported in disjoint regions of $X$, and the two orders are then
related by a reversal \eqref{greversal}.  Since the disk invariant of
the graph does not depend on the decomposition, the sum of the values
of two such steps does not depend on their order.  Here the two
pieces act on disjoint sets of markings, so the composition of their
correspondences is transverse and embedded by Lemma \ref{lem:graphcomp}
\eqref{gc2}, and Corollary \ref{cor:additive} applies.

\vskip .1in \noindent 
{\em Step 4: Path independence.} Let $\ul{m}, \ul{m}'$ be multisets and
consider any path from $\ul{m}$ to $\ul{m}'$.  Using Step 3, the path
may be reordered so that all additions of pairs come first: moving an
addition before a removal is always possible, since the removed
markings are still present after the addition.  The additions then
commute with each other, and so do the removals, since all removed
markings are present at the start of the removals.  So the value of
the path depends only on the numbers of pairs of each label added,
and the numbers of pairs of each label and of vertex triples removed.
A pair that is added and later removed contributes $c(\ul{n},j) -
c(\ul{n},j) = 0$, after reordering so that the two steps are
adjacent.  Finally, the vectors counting markings by label and
orientation of a pair labelled ${\bf 1}$, a pair labelled ${\bf 2}$,
and a vertex triple are linearly independent, since a vertex triple
contains two markings labelled ${\bf 1}$ with the same orientation.
So the net numbers of steps of each kind are determined by $\ul{m}$
and $\ul{m}'$, and the value $W(\ul{m},\ul{m}')$ of a path from
$\ul{m}$ to $\ul{m}'$ is independent of the path.  In particular a
path from $\ul{m}$ to itself has value zero, and values of paths are
additive under concatenation.

\vskip .1in \noindent 
{\em Step 5: Construction of $\delta$.}  Let $S$ be a set of objects
connected by morphisms, and fix $\ul{x}_0 \in S$.  For $\ul{x},
\ul{x}' \in S$ the difference of the multisets is an integer
combination of pairs and vertex triples, with some net number $\gamma
\in \Z$ of vertex triples removed.  If $\gamma \ge 0$ there is a path
from $\ul{x}$ to $\ul{x}'$: add sufficiently many pairs, then remove
pairs and $\gamma$ vertex triples.  If $\gamma \le 0$ there is a path
from $\ul{x}'$ to $\ul{x}$.  Define
\[ \delta(\ul{x}) := \begin{cases} W(\ul{x}_0,\ul{x}) & \text{if
  there is a path from} \ \ul{x}_0 \ \text{to} \ \ul{x}, \\
- W(\ul{x},\ul{x}_0) & \text{otherwise}. \end{cases} \]
By Step 4, $\delta(\ul{x}') - \delta(\ul{x}) = W(\ul{x},\ul{x}')$
whenever there is a path from $\ul{x}$ to $\ul{x}'$.  For a morphism
$[\Gamma]$ from $\ul{x}_-$ to $\ul{x}_+$, a decomposition as in Step 1
is such a path, so $w(\Phi([\Gamma])) = \delta(\ul{x}_+) -
\delta(\ul{x}_-)$.  Uniqueness up to a constant on each $S$ is clear.
Tangles are graphs without vertices, so the last claim follows.
\end{proof}

Combining Theorem \ref{thm:extendg}, Proposition \ref{prop:deltagraph}
and Corollary \ref{cor:categorify} we obtain an \ainfty-category-valued
field theory for graphs.

\begin{theorem} \label{thm:extendgc} {\rm (Category-valued field theory
    for graphs)} Let $r \ge 3$ and $d > 0$ be coprime integers and $w
  \in \Z$.  The assignment of
\begin{itemize}
\item to any object $(\ul{x},\ul{\mu})$ the Fukaya \ainfty category
  $\GFuk(M(X,\ul{x}),w + \delta(\ul{x}))$, with $\delta$ as in
  Proposition \ref{prop:deltagraph}, and
\item to any morphism $[\Gamma]$ with a standard labelling the
  homotopy class of the \ainfty functor $\Phi(\ul{L}(\Gamma))$
\end{itemize}
defines a field theory for graphs from $\Graph(X,r,d)$ to the category
of (small \ainfty categories, homotopy classes of \ainfty functors).
\end{theorem}

\begin{proof}
By Theorem \ref{thm:extendg} and Corollary \ref{cor:descentgraph},
$\ul{L}(\Gamma)$ is a well defined morphism of $\Symp^\#_{1/2r}$, so
by Corollary \ref{cor:categorify} the functor $\Phi(\ul{L}(\Gamma))$
is well defined up to homotopy and maps $\GFuk(M(X,\ul{x}_-),w +
\delta(\ul{x}_-))$ to $\GFuk(M(X,\ul{x}_+),w + \delta(\ul{x}_-) +
w(\ul{L}(\Gamma)))$.  By Proposition \ref{prop:deltagraph} the target
is $\GFuk(M(X,\ul{x}_+),w + \delta(\ul{x}_+))$.  Functoriality follows
from Corollary \ref{cor:categorify}.
\end{proof}

\bibliographystyle{plain}
\bibliography{fieldb}

\end{document}